\documentclass[reqno,a4paper]{amsart}
\usepackage[top=1.3in, bottom=1.3in, left=1.1in, right=1.1in]{geometry}   %set the margins;
\usepackage{color}
\usepackage{amsfonts}
\usepackage{amsmath, amsfonts, amsthm, amssymb, amscd}
\usepackage[mathcal]{euscript}
\usepackage{mathrsfs}
\usepackage{dsfont}
\usepackage{ulem}
\usepackage{bbm}
\usepackage{enumitem}

\usepackage{slashed}
\usepackage[colorlinks=true]{hyperref}

\usepackage{lineno} 

\newtheorem{theorem}{Theorem}[section]
\newtheorem{proposition}[theorem]{Proposition}
\newtheorem{lemma}[theorem]{Lemma}

\newtheorem{definition}[theorem]{Definition}
\newtheorem{remark}[theorem]{Remark}

\numberwithin{equation}{section}

\usepackage{multirow}
\usepackage{tabularx}
\usepackage{lscape}
\newcounter{mnotecount}[section]

\newcommand \Dcal {\mathcal {D}}

\newcommand \delb {\bar {\del}}

\newcommand \del \partial

\newcommand \RR{\mathbb{R}}
\newcommand{\NN}{\mathbb{N}}
\newcommand \eps{\varepsilon}

\newcommand{\ep}{\epsilon}
\newcommand {\vep}{\varepsilon}

\newcommand {\la}{\langle}
\newcommand {\ra}{\rangle}

\newcommand{\MM}{\mathcal{M}}
\newcommand{\MMonetwo}{\mathcal{M}_{\tau_1,\tau_2}}
\newcommand{\MMtwoinfty}{\mathcal{M}_{\tau_2,\infty}}

\newcommand{\II}{\mathcal{I}^+}

\newcommand{\IIonetwo}{\II_{\tau_1,\tau_2}}

\newcommand \nablas{\slashed{\nabla}}
\newcommand \gs {\slashed{g}}
\def\Deltas{\slashed{\Delta}}

\providecommand{\abs}[1]{|#1|}

\def\DDb{\mathbb{D}}

\providecommand{\norm}[2]{\vert #1 \vert_{#2}}
\providecommand{\Wnorm}[4]{\Vert #1 \Vert_{W^{#2}_{#3}(#4)}}
\providecommand{\SEnergy}[3]{\mathbf{SE}^{(#1)}_{#2}[#3]}
\providecommand{\Energy}[4]{\mathbf{E}^{(#1)}_{#2,#3}[#4]}
\providecommand{\Flux}[4]{\mathbf{F}^{(#1)}_{\II_{#2},#3}[#4]}

\providecommand{\Error}[3]{\mathcal{E}^{(#1)}_{#2}[#3]}

\providecommand{\LEFM}[4]{\mathfrak{S}^{(#1)}_{#2;#3}[#4]}
\providecommand{\EFM}[4]{\mathfrak{T}^{(#1)}_{#2;#3}[#4]}

\def\peta{\iota_{p,\eta}}

\providecommand{\dist}[2]{\mathrm{d}^{(#1)}[#2]}

\providecommand{\set}[1]{\mathcal{S}^{(#1)}}
\providecommand{\setsub}[1]{\tilde{\mathcal{S}}^{(#1)}}

\providecommand{\Etotal}[2]{\mathbf{E}_{\tau_0}^{(#1)}[#2]}

\makeatletter
\def\hlinew#1{%
	\noalign{\ifnum0=`}\fi\hrule \@height #1 \futurelet
	\reserved@a\@xhline}
\makeatother

\begin{document}
\title[S.~Dong, S.~Ma, Y.~Ma, and X.~Yuan]{Generically sharp decay for quasilinear wave equations with null condition}
%\author{Shijie Dong, Siyuan Ma, Yue Ma, and Xu Yuan}

\author[S.~Dong]{Shijie Dong}
	\address{Southern University of Science and Technology, Shenzhen International Center for Mathematics, and Department of Mathematics, 518055 Shenzhen, P.R. China.}
	\email{dongsj@sustech.edu.cn, shijiedong1991@hotmail.com}

	\author[S.~Ma]{Siyuan Ma}
	\address{Albert Einstein Institute, Am M\"uhlenberg 1, 14476 Potsdam, Germany.}
	\email{siyuan.ma@aei.mpg.de}
	
		\author[Y.~Ma]{Yue Ma}
	\address{Xi'an Jiaotong University, School of Mathematics and Statistics, 28 West Xianning Road, Xi'an Shaanxi 710049, P.R. China.}
	\email{yuemath@xjtu.edu.cn}

	\author[X.~Yuan]{Xu Yuan}
	\address{Department of Mathematics, The Chinese University of Hong Kong, Shatin, N.T., Hong Kong, P.R. China.}
	\email{xu.yuan@cuhk.edu.hk}

\maketitle

\begin{abstract}
We are interested in the three-dimensional quasilinear wave equations with null condition. Global existence and pointwise decay for this model {have} been proved in the celebrated works of Klainerman \cite{Klainerman86} and Christodoulou \cite{Christodoulou86} for small smooth initial data.
In this work, we illustrate the precise pointwise asymptotic behavior of the solutions for initial data posed on a hyperboloid and show that the decay rate $v^{-1}u^{-1}$ is optimal for a generic set of initial data. 
\end{abstract}

	{\hypersetup{linkcolor=black}
\tableofcontents}

\allowdisplaybreaks

%%%%%%%%%%%%%%%%%%
\section{Introduction}
%%%%%%%%%%%%%%%%%%

 We consider the following quasilinear wave equation in $\RR^{1+3}$
\begin{equation}
\label{Qwave:M}
\Box \psi = P^{\alpha\beta\gamma}\del_{\gamma}\psi\del_{\alpha}\del_{\beta}\psi,
\end{equation} 
in which the wave operator is denoted by $\Box = -\partial_t\partial_t + \Delta$ with $\Delta$ the spacial Laplacian and {$P^{\alpha\beta\gamma}$} is a constant-coefficient null cubic form, i.e., 
\begin{equation}
\label{nullcond:M}
P^{\alpha\beta\gamma}\xi_{\alpha}\xi_{\beta}\xi_{\gamma} = 0,\quad \forall\,\, \xi_0^2 = \xi_1^2+\xi_2^2+\xi_3^2.
\end{equation}
The requirement \eqref{nullcond:M} for the coefficients $P^{\alpha\beta\gamma}$ is called the \textit{null condition} for the nonlinear terms in the wave equation \eqref{Qwave:M}, which was first suggested by Klainerman \cite{Klainerman83, Klainerman86}. Therefore, we call such a wave equation \eqref{Qwave:M}, with the restriction \eqref{nullcond:M} imposed, as a quasilinear wave equation with the null condition.

The three-dimensional quasilinear wave equation \eqref{Qwave:M} with the null condition is a basic model in the study of nonlinear wave-type equations. 
One fundamental question to the model \eqref{Qwave:M} is whether it admits a global solution for initial data that are sufficiently small and smooth, and what is the asymptotic behavior of the global solution if it exists.

Let $\Sigma_{\tau_0}$ be a spacelike, asymptotically null three-dimensional hypersurface. We impose the initial data of \eqref{Qwave:M} on $\Sigma_{\tau_0}$, which are
\begin{equation}\label{eq:ID}
\big(\psi, \partial_t \psi \big)\big|_{\Sigma_{\tau_0}}
= (\psi_0, \psi_1).
\end{equation}
We recall the null coordinates $u=t-r, v=t+r$. We define the radiation field of $\psi$ by $\Psi := r\psi$. Besides, the $0$-th mode of $\Psi$ is denoted by $\Psi_{\ell=0}$, which is
$$
\Psi_{\ell=0} = {1\over 4\pi} \int_{\mathbb{S}^2} \Psi \, d^2 \mu,
$$
and its {higher-order} mode is denoted by $\Psi_{\ell\geq 1} = \Psi - \Psi_{\ell=0}$.

The main contribution of the present paper, in addition to {a revisited proof of} the global existence and asymptotic behavior of the solution to the wave equation \eqref{Qwave:M}, is to illustrate the precise poinwise decay of $\psi$ and demonstrate that this decay rate is optimal for generic initial data. {Let the weighted energy $\Energy{k}{p}{\tau}{\varphi}$ for any scalar function $\varphi$ be defined as in Section \ref{subsect:energies}.} 
Our main theorem states as follows.

\begin{theorem}
\label{thm:main}
Let $k\geq 20$ be an integer, and let $\delta \in (0,\frac12)$ be a small constant. There exists an $\eps>0$ sufficiently small such that for all initial data $(\psi_0, \psi_1)$ satisfying 
\begin{align}
\Energy{k}{2-\delta}{\tau_0}{\Psi}< \eps^2,
\end{align} 
the quasilinear wave equation \eqref{Qwave:M} with null condition \eqref{nullcond:M} admits a global solution $\psi$, which decays as 
\begin{equation}
|\psi(t, x)|
\leq C v^{-1} u^{-\frac{1}{2} + \frac{3\delta}{2}}.
\end{equation}

If we additionally assume 
\begin{equation}
\Energy{{k-10}}{p=3-\delta}{\tau_0}{\Psi_{\ell=0}} 
+
\Energy{k-10}{p={1+6\delta}}{\tau_0}{r^2V\Psi_{\ell\geq 1}}
<
\varepsilon^2
\end{equation}
and that there are constants $c_{init}$ and $D\geq 0$ such that 
\begin{equation}
\Big| {1\over 4\pi}\big(r^2 V \int_{\mathbb{S}^2}\Psi \, d^2\mu\big)\big|_{\Sigma_{\tau_0}}  - c_{init} \Big|
\leq D r^{-\delta}
\end{equation}
for some constants $c_{init}$ and $D>0$,
then there exist constants $c_{total},  C$ such that
\begin{equation}\label{eq:precise-decay}
\big| \psi(t, x) - c_{total} v^{-1} u^{-1}  \big|
\leq C {(D + |c_{init}| + \varepsilon)} v^{-1} u^{-1-\delta} .
\end{equation}
Further, the constant $c_{total}$ is non-zero for a generic set of {such} initial data.
\end{theorem}

\begin{remark}
What we mean by ``generic'' in the last sentence of Theorem \ref{thm:main} is that the set of initial data such that the constant $c_{total}$ is non-zero is open and dense with respect to a suitable topology. See more in Theorem \ref{thm:genericity}.
\end{remark}

\begin{remark}
As far as we know, the pointwise asymptotic behavior for the solutions to nonlinear wave equations in the existing literature is purely an upper bound, from which we do not know whether the decay is sharp or not. In contrast, our result provides a precise pointwise description on the solution $\psi$, i.e., we can obtain both upper and lower bounds of decay for the solution $\psi$.
\end{remark}

\begin{remark}
From the precise pointwise estimate of the solution $\psi$ in \eqref{eq:precise-decay}, we know that for any $\epsilon' \in (0, 1)$, there exists $T_{\epsilon'}$, such that the solution $\psi$ has a fixed sign in the region {$\{(t, x) : t \geq T_{\epsilon'}, \text{ and } |x|\leq (1- \epsilon') t  \}$} generically.
\end{remark}

\subsection{Brief history and related results}

\

The quasilinear wave equation with the null condition in three space dimensions is a basic model in studying nonlinear wave-type equations. 
In the celebrated works of Christodoulou \cite{Christodoulou86} (using the conformal method) and Klainerman \cite{Klainerman86} (using the vector field method) global existence and the asymptotic behavior of the solution to \eqref{Qwave:M} were proved for small smooth initial data. These lead to various tremendous advances in the field of nonlinear wave-type equations, and below we list some of them. 
\begin{itemize}
\item Nonlinear stability results for physical models governed by wave-type equations:
The nonlinear stability of the Minkowski spacetime was first established in the seminal work of Christodoulou and Klainerman \cite{Christo-Klainer1993}, and there appeared also other alternative proofs \cite{LindbladRodnianski2005, LindbladRodnianski2010} for instance.

\item Wave systems with multispeeds:
Global existence for nonlinear elastic waves was built by Sideris in  \cite{Sideris2000} and Agemi in \cite{Agemi2000}, and global existence and pointwise decay for wave systems with multispeeds were obtained by Yokoyama in \cite{Yokoyama2000} and by Sideris and Tu in \cite{Sideris2001}.

\item Waves in lower spacetime dimensions:
The quasilinear wave equation with the null condition \eqref{Qwave:M} in two space dimensions was shown to admit a global solution for small smooth initial data by Alinhac \cite{Alinhac2001}.

\item Generalizations of the classical null condition \eqref{nullcond:M}:
Such generalizations include the ``weak null condition" of Lindblad and Rodnianski \cite{LindbladRodnianski2003, LindbladRodnianski2005, LindbladRodnianski2010} and the ``non-resonance" condition by Pusateri and Shatah \cite{PusateriShatah2013}, etc. In particular, we mention some further results related to the weak null condition, which include \cite{Alinhac2003, Lindblad2008, Katayama2015, DengPusateri2020, Keir18}.
\end{itemize}

One notes that almost all of the existing results regarding the pointwise behavior of the wave solutions (after global existence) are actually upper bounds. Few results provide a precise description (i.e., both upper and lower bounds) of the pointwise decay of a wave solution. A recent result which is related to ours is the one \cite{DengPusateri2020} by Deng and Pusateri on a weak null quasilinear wave equation in $\RR^{1+3}$
\begin{equation} \label{eq:3D-weak}
-\Box w = w \Delta w,
\qquad
\Delta=\sum_{j=1,2,3}\del_{x^j}\del_{x^j}.
\end{equation}
Global existence and pointwise decay for this system were obtained in earlier works \cite{Alinhac2003, Lindblad2008}. Based on some Sobolev estimates in \cite{Lindblad2008}, the authors of \cite{DengPusateri2020}  obtained a detailed description of the behavior of the differentiated solution (such as $\partial w$) near the light cone in terms of the solution to a derived asymptotic system. The proof there relies on Fourier analysis based methods. As a contrast, our method is a physical-space approach, and we can derive the precise pointwise decay for the undifferentiated solution $\psi$ to \eqref{Qwave:M}.

In 2010, Dafermos and Rodnianski in \cite{DafermosRodnianski:rp} introduced the $r^p$-weighted energy method, which is a big success in treating wave equations in backgrounds with or without black holes. This method was then further developed by the many authors. We mention some (but not being exhaustive) recent results which are related to ours and are achieved using the $r^p$-weighted energy method: Keir in \cite{Keir18} proved global existence results for a large class of nonlinear wave equations satisfying weak null condition; 
Angelopoulos--Aretakis--Gajic \cite{angelopoulos2021late} and Ma--Zhang \cite{MZ21PLKerr} demonstrated the precise pointwise asymptotic behavior for linear wave-type equations in a Kerr black hole background. The analysis in the present paper is also based on the $r^p$-weighted energy method.
One feature of the $r^p$-weighted energy method is that one derives the pointwise information of a wave solution from the decay of its {standard} energy, which is in contrast to the classical vector field method of Klainerman where the pointwise behavior is obtained from the weighted energy.

\subsection{A heuristic argument on the sharp decay using the fundamental solution}

Before moving to an outline of the proof, let us heuristically argue {by} using the fundamental solution why the sharp decay rate for solutions to the quasilinear wave equation with the null condition is $v^{-1}u^{-1}$.

In Minkowski spacetime with odd spatial dimension, the {free linear} wave is known to satisfy the strong Huygens' principle, i.e., when the initial data are supported in the unit disc at $\{t=0\}$, the solution vanishes identically in $\{|r-t|\geq 1\}$, where $r$ denotes the radius parameter. In {other words}, there is no {\sl tail} coming after the {wavefront}.

As suggested {by Bizo\'{n}} in \cite{Bizon-2008}, the fact that a wave system satisfying the strong Huygens' principle is very rare and unstable. In particular, a nonlinear wave system cannot be expected to preserve such a delicate property. In other words, tails will form after the wave front for general wave systems.

Let us consider for simplicity the case that the initial data are imposed on $\{t=2\}$, and are $C_c^{\infty}$ in the unit disc and small in Sobolev norm with sufficient regularity. Then the corresponding solution is supported in $\{r\leq t-1\}$. 
Let $\psi$ be a global solution to \eqref{Qwave:M}. Then $\psi = \psi_{\text s} + \psi_{\text{i}}$, with
$$
\Box \psi_{\text{s}} =  P^{\alpha\beta\gamma}\del_{\gamma}\psi\del_{\alpha}\del_{\beta}\psi\quad \text{and}\quad
\Box\psi_{\text{i}} = 0,
$$
and
$$
\psi_{\text{s}}({2},x) = \del_t\psi_{\text{s}}(2,x) = 0,\quad \psi_{\text{i}}(2,x) = \psi(2,x),\quad 
\del_t\psi_{\text{i}}(2,x) = \del_t\psi(2,x).
$$
We note that in $\{r\leq t-1\}$, the null condition implies
$
|P^{\alpha\beta\gamma}\del_{\gamma}\psi\del_{\alpha}\del_{\beta}\psi|
\lesssim t^{-3}{(t-r)^{-3}}$. The decay rate in ${ t-r}$  depends  on the decay estimates of $\psi$ (and its derivatives) we used and may be improved, however, the $t$-decay is no better than $t^{-3}$.
Without loss of generality, we consider the point $|x|=0$. By recalling the Kirchhoff's formula, we obtain 
$$
|\psi_{s}(t,0)| \lesssim \int_{\frac{t+1}{2}}^t\frac{d\tau}{\tau^3(t-\tau)}\int_{|y| = t-\tau}| t-2\tau|^{-3}d\sigma(y) \lesssim t^{-4}\int_{{\frac{1}{2}(1+t^{-1})}}^1\frac{(1-\eta)d\eta}{\eta^3(\eta-1/2)^{-3}}
$$
by letting $\tau=\eta t$.
Near the only singular point $\eta\rightarrow (1/2)^{+}$, we have a margin $t^{-1}$. This leads to $|\psi_{s}(t,0)|\lesssim t^{-2}$. Note that the above scheme relies only on the decay rate in $t$ for $P^{\alpha\beta\gamma}\del_{\gamma}\psi\del_{\alpha}\del_{\beta}\psi$. This observation suggests that $v^{-1}u^{-1}$ {shall be} the sharp decay rate.

%%%%%%%%%%%%%%%%%%%
\subsection{Outline of the proof}
%%%%%%%%%%%%%%%%%%%

\

To show the precise asymptotic behavior of the solution $\psi$ in \eqref{eq:precise-decay}, we need to extract the leading-order term. Besides, after the leading-order term $c_{total}v^{-1} u^{-1}$ is obtained, we need to show this term is generically non-zero, i.e., $c_{total}\neq 0$ for generic initial data. Our proof of Theorem \ref{thm:main} can be summarized as {three} parts.

$Part~ 1:$ \textit{Global existence, energy decay, and weak pointwise decay}. The global existence for \eqref{Qwave:M} with null condition \eqref{nullcond:M} is classical and well-known for initial data posed on a $t=constant$ hypersurface. However, since we pose the initial data \eqref{eq:ID} on a hyperboloid (a spacelike, asymptotically null hypersurface), so we need to provide a proof. For this part, we can use any method to show global existence, and we choose the $r^p$-weighted energy method as we can establish a global $r^p$ estimate for the solution for $p\in [\delta, 2-\delta]$ as stated in Theorem \ref{thm:globalrp:M}. Next, we show energy decay for the solution by an integrated energy decay from Theorem \ref{thm:globalrp:M}. Then the weak pointwise decay can be derived from the obtained energy decay and some weighted Sobolev inequalities. 

$Part~ 2:$ \textit{Almost sharp decay and leading-order part}. To achieve the goals, we decompose the solution $\psi$ into the $0$-th mode $\psi_{\ell=0}$ and the higher mode $\psi_{\ell\geq 1} = \psi - \psi_{\ell=0}$. For the $\psi_{\ell=0}$ part, we establish $r^p$ estimates for a sharp range of $p \in [\delta, 3-\delta]$, which are used to derive almost sharp decay for $\psi_{\ell=0}$. Then, we can extract the leading-order term $c_{total} v^{-1} u^{-1}$ in $\psi_{\ell=0}$ which is composed of two parts: one is from {a weighted} integral of the nonlinear terms at null infinity, and the other is from the asymptotics of the initial data {towards infinity}. As for the higher mode part of $\psi_{\ell\geq 1}$, its $r^p$ estimates for $p\in [\delta, 1+6\delta]$ can be shown to enjoy extra decay by commuting the equation of {$r\psi_{\ell\geq 1}$ with} the vector field $r^2 V$. This yields a faster decay of the higher mode $\psi_{\ell\geq 1}$, and consequently the leading-order term in the $0$-th mode $\psi_{\ell=0}$ {turns} out to be also the leading-order term of $\psi$.

$Part~ 3:$ \textit{Genericity}. 
We want to show the constant $c_{total}$ in \eqref{eq:precise-decay} is generically non-zero. In other words, we want to show that for a subset of initial data, which is open and dense in the set of initial data under consideration under certain topology, the constant $c_{total}$ is non-zero. For the openness part, it can be shown via a similar argument as the above $Part~ 2$. To show it is dense, the argument relies on the key observation that suitable perturbations on the initial asymptotics would be the dominant part in affecting {the value of} the constant $c_{total}$.

%%%%%%%%%%%%%%%%%%%
\subsection{Further discussions}
%%%%%%%%%%%%%%%%%%%

\

%%%%%%%%%%%%%%%%%%%
\subsubsection{Quasilinear wave for compactly supported initial data}
%%%%%%%%%%%%%%%%%%%

It is known that quasilinear null forms $P(\del\psi,\del^2\psi)=P^{\alpha\beta\gamma}\del_{\gamma}\psi\del_{\alpha}\del_{\beta}\psi$ are linear combinations of the following three groups of classical null terms (one refers to \cite{LiDong2021} for instance)
\begin{equation}
\label{eq:classify:nullform}
\begin{cases}
\mathbf{P}_1:=\{\partial_\alpha \psi \Box \psi, \,\, \alpha =0, 1, 2, 3\},
\\
\mathbf{P}_2:=\{\partial_{\alpha} (\partial_\gamma \psi \partial^\gamma \psi), \,\, \alpha =0, 1, 2, 3\},
\\
\mathbf{P}_3:=\{\partial_\alpha \psi \partial_\beta \partial_\gamma \psi - \partial_\beta \psi \partial_\alpha \partial_\gamma \psi, \,\,  \alpha, \beta, \gamma =0, 1, 2, 3, \text{and }\alpha<\beta\}.
\end{cases}
\end{equation}

In most of the existing results for nonlinear wave-type equations people consider initial data posed on a $t=constant$ slice; see for instance \cite{Klainerman86, Christodoulou86, Sideris2000, Agemi2000, Alinhac2001, Alinhac2003, Lindblad2008}.
However, in Theorem \ref{thm:main} we consider the initial data posed on a hyperboloid $\Sigma_{\tau_0}$, i.e.,
$$
\big(\psi, \partial_t \psi \big)\big|_{\Sigma_{\tau_0}}
= (\psi_0, \psi_1).
$$
The Cauchy problem formulated with initial data posed on these two types of initial hypersurfaces are different, but they coincide when the initial data have compact supports.
In the case of compactly supported initial data, we have the following version of Theorem \ref{thm:main} for the quasilinear wave equation \eqref{Qwave:M}.

\begin{theorem}\label{thm:main-compact}
Let $k\geq 20$ be an integer. There exists an $\eps>0$ sufficiently small such that for all initial data $(\psi_0, \psi_1)$ that are compactly supported in $\Sigma_{\tau_0}$  and satisfy 
\begin{align}
\| \psi_0 \|_{W_0^{k+1}(\Sigma_{\tau_0})} + \|\psi_1\|_{W_0^{k}(\Sigma_{\tau_0})}  < \eps,
\end{align} 
the quasilinear wave equation {\eqref{Qwave:M}} with null condition \eqref{nullcond:M} admits a global solution $\psi$.
In addition, there exist constants $c_{total}, {\delta}>0, C$ such that
\begin{equation}\label{eq:precise-decay:compact}
\big| \psi(t, x) - c_{total} v^{-1} u^{-1}  \big|
\leq C v^{-1} u^{-1-{\delta}}.
\end{equation}
Further, if the quasilinear terms are of the forms in the sets $\mathbf{P}_1$ and $\mathbf{P}_2$, then the constant $c_{total}=0$ and the decay for $\psi$ is faster by $u^{-1}$, that is, there exists constants $c_{total}', {\delta'}>0, C'$ such that
\begin{equation}\label{eq:precise-decay:compact:special}
\big| \psi(t, x) - c_{total}' v^{-1} u^{-2}  \big|
\leq C' v^{-1} u^{-2-{\delta}'}.
\end{equation}
\end{theorem}

\begin{remark}
In Theorem \ref{thm:main-compact}, the constant $c_{total}$ is completely determined by the integral of the nonlinear terms at infinity, but we do not know whether there exists non-zero compactly supported initial data such that $c_{total}\neq 0$, which is quite different from Theorem \ref{thm:main}. Thus we cannot assert  {if such a pointwise decay rate is generically sharp or not}  for compactly supported initial data. 
\end{remark}

%%%%%%%%%%%%%%%%%%%
\subsubsection{Semilinear wave equation with null condition for compactly supported initial data}
%%%%%%%%%%%%%%%%%%%

We have extensively considered the quasilinear wave equation with null condition in the above discussions. A similar, yet simpler, model problem is the semilinear wave equation with null condition, and we now discuss the sharp decay rates for the solutions to this semilinear model problem which actually differs much from the quasilinear case since it enjoys arbitrarily fast polynomial decay {in $u$}. 

Consider the semilinear wave equation  $
\Box\psi=Q^{\alpha\beta}\del_{\alpha}\psi \del_{\beta}\psi$ in $\mathbb{R}^{1+3}$,
 with the constant coefficients $Q^{\alpha\beta}$ satisfying the null condition
$
Q^{\alpha\beta}\xi_{\alpha}\xi_{\beta}=0$ for all $\xi_0^2= \xi_1^2 + \xi_2^2 +\xi_3^2$. Indeed, such a semilinear quadratic null form has to take the form of $\del^{\alpha}\psi\del_{\alpha}\psi$ {(up to a constant)}, hence the solution solves the following wave equation 
\begin{align}
\Box\psi=\del^{\alpha}\psi\del_{\alpha}\psi.
\end{align}
Let us assume the initial data are smooth and compactly supported {for clarity}. The same approach in proving Theorem \ref{thm:main-compact} for the quasilinear case applies and the estimate \eqref{eq:precise-decay:compact} holds. However, one can easily verify that $c_{total}=0$, which suggests that $v^{-1}u^{-1}$  is not the optimal decay for solutions to this semilinear model equation.

 In fact, one can show $|\psi|\lesssim_{n} v^{-1} u^{-n}$ for any $n\geq 1$. We take $\psi_{(1)}=\psi-\frac{1}{2}\psi^2$, then it solves
\begin{align}
\Box\psi_{(1)}=-\psi\Box\psi=-\psi\del^{\alpha}\psi\del_{\alpha}\psi,
\end{align}
where the right-hand side again satisfies the null condition and has additional decay arising from the presence of a new factor $\psi$. This yields $v^{-1}u^{-2}$ decay for $\psi_{(1)}$, and hence for $\psi$ itself. We next define $\psi_{(2)}=\psi_{(1)} +\frac{1}{6}\psi^3$ and it satisfies $\Box \psi_{(2)}=\frac{1}{2} \Box \psi \psi^2=\frac{1}{2} \psi^2 \del^{\alpha}\psi\del_{\alpha}\psi$. Again, this yields $v^{-1} u^{-3}$ decay for $\psi_{(2)}$, hence for $\psi_{(1)}$ and $\psi$ as well. Such a scheme can be proceeded to any order $n\geq 1$ by defining 
$
\psi_{(n)}= \sum_{j=1}^{n+1}\frac{(-1)^{j-1}}{j!}\psi^j
$
and deriving its governing equation
$$
\Box \psi_{(n)}=\frac{(-1)^n}{n!}\psi^{n} \del^{\alpha}\psi\del_{\alpha}\psi,
$$
and this implies $v^{-1} u^{-n-1}$ decay for $\{{\psi_{(j)}}\}_{j=1,2,\ldots,n}$ and $\psi$.

If we are instead considering the semilinear wave equation with null condition with initial data that have non-compact support and decay polynomially towards infinity, then the decay rate of the solution is purely determined by this polynomial decay rate of the initial data towards infinity.   {We also note that for semilinear waves with the null condition in non-stationary, asymptotically flat spacetimes, Looi--Tohaneanu \cite{Looi2022Nullcondition} recently obtained sharp upper bound of pointwise decay.}

\subsection*{Organization of the paper}
In Section \ref{sect:conventions}, we set up the coordinates and foliations of the background and introduce some notation and conventions for later use.
Next, in Section \ref{sect:analyticpreliminary}, we prepare some analytic tools which in particular include the standard energy and Morawetz estimates for wave equations and an $r^p$ estimate near infinity for inhomogeneous wave equations.
Then we show global existence, weak energy decay, and weak pointwise decay for the solution $\psi$ in Section \ref{sect:existenceanddecay}.
The precise pointwise decay result and its genericity are proved in Sections \ref{sect:sharpdecay} and \ref{sect:genericity}, respectively.

\textit{Acknowledgement:}
S.M. acknowledges the support from Alexander von Humboldt foundation.

%%%%%%%%%%%%%%%%%%
%%%%%%%%%%%%%%%%%%

\section{Conventions and notation}
\label{sect:conventions}

%%%%%%%%%%%%%%%%%%
%%%%%%%%%%%%%%%%%%

This section is devoted to introducing the coordinates, foliations and some notation and making conventions. 

%%%%%%%%%%%%%%%%%%
\subsection{Coordinates and foliations of the spacetime}
%%%%%%%%%%%%%%%%%%
We work in the manifold $\MM:=\RR\times \RR^3$ equipped with the standard Minkowski metric $g^{\alpha\beta}$ with signature $(-,+,+,+)$. Let $r = \sqrt{\sum_{a=1}^3|x^a|^2}$ be the {spatial} radius. We call $(\MM,g^{\alpha\beta})$ as the Minkowski spacetime. The manifold equals $\MM=\{(t,r,\omega)| t\in \RR, r\in [0, +\infty), \omega\in \mathbb{S}^2\}$.

Define the  retarded and forward null coordinates
\begin{align}
u := t-r,\qquad v:= t+r .
\end{align}
The coordinates $(u, v, \omega)$ are the so-called double null coordinates of the Minkowski spacetime $(\MM, {g^{\alpha\beta}})$, and the manifold $\MM$ is identified as $(-\infty, +\infty) \times (-\infty, +\infty) \times \mathbb{S}^2$ in this double null coordinate system. One can relate the coordinate derivatives in different coordinate systems by
\begin{align}
\del_u=\frac{1}{2} (\del_t -\del_r), \qquad \del_v=\frac{1}{2} (\del_t +\del_r).
\end{align}
Further, for later convenience, we define
\begin{equation}
U := 2\del_u = \del_t-\del_r,\qquad V:= 2\del_v = \del_t+\del_r.
\end{equation}
They are the {future-pointing} ingoing and outgoing null derivatives, respectively.

\subsubsection{Construction of a hyperboloidal foliation}
\label{subsubsect:constructionfoliation}

\def\hh{\mathfrak{h}}
\def\hhp{\hh'}
Fix a constant ${\eta\in (0,1]}$. Define a \textit{hyperboloidal time function}
\begin{align}
\tau:=t- \hh(r)
\end{align}
 with $\hh=\hh(r)$ being a smooth function of $r$ and $\hhp=\del_r\hh$ satisfying the following requirements:
 \begin{enumerate}[label=\roman*)]
\item  constant-$\tau$ hypersurfaces are spacelike hypersurfaces with respect to the Minkowski metric, that is,
$|\hhp |<1$;
\item and  there are positive constants $0<c_{\hh}<C_{\hh}<+\infty$ such that 
 $c_{\hh} r^{-1-\eta}\leq 1-\hhp \leq C_{\hh} r^{-1-\eta}$ for $r$ large.
\end{enumerate}
We call $(\tau, \rho, \omega)$, where $\rho=r$ and $\omega\in \mathbb{S}^2$, as the \textit{hyperboloidal coordinates}. In such a  hyperboloidal coordinate system, one finds $\MM=(-\infty, +\infty) \times [0, +\infty)\times \mathbb{S}^2$, and the coordinate derivatives are 
\begin{align}
\label{relation:partialderis:diffcoords}
 \bar{\del}_r :=\del_{\rho}= \del_r + \hhp \del_t = V+ (\hhp -1)\del_t , \qquad \del_\tau=\del_t. 
\end{align}

In our framework, we fix a $\tau_0\geq 1$, and $\Sigma_{\tau_0}$ is the hypersurface where our initial data for equation \eqref{Qwave:M} are imposed.

Indeed, such a height function $\hh$ can be easily constructed by taking, for instance, $\hh = \int_{0}^r (1-\la r'\ra^{-1-\eta}) dr'$. The second requirement ensures that the constant-$\tau$ hypersurfaces are asymptotically null to the future null infinity, since ${g^{\alpha\beta} \partial_\alpha \tau \partial_\beta \tau}=-1+(\hhp)^2\simeq r^{-1-\eta}$ for large $r$.   
Note that $\tau-u \simeq {r^{-\eta}}$ for large $r$, while in finite $r$, $\tau\simeq t$.

\subsubsection{Geometry of the hyperboloidal foliation}

We define a few submanifolds of $\MM$ in this foliation. Let $\II$ be the future null infinity. Let $\tau',\tau_1,\tau_2\in \RR$ and $\tau_1<\tau_2$. Define 
\begin{align}
\Sigma_{\tau'}:={}&\MM\cap \{\tau=\tau'\}, \\
\MMonetwo:={}&\MM\cap\{\tau_1\leq \tau\leq \tau_2\},\\
\IIonetwo:={}&\II\cap \{\tau_1\leq \tau\leq \tau_2\}.
\end{align}
Additionally, we define truncated subregions of $\Sigma_{\tau'}$ and $\MMonetwo$ by
\begin{subequations}
\begin{align}
\Sigma_{\tau'}^{\geq r_0}:=&\Sigma_{\tau'}\cap\{\rho\geq r_0\}, \quad\forall \, r_0\geq 0,\\
\MMonetwo^{\geq r_0}:=&\MMonetwo\cap\{\rho\geq r_0\},\quad\forall \,r_0\geq 0,\\
\Sigma_{\tau'}^{r_1,r_2}:=&\Sigma_{\tau'}\cap\{r_1\leq \rho\leq r_2\},\quad\forall \,0\leq r_1<r_2,\\
\MMonetwo^{r_1,r_2}:=&\MMonetwo\cap\{r_1\leq \rho\leq r_2\},\quad\forall \,0\leq r_1<r_2.
\end{align}
\end{subequations}

	For the geometry of $\Sigma_{\tau}$, since $d\tau = dt - \hh'(r)dr$, one has, within Euclidean metric of $\RR^4\simeq \RR^{1+3}$, {that the future-directed unit normal vector of $\Sigma_\tau$ is}
	$$
	\vec{n}_{\Sigma_{\tau}} = \frac{1}{\sqrt{1+|\hh'(r)|^2}} \big(\del_t - \hh'(r)\del_r\big).
	$$
	And the volume element is written as
	$$
	d\Sigma_{\tau} = \sqrt{1+|\hh'(r)|^2}\, dx^1dx^2dx^3 = \sqrt{1+|\hh'(r)|^2} r^2\, drd\omega.
	$$
	Finally
	$$
	\vec{n}_{\Sigma_{\tau}}d\Sigma_{\tau} = \big(\del_t - \hh'(r)\del_r\big)r^2\, drd\omega.
	$$

%%%%%%%%%%%%%%%%%%
\subsection{Basic conventions}
%%%%%%%%%%%%%%%%%%

For any real value $x$, denote  $\langle x\rangle=\sqrt{x^2+1}$.

Let $\mathbb{R}$ be the set of real numbers, and let $\mathbb{N}$ be the set of natural numbers $\{0,1,2,\ldots\}$.

Constants in this work may depend on the hyperboloidal foliation via the height function $\hh$. For simplicity, we shall always suppress this dependence throughout this work since one can fix this function a priori. 

A universal constant is a constant that depends only on the wave equation and the height function $\hh$ and is usually denoted by $C$. 
If a constant depends additionally on a set of other parameters $\mathbf{P}$, we denote it by $C(\mathbf{P})$.

We say $F_1\lesssim F_2$ if there exists a universal constant $C$ such that $F_1\leq CF_2$. Similarly for $F_1\gtrsim F_2$. If both $F_1\lesssim F_2$ and $F_1\gtrsim F_2$ hold, we say $F_1\simeq F_2$.

Let $\mathbf{P}$ be a set of parameters. We say $F_1\lesssim_{\mathbf{P}} F_2$ if there exists a constant $C(\mathbf{P})$ such that $F_1\leq C(\mathbf{P})F_2$. Similarly for $F_1\gtrsim_{\mathbf{P}} F_2$. If both $F_1\lesssim_{\mathbf{P}} F_2$ and $F_1\gtrsim_{\mathbf{P}} F_2$ hold, then we say $F_1\simeq_{\mathbf{P}} F_2$.

For a function $f=f(r,\omega)$ and any integer $n$, we say  $f=O(r^{n})$ in the region $\{r\geq 1\}$ if for any $j\geq 0$ and any {multi-index} $\alpha$, there exists $C_{j, |\alpha|}<+\infty$ such that $f$ satisfies 
\begin{align}
	|\Omega^{\alpha}\del_r^{j_1} f|\leq C_{j,|\alpha|} r^{n-j},
\end{align}
{where $\Omega$ is the set of rotational vector fields as in Section \ref{subsec:norm}.}

{The Einstein summation convention is adopted for repeated indices unless otherwise specified.}

%%%%%%%%%%%%%%%%%%
\subsection{Operators and norms}\label{subsec:norm}
%%%%%%%%%%%%%%%%%%

Let $\del$ be the operator set comprised of $\{\del_t, \del_{x^1}, \del_{x^2}, \del_{x^3}\}$.

Let $\nablas$ be the metric connection over unit sphere $\mathbb{S}^2$ and $\gs$ be its associated metric on $\mathbb{S}^2$. 
Denote  $\Deltas:=\gs^{\mu\nu}\nablas_{\mu}\nablas_{\nu}$ as the spherical Laplacian over $\mathbb{S}^2$. We denote $|\nablas\psi|^2:=\gs^{\mu\nu}\nablas_{\mu} \psi\nablas_{\nu}\psi$.

Let $\Omega_i=\in_{ijk} x^j \del_{x^k}$, $i,j,k\in\{1,2,3\}$, be the rotational Killing vector fields on $\mathbb{S}^2$, where $\in_{ijk}$ is antisymmetric in its indices and $\in_{123}=1$. They span the tangent space at any point on $\mathbb{S}^2$, and they satisfy $[\Omega_i,\Omega_j]=-\sum_{k=1,2,3}\in_{ijk}\Omega_k$  and $\Omega_i(r)=0$ for any $i,j\in \{1,2,3\}$. Denote the set of these rotational vector fields by $\Omega$.

The following lemma provides some relations among $\nablas$, $\Deltas$, and $\Omega_i$, and they can be easily verified.

\begin{lemma}
\label{lem:nablasandOmega:relation}
For a scalar function $\varphi$,
\begin{align}
|\nablas\varphi|^2 = \sum_{i=1,2,3} |\Omega_i \varphi|^2, \qquad \Deltas\varphi = \sum_{i=1,2,3}\Omega_i^2 \varphi.
\end{align}
\end{lemma}

Define a set of operators, which is in particular used in proving a general $r^p$ lemma, by
\begin{equation}
\label{def:commset:DDb}
\DDb:=\{\del_t, rV\}\cup \Omega.
\end{equation}

To properly define norms on manifolds, we shall first introduce a few reference volume elements.

Let $d^2\mu$ be the volume element on unit sphere $\mathbb{S}^2$. Further, we denote a couple of reference volume forms {in the $(\tau, \rho, \omega)$ coordinates} by
\begin{align}
d^3\mu:=d\rho d^2\mu,\qquad 
d^4\mu:=d\tau d^3\mu.
\end{align}

\begin{definition}
\label{defs:generalnorms:M}
Let $\varphi$ be a scalar function on $\MM$, and let $k\in \NN$. Define its $k$-th order pointwise norm by
\begin{subequations}
\label{defs:generalPTWandEnergynorms:M}
\begin{align}
\label{defs:generalPTWnorms:M}
\norm{\varphi}{k}:=&\sum_{|\alpha|+|\beta|\leq k} \vert \DDb^{\alpha}(r \del^{\beta}(r^{-1}\varphi))\vert.
\end{align}
 
Define yet another $k$-th order pointwise norm by 
\begin{align}
\label{defs:generalPTWnorms:M}
\norm{\varphi}{(k)}:=&\sum_{|\alpha|+|\beta|\leq k} \vert \DDb^{\alpha} \del^{\beta}\varphi\vert.
\end{align}
{Let $p\in \RR$ and define a $k$-th order} weighted norm over a spacetime domain $\mathcal{D}\subseteq \MM$ by
\begin{align}
\label{defs:generalSTnorms:M}
\Wnorm{\varphi}{k}{p}{\mathcal{D}}:=&\bigg(\int_{\mathcal{D}} \la r\ra^p \norm{\varphi}{k}^2 d^4\mu \bigg)^{\frac12}.
\end{align}
Similarly, for a $3$-dimensional {hypersurface} $\Sigma$ that can be parameterized by $(\rho, \omega\in\mathbb{S}^2)$, we define its {$k$-th order} weighted norm over $\Sigma$ by
\begin{align}
\label{defs:generalEnergynorms:M}
\Wnorm{\varphi}{k}{p}{\Sigma}:=&\bigg(\int_{\Sigma} \la r\ra^p \norm{\varphi}{k}^2 d^3\mu \bigg)^{\frac12}.
\end{align}
\end{subequations}
\end{definition}

\begin{remark}
In application, the norm $\norm{\varphi}{(k)}$ is used for the scalar field $\varphi=\psi$, while the norm $\norm{\varphi}{k}$ is used for the radiation field $\varphi=\Psi=r\psi$. Notice that the two pointwise norms $\norm{\varphi}{k}$ and $\norm{\varphi}{(k)}$ are equivalent to each other in the region $\{r\geq 1\}$.
\end{remark}

%%%%%%%%%%%%%%%%%%
\subsection{Mode decomposition of scalar functions}
%%%%%%%%%%%%%%%%%%

Let $\varphi$ be a scalar function on $\mathbb{S}^2$. Define its spherical mean over $\mathbb{S}^2$  by
\begin{align}
\varphi_{\ell=0}:=\frac{1}{4\pi}\int_{\mathbb{S}^2} \varphi \,d^2\mu.
\end{align}
The subscript $\ell=0$ indicates that the above defined $\varphi_{\ell=0}$ is actually the $\ell=0$ spherical harmonic mode of $\varphi$ when one decomposes the scalar $\varphi$ with respect to the spherical harmonics over $\mathbb{S}^2$.

In addition,  we define
\begin{align}
\varphi_{\ell\geq 1}:=\varphi -\varphi_{\ell=0},
\end{align}
which in fact corresponds to the $\ell\geq 1$ modes of $\varphi$.

These definitions naturally extend to scalar functions on $\MM$.

%%%%%%%%%%%%%%
\subsection{Energies}
\label{subsect:energies}
%%%%%%%%%%%%%%

Recall that $\eta\in (0,1]$ is a fixed constant that characterizes the hyperboloidal foliation in Section \ref{subsubsect:constructionfoliation}.  For any $k\in\NN$, $p\geq 0$,  $\tau\geq\tau_0$ and any real scalar function $\varphi$, let $\peta:=\max\{-2, p-\eta-3\}$ and define a $k$-th order weighted energy of $\varphi$ at $\Sigma_{\tau}$ by
 \begin{align}
\label{def:globalrpEnergy:highorder:M}
\Energy{k}{p}{\tau}{\varphi}:=  &\Wnorm{rV\varphi}{k}{p-2}{\Sigma_{\tau}}^2
+\Wnorm{U\varphi}{k}{-1-\eta}{\Sigma_{\tau}}^2
+\Wnorm{\nablas\varphi}{k}{\peta}{\Sigma_{\tau}}^2
+\Wnorm{\varphi}{k+1}{\peta}{\Sigma_{\tau}}^2 .
\end{align}
For any $\tau_2>\tau_1\geq \tau_0$, define a $k$-th order weighted energy of $\varphi$ at future null infinity $\II_{\tau_1,\tau_2}$ by
\begin{align}
\label{def:weightedscriflux:M}
\Flux{k}{\tau_1,\tau_2}{p}{\varphi}:= \sum_{\abs{\alpha}\leq k} \int_{\II_{\tau_1,\tau_2}} \big( r^{p-2}\abs{\nablas\DDb^{\alpha}\varphi}^2
+ \abs{U\DDb^{\alpha}\varphi}^2 \big) d\tau d^2\mu .
\end{align}
In this work, $\varphi$ will be taken to be one of $\Psi$, $\Psi_{\ell=0}$ and $\Psi_{\ell\geq 1}$. 

%%%%%%%%%%%%%%%%%%
%%%%%%%%%%%%%%%%%%

\section{Analytic preliminaries}
\label{sect:analyticpreliminary}

%%%%%%%%%%%%%%%%%%
%%%%%%%%%%%%%%%%%%

This section contains {several} important analytic estimates that are frequently used in this work. Some useful Hardy and Sobolev inequalities are collected in Section \ref{subsect:HardySobolev}. Afterwards, in Sections \ref{sect:BasicEnerEsti:M}--\ref{sect:HighEMEstis:M} we provide a simple proof of the standard energy estimate and Morawetz estimate for a general wave equation.  We next  show in Section \ref{subsect:rpnearinf:generalwave} a general $r^p$ lemma near infinity for inhomogeneous wave equations and in the end  present in Section \ref{subsect:HierImplyDecay} a lemma which proves energy decay from a hierarchy of integrated energy {decay} estimates.

%%%%%%%%%%%%%%%%%%
\subsection{Hardy and Sobolev inequalities}
\label{subsect:HardySobolev}

%%%%%%%%%%%%%%%%%%

The following Hardy inequality is useful. Its proof follows from a standard way of integrating $q^{-1}\del_r (r^{q} |h(r)|^2)$ from $r_0$ to $r_1$; one refers to \cite{andersson2019stability} for instance. 

\begin{lemma}
\label{lem:HardyIneq}
Let $q \in \mathbb{R}\backslash \{0\}$  and $h: [r_0,r_1] \rightarrow \mathbb{R}$ be a $C^1$ function.
\begin{enumerate}[label=\roman*)]
\item \label{point:lem:HardyIneqLHS} If $r_0^{q}\vert h(r_0)\vert^2 \leq D_0$ and $q<0$, then
\begin{subequations}
\begin{align}\label{eq:HardyIneqLHS}
-2q^{-1}r_1^{q}\vert h(r_1)\vert^2+\int_{r_0}^{r_1}r^{q -1} \vert h(r)\vert ^2 d r \leq \frac{4}{q^2}\int_{r_0}^{r_1}r^{q +1} \vert \partial_r h(r)\vert ^2 d r-2q^{-1}D_0;
\end{align}
\item \label{point:lem:HardyIneqRHS} If $r_1^{q}\vert h(r_1)\vert^2 \leq D_0$ and $q>0$, then
\begin{align}\label{eq:HardyIneqRHS}
2q^{-1}r_0^{q}\vert h(r_0)\vert^2+\int_{r_0}^{r_1}r^{q -1} \vert h(r)\vert ^2 d r \leq \frac{4}{q^2}\int_{r_0}^{r_1}r^{q +1} \vert \partial_r h(r)\vert ^2 d r +2q^{-1}D_0.
\end{align}
\end{subequations}
\end{enumerate}
\end{lemma}

The following Sobolev-type estimates are similar to the ones in \cite[Lemmas 4.32 and 4.33]{andersson2019stability} but adapted to our foliation. These are in particular used to derive pointwise decay estimates from energy decay estimates.

\begin{lemma}[Sobolev inequalities]
\label{lem:Sobolev}
It holds that
\begin{subequations}
\begin{align}
\label{eq:Sobolev:1}
\sup_{\Sigma_{\tau}}\abs{\varphi}\lesssim{} \Wnorm{\varphi}{3}{-1}{\Sigma_{\tau}}.
\end{align}
If $\varphi |_{r=0}=0$ and $q\in (0,1]$, then
\begin{align}
\label{eq:Sobolev:2}
\sup_{\Sigma_{\tau}}\abs{\varphi}\lesssim_{q} {}& \Big(\Wnorm{\varphi}{2}{-2}{\Sigma_{\tau}}^2
+\Wnorm{U\varphi}{2}{-1-q-2\eta}{\Sigma_{\tau}}^2
+\Wnorm{rV\varphi}{2}{-1-q}{\Sigma_{\tau}}^2\Big)^{\frac{1}{4}}\notag\\
&\times \Big(\Wnorm{\varphi}{2}{-2}{\Sigma_{\tau}}^2
+\Wnorm{U\varphi}{2}{-1+q-2\eta}{\Sigma_{\tau}}^2
+\Wnorm{rV\varphi}{2}{-1+q}{\Sigma_{\tau}}^2\Big)^{\frac{1}{4}}.
\end{align}
If $\lim\limits_{\tau\to\infty}\abs{r^{-1}\varphi}=0$ pointwise in $(r,\omega)$, then
\begin{align}
\label{eq:Sobolev:3}
\abs{r^{-1}\varphi}\lesssim {}\Big(\Wnorm{\varphi}{3}{-3}{\MM_{\tau,\infty}}^2
\Wnorm{\del_{\tau}\varphi}{3}{-3}{\MM_{\tau,\infty}}^2\Big)^{\frac{1}{4}}.
\end{align}
\end{subequations}
\end{lemma}

\begin{proof}
Dropping the dependence on $\tau$ and $\omega$, we have for any $r_1,  r\in [0,\infty)$ that
\begin{align}
\label{eq:Sobolev:pf:FTC:M}
|\varphi(r)|^2 \leq &\bigg|\int_{r_1}^{r}\delb_{r}\big( | \varphi({r'})|^2 \big)d {r'} \bigg| + |\varphi(r_1)|^2 \notag\\
\lesssim &{\bigg|}\int_{r_1}^{r} \la {r'}\ra ^{-1-q} | \varphi({r'})|^2 d {r'} \int_{r_1}^{r} \la {r'}\ra^{-1+q} |\la {r'}\ra\delb_r \varphi ({r'})|^2 d {r'}{\bigg|}^{1/2} + |\varphi(r_1)|^2,
\end{align}
for any $q\in \mathbb{R}$. 
By the mean-value principle, there exists a sequence of $\{r_{1,n}\}$, with $r_{1,n}\in [2^n, 2^{n+1})$,  such that $|\varphi(r_{1,n})|^{2}\lesssim \int_{2^n}^{2^{n+1}} (r')^{-1} |\varphi(r')|^2 dr'$. Thus, 
the first inequality \eqref{eq:Sobolev:1} follows by taking $q=0$ and $r_1=r_{1,n}$ in \eqref{eq:Sobolev:pf:FTC:M}, letting $n\to\infty$, and using the standard Sobolev imbedding on spheres.

In the region $r\in [0,1]$, the inequality \eqref{eq:Sobolev:2} follows immediately from \eqref{eq:Sobolev:pf:FTC:M} with $r_1=0$ and $q=0$. Let us now take $r_1=1$ and $r\geq 1$ in the above inequality \eqref{eq:Sobolev:pf:FTC:M}.  Then, by the Hardy inequality \eqref{eq:HardyIneqLHS}, it holds true {(for $q>0$)} that
\begin{align*}
&\int_{1}^{r} \la r'\ra ^{-1-q} | \varphi(r')|^2 d r' \notag\\
\lesssim_{q} &\int_{1}^{r} (r')^{-1-q} |{ r'}\delb_r\varphi(r')|^2 d r'+ |\varphi(r=1)|^2\notag\\
\lesssim_{q} &\int_{1}^{r} \big( (r')^{-1-q} | {r' V\varphi(r')}|^2  + (r')^{-1 - q -2\eta} |U\varphi|^2 \big)d r'+ |\varphi(r=1)|^2,
\end{align*}
since $\delb_r = V+ (\hhp-1)\del_t = \frac{\hhp +1}{2} V + \frac{\hhp-1}{2}U=O(1) V + O(r^{-1-\eta})U$. This thus proves the second inequality  \eqref{eq:Sobolev:2}.

By H\"older inequality, we have 
\begin{align}
\label{eq:Sobolev3:pf:Holder}
|r^{-1}\varphi|^2 = -\int_{\tau}^{\infty} \del_{\tau} (|r^{-1}\varphi|^2) d\tau' \leq 2\bigg(\int_{\tau}^{\infty} |r^{-1}\del_{\tau} \varphi |^2 d\tau' \bigg)^{1/2}  \bigg(\int_{\tau}^{\infty} |r^{-1}\varphi |^2 d\tau' \bigg)^{1/2} .
\end{align}
{The last inequality \eqref{eq:Sobolev:3} for $r\geq 1$ then follows in the same way as proving \eqref{eq:Sobolev:1}  on each $\Sigma_{\tau}$. In the region $r=|x|\leq 1$, where $x=(x^1, x^2, x^3)$, we apply the standard Sobolev imbedding estimate
$|h (x)|^2
 \lesssim \int_{|x|^2 \leq 2} |\del_{x}^{\leq 2} h(x)|^2 d x$ and find that the right-hand side of \eqref{eq:Sobolev3:pf:Holder} is bounded by $C\Wnorm{\varphi}{3}{-3}{\MM_{\tau,\infty}}
\Wnorm{\del_{\tau}\varphi}{3}{-3}{\MM_{\tau,\infty}}$. This hence proves inequality \eqref{eq:Sobolev:3}.}
\end{proof}

%%%%%%%%%%%%%%%%%%%
\subsection{Basic energy-Morawetz estimates for wave equations}
\label{sect:BasicEnerEsti:M}
%%%%%%%%%%%%%%%%%%%

The wave operator acting on a scalar field $\phi$ takes the following form:
\begin{equation}
\label{eq:wave:standard:M}
\Box \phi =- \del_t\del_t\phi+{r^{-2}} \del_r(r^2\del_r\phi) + r^{-2}\Deltas \phi.
\end{equation}
Denoting the radiation field of $\phi$ by $\Phi := r\phi$, the above form can be written as
\begin{equation}\label{eq1-07-06-2022}
\aligned
r\Box\phi =&-\del_t\del_t \Phi+ \del_r\del_r \Phi+ r^{-2}\Deltas\Phi.
\endaligned
\end{equation}
We define the following space-time region: 
$$
\Dcal^{r_{1},r_{2}}_{\tau_1,\tau_2} := \big\{(\tau,r)|\tau_1\leq \tau\leq \tau_2 , r_{1}\le r, t(\tau,r)+r\leq r_{2} \big\}.
$$
Note that, the frontier of $\Dcal^{r_{1},r_{2}}_{\tau_1,\tau_2}$ is composed by four pieces:
$$
\aligned
\del \Dcal^{r_{1},r_{2}}_{\tau_1,\tau_2} 
=& \big\{\tau = \tau_1, r_{1}\leq r\leq R(\tau_1,r_{2})\big\} \cup \big\{\tau=\tau_2, r_{1}\leq r\leq R(\tau_2,r_{2}) \big\} \\
&\cup \big\{\tau_{1}\le \tau\le \tau_{2}, r=R(\tau,r_{2}) \big\}\cup\big\{\tau_1\leq \tau\leq \tau_2, r=r_{1} \big\},
\endaligned
$$
where $R(\tau,\rho_0) $ is the solution to $t(\tau,r) + r = \rho_0$ for any fixed $(\tau,\rho_0)$.
For each piece, we write their normal vector and volume form (with respect to Euclidean metric with respect to $(\tau,r)$)
$$ 
\aligned
&\{\tau_1\}\times [0,\infty):\quad \quad \quad \ \vec{n} d\sigma_{\Sigma} = - \del_\tau\, dr,
\\
&\{\tau_2\}\times [0,\infty):\quad \quad \quad \ \vec{n} d\sigma_{\Sigma} = \del_\tau\,dr,
\\
&{[\tau_0, \infty) \times\{r=r_{1}\}}:\quad \vec{n}d\sigma = -\delb_r d\tau,
\\
&{
	\aligned
	\{t(\tau,r)+r = r_{2}\}:\quad \ \vec{n}d\sigma =& (\del_{\tau} + (1+\hhp)\delb_r)\, dr
	= (\frac{1}{1+\hhp}\del_{\tau} + \delb_r)\, d\tau.
	\endaligned
}
\endaligned
$$

First, we introduce the following basic energy estimate for~\eqref{eq1-07-06-2022}.

\begin{lemma}[Basic energy estimate for wave equations]
It holds that 
\begin{equation}\label{eq2-06-06-2022}
\aligned
&\SEnergy{0}{\tau_2}{\Phi}  
+ \int_{\IIonetwo}
\big(|U\Phi|^2 + r^{-2}| \nablas \Phi|^2\big) d\tau d^2\mu \\
\lesssim &\SEnergy{0}{\tau_1}{\Phi}  +\bigg|\int_{\MMonetwo}r\del_t\Phi \Box \psi d^4\mu\bigg|.
\endaligned
\end{equation}
Here, the standard energy $\SEnergy{0}{\tau}{\varphi}$ for a scalar $\varphi$ on $\Sigma_{\tau}$, $\tau\geq \tau_0$, is defined as
\begin{equation}
\label{standardEner:zero:M}
\SEnergy{0}{\tau}{\varphi} := \int_{\Sigma_\tau}\Big(|\delb_r\varphi|^2 + (1-(\hhp)^2)|\del_{\tau}\varphi|^2 + r^{-2}| \nablas\varphi|^2
+ r^{-2} |\varphi|^2 \Big)d^3\mu.
\end{equation}
\end{lemma}

\begin{proof}
	
	We multiply on both sides of~\eqref{eq1-07-06-2022} by $ -\del_t\Phi$ and then integrate on $\mathbb{S}^2$, 
	\begin{equation*}
	\frac{1}{2}\del_t\int_{\mathbb{S}^2}\big( |\del_t \Phi|^2 + r^{-2}| \nablas \Phi|^2 + |\del_r\Phi|^2\big) d^2 \mu
	- \del_r\int_{\mathbb{S}^2}\del_t\Phi\del_r\Phi d^2 \mu
	= - r\int_{\mathbb{S}^2}\del_t\Phi\Box\phi d^2 \mu.
	\end{equation*}
Using $\partial_t=\partial_\tau$ and $\del_r = \delb_r - \hhp \del_{\tau}$, we rewrite the above formula as follows:
	\begin{align}\label{eq1-06-19-2022}
	&\frac{1}{2}\del_\tau\int_{\mathbb{S}^2}\Big[\big( |\del_t \Phi|^2 + r^{-2}| \nablas \Phi|^2 + |\del_r\Phi|^2\big) + 2\hhp \del_t\Phi\del_r\Phi \Big]d^2 \mu\\
	&- \delb_r\int_{\mathbb{S}^2}\del_t\Phi\del_r\Phi d^2 \mu
	\notag= - r\int_{\mathbb{S}^2}\del_t\Phi\Box\phi d^2 \mu.
	\end{align}
	Then by integrating \eqref{eq1-06-19-2022} in the region
	$\Dcal^{r_{1},r_{2}}_{\tau_1,\tau_2}$ and applying the Stokes' formula on $\Dcal^{r_{1},r_{2}}_{\tau_1,\tau_2}$ with respect to the Euclidean metric, and letting $r_{1}\to 0^+$, we have
	\begin{equation}\label{eq:SE}
	\aligned
	&\SEnergy{0}{\tau_2, r_{2}}{\Phi} 
	+ \lim_{r_{1}\to 0^+} \int_{\tau_1}^{\tau_2}\int_{\mathbb{S}^2} \del_t\Phi\del_r\Phi d^2\mu\Big|_{r=r_{1}}d\tau
	\\
	&+\frac{1}{2}\int_{\tau_1}^{\tau_2}\int_{\mathbb{S}^2}
	\Big(\frac{1}{1+\hhp}((\del_t\Phi-\del_r\Phi)^2 + r^{-2}| \nablas \Phi|^2) \Big) d^2\mu\Big|_{t+r = r_{2}}d\tau\\
	&=\SEnergy{0}{\tau_1,r_{2}}{\Phi}  - \lim_{r_{1}\to 0^+}\int_{\Dcal^{r_{1},r_{2}}_{\tau_1,\tau_2}}\int_{\mathbb{S}^2}r\del_t\Phi\Box\phi d^4\mu,
	\endaligned	
	\end{equation}
	where we define
	$$
	\aligned
	\SEnergy{0}{\tau, r_{2}}{\Phi} :=& \frac{1}{2}\int_{\Sigma_\tau\cap\{t+r \leq r_{2}\}}\Big(|\delb_r\Phi|^2 + (1-(\hhp)^2)|\del_{\tau}\Phi|^2 + r^{-2}| \nablas\Phi|^2 \Big)d^3\mu.
	\endaligned
	$$
	Note that, from the definition of $\Phi$, we have 
	\begin{equation*}
	\lim_{r_{1}\to 0^+}\del_{t}\Phi \del_{r}\Phi= \lim_{r_{1}\to 0^+} (r^2 \del_{t}\phi \del_r\phi + r\phi\del_t\phi )=0.
	\end{equation*}
	Hence, we let $r_{2}\to \infty$ in~\eqref{eq:SE} and can obtain
	\begin{equation}\label{eq3-06-06-2022}
	\aligned
	&\lim_{r_{2}\to \infty}\SEnergy{0}{\tau_2,r_{2}}{\Phi}  
	+ \frac{1}{2}\int_{\IIonetwo}
	\big((\del_t\Phi-\del_r\Phi)^2 + r^{-2}| \nablas \Phi|^2\big) d\tau d^2\mu \\
	&\le \lim_{r_{2}\to \infty}\SEnergy{0}{\tau_1,r_{2}}{\Phi}  -\int_{\MMonetwo}\del_t\Phi\times r\Box \psi d^4\mu.
	\endaligned
	\end{equation}
Note also that, from Hardy's inequality \eqref{eq:HardyIneqLHS} and $(r^{-1}|\Phi|^2)\vert_{r=0}=(r|\phi|^2) \vert_{r=0}=0$,
\begin{equation*}
\int_{\Sigma_{\tau}} r^{-2}|\Phi|^2 d^3\mu \lesssim \int_{\Sigma_{\tau}} \Big(|\delb_r\Phi|^2 + (1-(\hhp)^2)|\del_{\tau}\Phi|^2 + r^{-2}| \nablas\Phi|^2 \Big)d^3\mu.
\end{equation*}
Combining the above inequality with~\eqref{eq3-06-06-2022}, we complete the proof of~\eqref{eq2-06-06-2022}.
\end{proof}

Second, we review the following lemma on the basic Morawetz estimate for~\eqref{eq1-07-06-2022}.

\begin{lemma}[Basic Morawetz estimate]
\label{lem:Morawetz:psi:M}
	Let $\delta\in (0,\frac{1}{2})$. Then 
	\begin{equation}
	\label{eq:Morawetz:psi:M}
	\aligned
	&\int_{\tau_1}^{\tau_2}|{\phi(\tau, 0)}|^2d\tau 
	+ \int_{\MMonetwo}\bigg( 
	\delta\frac{|\del_r\Phi|^2+|\del_t\Phi|^2+|r^{-1}\Phi|^2}{(1+r)^{1+\delta}}
	+\frac{|r^{-1}\nablas \Phi|^2  }{r}\bigg)
	d^4\mu
	\\
	\lesssim &   \SEnergy{0}{\tau_1}{\Phi} 
+\bigg|\int_{\MMonetwo}\del_t\Phi\cdot r\Box \phi d^4\mu \bigg|
+\bigg|\int_{\MMonetwo}\big(2 - (1+r)^{-\delta}\big)\del_r\Phi\cdot  r\Box \phi d^4\mu \bigg|.
	\endaligned
\end{equation}
\end{lemma}

\begin{proof}
	Fix $\chi_0 (r) = 2 -  (1+r)^{-\delta}$ with $\delta\in (0,\frac{1}{2})$, and one finds
	\begin{equation}\label{def:chi0}
	\begin{split}
	&\chi_0(0)=1, \qquad \chi_0(\infty) = 2,
	\quad \chi_0'(r) =\delta(1+r)^{-1-\delta},
	\\
	&\chi_0(r) r^{-1} - \frac{1}{2}\chi_0'(r) = \big(2-(1+r)^{-\delta}\big)r^{-1} - \frac{1}{2}\delta(1+r)^{-1-\delta}\geq \frac{1}{2r}.
	\end{split}
	\end{equation} 
	
	Recall from \eqref{eq1-07-06-2022} that
	$$
	-r\Box \phi = (\del_t\del_t - \del_r\del_r - r^{-2}\gs^{\mu\nu}\nablas_{\mu}\nablas_{\nu})\Phi. 
	$$
	Thus, by multiplying both sides by $\chi_0(r)\del_r\Phi$, with $\chi_0(r)$ being a radial weight function, and using the Leibniz's rule, we deduce that
\begin{align*}
	-\chi_0(r)\del_r\Phi\cdot r\Box \phi 
	=& \del_t\big(\chi_0(r)\del_r\Phi\del_t\Phi\big) 
-\gs^{\mu\nu}\nablas_{\mu}(\chi_0(r)r^{-2}\del_r\Phi\nablas_{\nu}\Phi)\notag\\
	&	- \frac{1}{2}\del_r\big(\chi_0(r) (|\del_r\Phi|^2 - |r^{-1}\nablas\Phi|^2+|\del_t\Phi|^2)\big)
	\notag\\
	&
	+ \frac{1}{2}\chi_0'(r)(|\del_r\Phi|^2 - |r^{-1}\nablas\Phi|^2+|\del_t\Phi|^2) + r^{-3}\chi_0(r)|\nablas\Phi|^2.
	\end{align*}
	We integrate the above equation on $\{r = \text{constant}\}$ with respect to $d^2\mu$ and, within the parameterization $(\tau,r)$, the above identity is written as
	\begin{equation}\label{eq2-21-06-2022}
		\aligned
		&-\int_{\mathbb{S}^2} \chi_0(r)\del_r\Phi \cdot r\Box \phi d^2\mu
		\\
		=& \frac{1}{2}\del_\tau\int_{\mathbb{S}^2}
		\chi_0(r)\Big(\hhp(r)(|\del_r\Phi|^2 - |r^{-1}\nablas\Phi|^2 + |\del_t\Phi|^2) + 2\del_r\Phi\del_t\Phi\Big) d^2\mu
		\\
		&-\frac{1}{2}\delb_r\int_{\mathbb{S}^2}\chi_0(r)\Big(|\del_r\Phi|^2 - |r^{-1}\nablas \Phi|^2 + |\del_t\Phi|^2\Big) d^2\mu
		\\
		&+\frac{1}{2}\chi_0'(r)\int_{\mathbb{S}^2}(|\del_r\Phi|^2 - |r^{-1}\nablas\Phi|^2+|\del_t\Phi|^2) d^2\mu
		+ \chi_0(r)r^{-3}\int_{\mathbb{S}^2}|\nablas \Phi|^2 d^2\mu.
		\endaligned
	\end{equation}

		Then by integrating \eqref{eq2-21-06-2022} in the region $\Dcal^{r_{1},r_{2}}_{\tau_1,\tau_2}$, we obtain 
	\begin{equation*}
		\aligned
		&-\int_{\Dcal^{r_{1},r_{2}}_{\tau_1,\tau_2}}\int_{\mathbb{S}^2}
		\chi_0(r) \del_r\Phi \cdot r\Box \phi d^{4}\mu
		\\
		=& \frac{1}{2}\int_{r_{1}}^{R(\tau,r_{2})}\int_{\mathbb{S}^2} 
		\chi_0(r)\Big(\hhp(r)(|\del_r\Phi|^2 - |r^{-1}\nablas\Phi|^2 + |\del_t\Phi|^2) + 2\del_r\Phi\del_t\Phi\Big)   
		 { d^3\mu \bigg|_{\tau_1}^{\tau_2} }
		\\
		&+\frac{1}{2}\int_{\tau_1}^{\tau_2}\int_{\mathbb{S}^2}
		\frac{\chi_0(r)}{1+\hhp}\Big(|r^{-1}\nablas\Phi|^2 - (\del_t\Phi - \del_r\Phi)^2\Big)\bigg|_{t+r = r_{2}} d^2\mu d\tau
		\\
		&+\frac{1}{2}\int_{\tau_1}^{\tau_2}\int_{\mathbb{S}^2}\chi_0(r)\Big(|\del_r\Phi|^2 - |r^{-1}\nablas \Phi|^2 + |\del_t\Phi|^2\Big)\bigg|_{r=r_{1}} d^2\mu d\tau
		\\
		&
		+\int_{\Dcal^{r_{1},r_{2}}_{\tau_1,\tau_2}}
		\int_{\mathbb{S}^2}\bigg(\frac{1}{2}\chi_0'(r)\big(|\del_r\Phi|^2 - |r^{-1}\nablas\Phi|^2+|\del_t\Phi|^2\big) +\chi_0(r)  r^{-3}|\nablas \Phi|^2\bigg)  d^4\mu.
		\endaligned
	\end{equation*}
	For the third last line of the above equation, if we take $r_{1}\to 0^+$, then
	$$
	\lim_{r_{1}\to 0^+} \frac{1}{2}\int_{\tau_1}^{\tau_2}\int_{\mathbb{S}^2}\chi_0(r)\Big(|\del_r\Phi|^2 - |r^{-1}\nablas \Phi|^2 + |\del_t\Phi|^2\Big) d^2\mu  \bigg|_{r=r_{1}}d\tau
	=2\pi\int_{\tau_1}^{\tau_2}|\phi(\tau, 0)|^2d\tau.
	$$
Therefore, by letting $(r_{1},r_{2})\to (0^{+},\infty)$ and using~\eqref{def:chi0}, we arrive at
\begin{equation*}
\aligned
&-\int_{\MMonetwo}\chi_0(r) \del_r\Phi \cdot r\Box \phi d^4\mu
\\
=&\frac{1}{2}\int_{0}^{\infty}\int_{\mathbb{S}^2} 
		\chi_0(r)\Big(\hhp(|\del_r\Phi|^2 - |r^{-1}\nablas\Phi|^2 + |\del_t\Phi|^2) + 2\del_r\Phi\del_t\Phi\Big)   
		  { d^3\mu \bigg|_{\tau_1}^{\tau_2} }\\
&+\frac{1}{2}\int_{\IIonetwo}
\big(|r^{-1}\nablas\Phi|^2 - (\del_t\Phi - \del_r\Phi)^2\big) d^2\mu d\tau
+2\pi\int_{\tau_1}^{\tau_2}|\phi(\tau, 0)|^2d\tau
\\
&+\int_{\MMonetwo}
\bigg(\frac{1}{2}\chi_0'(r)\big(|\del_r\Phi|^2 - |r^{-1}\nablas\Phi|^2+|\del_t\Phi|^2\big) +\chi_0(r)  r^{-3}|\nablas \Phi|^2\bigg)
d^4\mu .
\endaligned
\end{equation*}
It follows from~\eqref{def:chi0} that
\begin{align}\label{eq31-01-07-2022-M}
&-\int_{\MMonetwo}\big(2 - (1+r)^{-\delta}\big)\del_r\Phi\cdot r\Box \phi \, d^4\mu
\notag+\int_{\IIonetwo}|U\Phi|^2 d^2\mu d\tau
\notag\\
&-\frac{1}{2}\int_{0}^{\infty}\int_{\mathbb{S}^2} 
		\chi_0(r)\Big(\hhp(|\del_r\Phi|^2 - |r^{-1}\nablas\Phi|^2 + |\del_t\Phi|^2) + 2\del_r\Phi\del_t\Phi\Big)   
		  d^3\mu \bigg|_{\tau_1}^{\tau_2}
\notag\\
\geq & \frac{1}{2}\int_{\IIonetwo}|r^{-1}\nablas\Phi|^2 d^2\mu d\tau 
\notag +2\pi\int_{\tau_1}^{\tau_2}|\phi(\tau, 0)|^2d\tau\\
&  
+\int_{\MMonetwo}
\bigg(\frac{1}{2r}|r^{-1}\nablas\Phi|^2  + \frac{1}{2}\delta \frac{|\del_r\Phi|^2+|\del_t\Phi|^2}{(1+r)^{1+\delta}} \bigg) d^4\mu.
\end{align}
Note that by the Hardy's inequality \eqref{eq:HardyIneqLHS}, one finds that the last line bounds the integral $c_0\int_{\MMonetwo}\delta \frac{|r^{-1}\Phi|^2}{(1+r)^{1+\delta}}d^4\mu$, $c_0>0$ being a universal constant.

On the other hand, in view that
	$$
	\hhp(|\del_r\Phi|^2 + |\del_t\Phi|^2) + 2\del_r\Phi\del_t\Phi
	=\hhp|\delb_r\Phi|^2 
	+2(1-(\hhp)^2)\delb_r\Phi\del_\tau\Phi 
	+ \hhp ((\hhp)^2-{1})|\del_\tau\Phi|^2 
	,
	$$
and since $1\leq \chi_0\leq 2$, the second line of            
\eqref{eq31-01-07-2022-M} is bounded by 
$$
C_0( \SEnergy{0}{\tau_2}{\Phi}  +  \SEnergy{0}{\tau_1}{\Phi} ).
$$
The third line of \eqref{eq31-01-07-2022-M}, by recalling \eqref{eq2-06-06-2022}, is bounded by
$$
C_0\bigg( \SEnergy{0}{\tau_1}{\Phi} +\bigg|\int_{\MMonetwo}\del_t\Phi\cdot r\Box \phi d^4\mu \bigg|\bigg),
$$
with $C_0>0$ a finite universal constant. 
Recall that by \eqref{eq2-06-06-2022} again, 
$$
 \SEnergy{0}{\tau_2}{\Phi} \lesssim  \SEnergy{0}{\tau_1}{\Phi} +\bigg|\int_{\MMonetwo}\del_t\Phi\cdot r\Box \phi d^4\mu \bigg|. 
$$
Consequently, the left-hand side of \eqref{eq31-01-07-2022-M} is bounded by
$$
C_1\bigg( \SEnergy{0}{\tau_1}{\Phi} 
+\bigg|\int_{\MMonetwo}\del_t\Phi\cdot r\Box \phi d^4\mu \bigg|
+\bigg|\int_{\MMonetwo}\big(2 - (1+r)^{-\delta}\big)\del_r\Phi\cdot  r\Box \phi d^4\mu \bigg|\bigg),
$$
with $C_1>0$ again being a finite universal constant.  All these discussions together 
 thus {yield} the desired estimate \eqref{eq:Morawetz:psi:M}.
\end{proof}

%%%%%%%%%%%%%%%%%%
\subsection{High-order energy--Morawetz estimates for wave equations}
\label{sect:HighEMEstis:M}
%%%%%%%%%%%%%%%%%%

The above energy estimate \eqref{eq2-06-06-2022} and Morawetz estimate \eqref{eq:Morawetz:psi:M} can be combined to yield an energy--Morawetz estimate. Moreover, we are able to show a high-order regularity version of such an energy--Morawetz estimate for~\eqref{eq1-07-06-2022}.

\begin{proposition}[High-order energy--Morawetz estimate for wave equations]\label{prop:HighEnerMora:psi:M}
	Let $k\in \NN$ and let $\delta\in (0,\frac{1}{2})$. Then
the following high-order energy--Morawetz estimate for $\phi$ holds:
	\begin{align}
	\label{eq:HighEnerMora:psi:M}
	&\SEnergy{k}{\tau_2}{\Phi}\notag + \sum_{|\alpha|\leq k}  \int_{\MMonetwo}\frac{|r^{-1}\nablas (r\del^{\alpha}\phi)|^2  }{r}
	d^4\mu
	\notag
	\\
	&
	+\delta \sum_{|\alpha|\leq k}  \int_{\MMonetwo}\bigg( 
 \frac{|\del_r(r\del^{\alpha}\phi)|^2+|\del_t(r\del^{\alpha}\phi)|^2+|r^{-1}(r\del^{\alpha}\phi)|^2}{(1+r)^{1+\delta}}
\bigg)
	d^4\mu
	\notag\\
	\lesssim & \SEnergy{k}{\tau_1}{\Phi} 
+ \sum_{|\alpha|\leq k} \bigg|\int_{\MMonetwo}\del_t(r\del^{\alpha}\phi)\cdot r\Box (\del^{\alpha}\phi) d^4\mu \bigg|\notag\\
&
+ \sum_{|\alpha|\leq k} \bigg|\int_{\MMonetwo}\big(2 - (1+r)^{-\delta}\big)\del_r(r\del^{\alpha}\phi)\cdot  r\Box (\del^{\alpha}\phi) d^4\mu \bigg|,
\end{align}
where the standard high-order energy\footnote{This is simply the high-order version of the first-order energy defined in \eqref{standardEner:zero:M}.} of a scalar function $\varphi$ on $\Sigma_{\tau}$, $\tau\geq \tau_0$, is 
\begin{align*}
\SEnergy{k}{\tau}{\varphi} :=&\sum_{|\alpha|\leq k} \int_{\Sigma_\tau}(1-(\hhp)^2) |\del_{\tau}(r\del^{\alpha}(r^{-1}\varphi))|^2 d^3\mu\\
&+\sum_{|\alpha|\leq k} \int_{\Sigma_\tau}\Big(|\delb_r(r\del^{\alpha}(r^{-1}\varphi))|^2 + \frac{| \nablas(r\del^{\alpha}(r^{-1}\varphi))|^2}{r^{2}} \Big)d^3\mu.
\end{align*}
\end{proposition}

\begin{remark}
For this standard high-order energy, we use $V=(1-\hhp)\del_{\tau}+\delb_r $ and $U=(1+\hhp)\del_{\tau}-\delb_r $ to find
\begin{equation}
\label{standardEner:high:M}
\begin{aligned}
\SEnergy{k}{\tau}{\varphi} 
=&\sum_{|\alpha|\leq k} \int_{\Sigma_\tau}\Big(\frac{1+\hhp}{2}|V(r\del^{\alpha}(r^{-1}\varphi))|^2 +\frac{1-\hhp}{2} |U(r\del^{\alpha}(r^{-1}\varphi))|^2 \Big)d^3\mu\\
+&\sum_{|\alpha|\leq k} \int_{\Sigma_\tau} \frac{| \nablas(r\del^{\alpha}(r^{-1}\varphi))|^2}{r^{2}} d^3\mu\\
\simeq &\sum_{|\alpha|\leq k} \int_{\Sigma_\tau}\Big(|V(r\del^{\alpha}(r^{-1}\varphi))|^2 +\la r\ra^{-1-\eta} |U(r\del^{\alpha}(r^{-1}\varphi))|^2  \Big)d^3\mu\\
+&\sum_{|\alpha|\leq k} \int_{\Sigma_\tau}\frac{| \nablas(r\del^{\alpha}(r^{-1}\varphi))|^2}{r^{2}} d^3\mu.
\end{aligned}
\end{equation}
\end{remark}

\begin{proof}[Proof of Proposition~\ref{prop:HighEnerMora:psi:M}]
The $k=0$ case follows by combining the proven energy estimate \eqref{eq2-06-06-2022} and Morawetz estimate \eqref{eq:Morawetz:psi:M}. The estimate \eqref{eq:HighEnerMora:psi:M} for general $k\in\NN$ holds since $[\del^{\alpha}, \Box]=0$.
\end{proof}

%%%%%%%%%%%%%%%%%%
\subsection{$r^p$ estimates near infinity for a general inhomogeneous wave equation}
\label{subsect:rpnearinf:generalwave}
%%%%%%%%%%%%%%%%%%

{The following is a refined $r^p$ argument for a general inhomogeneous wave equation, following the ideas originally given in \cite{DafermosRodnianski:rp} and some ingredients presented in \cite{andersson2019stability}.}  Essentially one uses the vector field method with the vector {fields} $(1+ r^{-\delta/2})U  +r^p V $ with $\delta>0$ small and $p\in [\delta,2-\delta]$.

\begin{lemma}[$r^p$ estimates near infinity for inhomogeneous wave equations]
\label{lem:rpForInhoWave:M}
Let $\delta \in (0,\frac12)$ be a small constant. Assume the scalar function {$\varphi$} satisfies an inhomogeneous wave equation
\begin{align}
\label{InhoWave:rplemma:M}
-UV\varphi + r^{-2}\Deltas\varphi- b_V r^{-1}V\varphi - b_0 r^{-2} \varphi=\vartheta.
\end{align} 
Assume $b_V$ and $b_0$ are real, smooth functions of $r$ such that
\begin{enumerate}[label=\arabic*)]
\item there exists a constant $b_{V,0}\geq 0$ such that $b_V= b_{V,0} + O(r^{-1})$  holds for large $r$,
\item\label{item:2:rplemma:InhoWave:M} and there exists a constant $b_{0, 0}\geq 0$ such that $b_0=b_{0,0}  + O(r^{-1})$ holds for large $r$.
\end{enumerate}
Then, for any $p\in [\delta, 2-\delta]$ and $k\in \mathbb{N}$, there exist constants $\bar{R}_0=\bar{R}_0(b_0, b_V, \delta, k)>10$ and $C_0=C_0(b_0, b_V, \delta, k)>0$ such that for all $R_0\geq \bar{R}_0$ and $\tau_2>\tau_1\geq \tau_0$, 
\begin{align}
\label{generalrp:M}
&\Wnorm{rV\varphi}{k}{p-2}{\Sigma_{\tau_2}^{\geq R_0}}^2
+\Wnorm{U\varphi}{k}{-1-\eta}{\Sigma_{\tau_2}^{\geq R_0}}^2
+\Wnorm{\nablas^{\leq 1}\varphi}{k}{\peta}{\Sigma_{\tau_2}^{\geq R_0}}^2
+\Flux{k}{\tau_1,\tau_2}{p}{\varphi}\notag\\
&
+\Wnorm{\varphi}{k+1}{p-3}{\MMonetwo^{\geq R_0}}^2
+\Wnorm{U\varphi}{k}{-1-\delta/2}{\MMonetwo^{\geq R_0}}^2\notag\\
\leq {}& C_0\bigg(
\Wnorm{rV\varphi}{k}{p-2}{\Sigma_{\tau_1}^{\geq R_0}}^2
+\Wnorm{U\varphi}{k}{-1-\eta}{\Sigma_{\tau_1}^{\geq R_0}}^2
+\Wnorm{\nablas^{\leq 1}\varphi}{k}{\peta}{\Sigma_{\tau_1}^{\geq R_0}}^2
\notag\\
&\qquad +\Wnorm{\varphi}{k+1}{0}{\MMonetwo^{R_0-1, R_0}}^2
+\sum_{\tau\in\{\tau_1,\tau_2\}}\Wnorm{\varphi}{k+1}{0}{\Sigma_{\tau}^{R_0-1, R_0}}^2\notag\\
&\qquad +\sum_{|\alpha|\leq k} \bigg|\int_{\MMonetwo^{\geq R_0-1}} \chi^2(r) r^{-1} \DDb^{\alpha}(r\vartheta) \times r^p V \DDb^{\alpha} \varphi d^4\mu\bigg|\notag\\
&\qquad +\sum_{|\alpha|\leq k} \bigg|\int_{\MMonetwo^{\geq R_0-1}} \chi^2(r)r^{-1} \DDb^{\alpha}(r\vartheta) \times (1+r^{-\delta/2})U \DDb^{\alpha} \varphi d^4\mu\bigg|\bigg),
\end{align}
where $\chi(r):=\chi_1(r-R_0)$ with $\chi_1(x)$ being a standard smooth cutoff function that equals $1$ for $x\geq 0$ and vanishes for $x\leq -1$, {and the constant $\peta=\max\{-2, p-\eta-3\}$ and the term $\Flux{k}{\tau_1,\tau_2}{p}{\varphi}$ are defined as in Section \ref{subsect:energies}}.

Further, if we replace the above assumption \ref{item:2:rplemma:InhoWave:M} by the following assumption:

\quad 2') and there exists a finite constant $b_{0,0}$ such that $b_0=b_{0,0}+O(r^{-1})$ holds for large $r$ and such that for the considered scalar $\varphi$,
\begin{align}
\int_{\mathbb{S}^2} ( |\nablas \varphi|^2 + b_{0,0} |\varphi|^2 )d^2\mu \geq 0,
\end{align}
then the above conclusion, in particular, the estimate \eqref{generalrp:M}, still holds true.
\end{lemma}

\begin{proof}
Note that though the various norms in \eqref{generalrp:M} are defined as in \eqref{defs:generalPTWandEnergynorms:M} using the operators in $\DDb$ and $\del$, they are equivalent to the ones defined with respect only to the set $\DDb$ when $r$ is uniformly away from $0$. That is, for a spacetime region $\mathcal{D}$ uniformly away from the axis $r=0$, 
\begin{align}
\Wnorm{\varphi}{k}{p}{\mathcal{D}}\simeq&\bigg(\int_{\mathcal{D}} \la r\ra^p \sum_{|\alpha|\leq k} \vert \DDb^{\alpha}\varphi\vert d^4\mu \bigg)^{\frac12}{,}
\end{align}
and similarly for $\Wnorm{\varphi}{k}{p}{\Sigma}$ with $\Sigma$ being a hypersurface uniformly away from $r=0$. Consequently, in the remainder of this proof, we shall only prove inequality \eqref{generalrp:M} with the norms replaced by the above equivalent ones that are defined using only the operator set $\DDb$.
 
To begin with, we muitiply the wave equation \eqref{InhoWave:rplemma:M} by $-2 \chi^2 ((1+r^{-\delta/2})U\varphi + r^p V\varphi)$ and integrate the resulting equation over $\MMonetwo^{\geq R_0-1}$.
In view of the following calculations
\begin{align*}
(- UV\varphi)(-2 \chi^2 (1+r^{-\delta/2})U\varphi)
&={}V(\chi^2 (1+r^{-\delta/2}) \abs{U\varphi}^2)
-\partial_r (\chi^2 (1+r^{-\delta/2}))\abs{U\varphi}^2,\\
r^{-2}(\Deltas\varphi-b_0\varphi)(-2 \chi^2 (1+r^{-\delta/2})U\varphi)
&=_{\mathbb{S}^2}{}U\big(\chi^2 r^{-2} (1+r^{-\delta/2})(\abs{\nablas\varphi}^2+b_0\abs{\varphi}^2)\big)\notag\\
&
\qquad +\partial_r (\chi^2r^{-2}(1+r^{-\delta/2}))\abs{\nablas\varphi}^2 \notag\\
&\qquad +\partial_r (b_0\chi^2 r^{-2}(1+r^{-\delta/2}))\abs{\varphi}^2{,}
\end{align*}
and
\begin{align*}
(- UV\varphi)(-2\chi^2 r^{p}V{\varphi})&={}U(\chi^2 r^p \abs{V\varphi}^2)+\partial_r (\chi^2 r^p)\abs{V\varphi}^2,\\
r^{-2}(\Deltas\varphi-b_0\varphi)(-2\chi^2 r^{p}V{\varphi})
&=_{\mathbb{S}^2}{}V\big(\chi^2 r^{p-2}(\abs{\nablas\varphi}^2+b_0\abs{\varphi}^2)\big)\notag\\
&
\qquad -\big(\partial_r (\chi^2r^{p-2})\abs{\nablas\varphi}^2 +\partial_r (b_0\chi^2 r^{p-2})\abs{\varphi}^2\big),\\
(-b_V r^{-1}V\varphi)(-2\chi^2 r^{p}V{\varphi})&={}2\chi^2 b_Vr^{p-1}\abs{V\varphi}^2,
\end{align*}
where $=_{\mathbb{S}^2}$ means the two sides are equal after an integration over $\mathbb{S}^2$,
we arrive at
\begin{align}
\label{eq:rplemma:0to2:general:mul}
&\int_{\MMonetwo^{\geq R_0-1}}\Big[V\big(\chi^2 r^{p-2}(\abs{\nablas\varphi}^2
+b_{0}\abs{\varphi}^2)+\chi^2 (1+r^{-\delta/2}) \abs{U\varphi}^2\big)\notag\\
&\qquad\qquad \quad 
+U\big(\chi^2r^p\abs{V\varphi}^2+\chi^2 r^{-2} (1+r^{-\delta/2})(\abs{\nablas\varphi}^2+b_0\abs{\varphi}^2)\big)\Big]d^4\mu
\notag\\
&
+\int_{\MMonetwo^{\geq R_0-1}}\Big[(\partial_r(\chi^2r^p)
+2\chi^2 b_{V}r^{p-1})\abs{V\varphi}^2
-\partial_r (\chi^2 (1+r^{-\delta/2}))\abs{U\varphi}^2
\notag\\
&\qquad\qquad\qquad 
-\partial_r(\chi ^2r^{-2} (r^p-1-r^{-\delta/2}))
\abs{\nablas\varphi}^2
-\partial_r(\chi ^2r^{-2}b_0 (r^p-1-r^{-\delta/2}))
\abs{\varphi}^2\Big]d^4\mu\notag\\
&
={}\int_{\MMonetwo^{\geq R_0-1}} \Big[-2\chi^2 b_V r^{-1}(1+r^{-\frac{\delta}{2}})V\varphi U\varphi-2 \chi^2 ((1+r^{-\frac{\delta}{2}})U\varphi + r^p V\varphi)\vartheta\Big] d^4\mu.
\end{align}
One finds that the first two lines contribute the flux terms by divergence theorem, and the remaining lines are spacetime integrals.
By assumption, for any $p\in [\delta, 2-\delta]$ and suitably large $\bar{R}_0$, the coefficients of each term  in the third and fourth lines satisfy in the region $r\geq {R}_0\geq \bar{R}_0$ that
\begin{align*}
\partial_r(\chi^2r^p)
+2\chi^2 b_{V}r^{p-1}\geq &\frac{1}{2}(p+ b_{V,0})r^{p-1},\\
-\partial_r (\chi^2 (1+r^{-\frac{\delta}{2}}))\geq & \frac{\delta}{4} r^{-1-\frac{\delta}{2}},\\
-\partial_r(\chi ^2r^{-2} (r^p-1-r^{-\frac{\delta}{2}}))\geq &\frac{1}{2} (2-p) r^{p-3},\\
-\partial_r(\chi ^2r^{-2}b_0 (r^p-1-r^{-\frac{\delta}{2}}))\geq & \frac{1}{2}(2-p)b_{0,0} r^{p-3} -Cr^{p-4}.
\end{align*}
Here, we have used $\chi=1$ for $r\geq {R}_0$. 
Therefore, we deduce from inequality \eqref{eq:rplemma:0to2:general:mul} the following inequality for  any $R_0\geq \bar{R}_0$ with $\bar{R}_0$ suitably large:
\begin{align}
\label{generalrp:commrV:once:prel:M}
&\Wnorm{rV\varphi}{0}{p-2}{\Sigma_{\tau_2}^{\geq R_0}}^2
+\Wnorm{U\varphi}{0}{-1-\eta}{\Sigma_{\tau_2}^{\geq R_0}}^2
+\Wnorm{\nablas\varphi}{0}{\peta}{\Sigma_{\tau_2}^{\geq R_0}}^2
+\Flux{0}{\tau_1,\tau_2}{p}{\varphi}
\notag\\
&
+\sum_{\mathbb{X}\in \DDb}\int_{\MMonetwo^{\geq R_0}} r^{p-3} \abs{\mathbb{X}\varphi}^2 d^4\mu
+\Wnorm{U\varphi}{0}{-1-\delta/2}{\MMonetwo^{\geq R_0}}^2\notag\\
\leq {}& C_0\bigg(
\Wnorm{rV\varphi}{0}{p-2}{\Sigma_{\tau_1}^{\geq R_0}}^2
+\Wnorm{U\varphi}{0}{-1-\eta}{\Sigma_{\tau_1}^{\geq R_0}}^2
+\Wnorm{\nablas^{\leq 1}\varphi}{0}{\peta}{\Sigma_{\tau_1}^{\geq R_0}}^2\notag\\
&
\qquad +\Wnorm{\varphi}{1}{0}{\MMonetwo^{R_0-1, R_0}}^2
+\sum_{\tau\in\{\tau_1,\tau_2\}}\Wnorm{\varphi}{1}{0}{\Sigma_{\tau}^{R_0-1, R_0}}^2\notag\\
&\qquad +\bigg|\int_{\MMonetwo^{\geq R_0-1}} \chi^2(r) \vartheta \times r^p V  \varphi d^4\mu\bigg|
+\bigg|\int_{\MMonetwo^{\geq R_0-1}} \chi^2(r) \vartheta \times (1+r^{-\frac{\delta}{2}})U\varphi d^4\mu\bigg|\notag\\
&\qquad +\int_{\MMonetwo^{\geq R_0}} r^{p-4} \abs{\varphi}^2 d^4\mu 
+\bigg|\int_{\MMonetwo^{\geq R_0-1}} 2\chi^2 b_V r^{-1}(1+r^{-\frac{\delta}{2}})V\varphi U\varphi d^4\mu \bigg|
\bigg).
\end{align}
\textit{We note that a derivation of the above inequality relies essentially on the bound $p<2$, which makes the coefficient of the term $\Wnorm{\nablas\varphi}{0}{p-3}{\MM_{\tau_1,\tau_2}^{\geq R_0}}^2$ on the left-hand side positive. }

The last term in \eqref{generalrp:commrV:once:prel:M} is controlled via the Cauchy--Schwarz by $  \int_{\MMonetwo^{\geq R_0-1}}( \ep r^{-1-\frac{\delta}{2}}\abs{U\varphi}^2 +{\ep^{-1}} r^{-1+\frac{\delta}{2}} \abs{V\varphi}^2) d^4\mu$, whose integral over $\MMonetwo^{\geq R_0}$ is absorbed by the left-hand side of \eqref{generalrp:commrV:once:prel:M}  by first taking $\ep$ small and then taking $\bar{R}_0$ much larger than $\ep^{-2/\delta}$. On the other hand, in view that $\delb_r = O(1) V + O(r^{-1-\eta}) U$, we apply the Hardy inequality \eqref{eq:HardyIneqLHS} to obtain
\begin{subequations}
\label{eq:Hardyapplied:generalrp:M}
\begin{align}
\int_{\MMonetwo^{\geq R_0}} r^{p-3} \abs{\varphi}^2 d^4\mu 
\lesssim &\int_{\MMonetwo^{\geq R_0{-1}}} r^{p-3} (\abs{rV\varphi}^2 +\abs{r^{-\eta} U\varphi}^2)d^4\mu 
+\Wnorm{\varphi}{1}{0}{\MMonetwo^{R_0-1, R_0}}^2 \notag\\
\lesssim &\int_{\MMonetwo^{\geq R_0}} r^{p-3} (\abs{rV\varphi}^2 +\abs{r^{-\eta} U\varphi}^2)d^4\mu 
+\Wnorm{\varphi}{1}{0}{\MMonetwo^{R_0-1, R_0}}^2,\\
\int_{\Sigma_{\tau_2}^{\geq R_0}} r^{\peta} \abs{\varphi}^2 d^3\mu \lesssim &\int_{\Sigma_{\tau_2}^{\geq R_0{-1}}} r^{\peta} (\abs{rV\varphi}^2 +\abs{r^{-\eta} U\varphi}^2) d^3\mu
 +\Wnorm{\varphi}{1}{0}{\Sigma_{\tau_2}^{R_0-1, R_0}}^2\notag\\
 \lesssim &\Wnorm{rV\varphi}{0}{p-2}{\Sigma_{\tau_2}^{\geq R_0}}^2
+\Wnorm{U\varphi}{0}{-1-\eta}{\Sigma_{\tau_2}^{\geq R_0}}^2
 +\Wnorm{\varphi}{1}{0}{\Sigma_{\tau_2}^{R_0-1, R_0}}^2.
\end{align}
\end{subequations}
We can add the estimates \eqref{eq:Hardyapplied:generalrp:M} to inequality \eqref{generalrp:commrV:once:prel:M} and the first term in the last line of \eqref{generalrp:commrV:once:prel:M} is absorbed by taking $\bar{R}_0$ suitably large. Combining these discussions together, we thus prove $k=0$ case of inequality \eqref{generalrp:M}.

Note that in deriving the above estimate starting from equation \eqref{eq:rplemma:0to2:general:mul}, assumption \ref{item:2:rplemma:InhoWave:M} is a bit too much; in fact, it suffices to impose assumption 2') to obtain the above estimate (that is, the estimate \eqref{generalrp:M} for $k=0$) from equation \eqref{eq:rplemma:0to2:general:mul}. In the following discussions for the general $k>0$ case of inequality \eqref{generalrp:M}, one can go through the proof and find that assumption 2'), instead of assumption \ref{item:2:rplemma:InhoWave:M}, is sufficient.

We proceed to prove the general $k> 0$ case of inequality \eqref{generalrp:M}.  In the following steps,  inequality \eqref{generalrp:M} is proved for an increasingly  sequence of sets of operators until the estimate is proved for $\mathbb{X}=\DDb$. 

\textbf{Step 1: $\mathbb{X}=\DDb_1:= \{\del_t, \Omega_1,\Omega_2,\Omega_3\}$.} Since the operators  in $\DDb_1$ commute with the left-hand side of equation \eqref{InhoWave:rplemma:M}, the estimate \eqref{generalrp:M} trivially holds if we replace $\DDb$ by $\DDb_1$ and the norms that are defined via $\DDb$ by $\DDb_1$. 

\textbf{Step 2: $\mathbb{X}=\DDb$ but only first-order in $rV$.}
We commute wave equation \eqref{InhoWave:rplemma:M} with $rV$ and compute its commutator with the wave operator on the left-hand side by
\begin{align}
\label{comm:rVwithgeneralwave:M}
&[rV, - UV +r^{-2}\Deltas - b_V r^{-1} V - b_0 r^{-2}]\varphi\notag\\
=& - rU (VV\varphi) + UV(rV\varphi) + r\del_r(r^{-2}) \Deltas\varphi 
-r\del_r (b_V r^{-2})  rV\varphi - r\del_r ( b_0 r^{-2})\varphi\notag\\
=&UV\varphi - 2r^{-2} \Deltas\varphi -r^{-1} V(rV\varphi) +(r^{-2} -r\del_r (b_V r^{-2}))  rV\varphi - r\del_r ( b_0 r^{-2})\varphi\notag\\
=&- \vartheta - r^{-2} \Deltas\varphi-r^{-1} V(rV\varphi) +(r^{-2} -\del_r (b_V r^{-1}))  rV\varphi - \del_r ( b_0 r^{-1})\varphi ,
\end{align}
where we have used in the last step the wave equation \eqref{InhoWave:rplemma:M}  to expand $UV \varphi - r^{-2}\Deltas\varphi$. Then, by denoting $\varphi_{(1)}:=rV\varphi$, 
we obtain the following wave equation for $\varphi_{(1)}$:
\begin{align}
\label{eq:generalrp:wavecommrV:1:M}
&- UV \varphi_{(1)}+r^{-2}\Deltas\varphi_{(1)} - (b_V+1) r^{-1} V\varphi_{(1)} - b_0 r^{-2}\varphi_{(1)}\notag\\
=&V(r\vartheta)  -(r^{-2}-\del_r (b_V r^{-1}))rV\varphi+r^{-2}\Deltas\varphi + \del_r ( b_0 r^{-1})\varphi\doteq \tilde\vartheta_{(1)}.
\end{align}
This is in the form of equation \eqref{InhoWave:rplemma:M} with all the assumptions satisfied, hence applying the proven estimate \eqref{generalrp:M} with $k=0$ yields for any $p\in [\delta, 2-\delta]$,
\begin{align}
\label{generalrp:M:commrVonce}
&\Wnorm{rV\varphi_{(1)}}{0}{p-2}{\Sigma_{\tau_2}^{\geq R_0}}^2
+\Wnorm{U\varphi_{(1)}}{0}{-1-\eta}{\Sigma_{\tau_2}^{\geq R_0}}^2
+\Wnorm{\nablas^{\leq 1}\varphi_{(1)}}{0}{\peta}{\Sigma_{\tau_2}^{\geq R_0}}^2
+\Flux{0}{\tau_1,\tau_2}{p}{\varphi_{(1)}}
\notag\\
&
+\Wnorm{\varphi_{(1)}}{1}{p-3}{\MMonetwo^{\geq R_0}}^2
+\Wnorm{U\varphi_{(1)}}{0}{-1-\delta/2}{\MMonetwo^{\geq R_0}}^2\notag\\
\leq {}& C_0\bigg(
\Wnorm{rV\varphi_{(1)}}{0}{p-2}{\Sigma_{\tau_1}^{\geq R_0}}^2
+\Wnorm{U\varphi_{(1)}}{0}{-1-\eta}{\Sigma_{\tau_1}^{\geq R_0}}^2
+\Wnorm{\nablas^{\leq 1}\varphi_{(1)}}{0}{\peta}{\Sigma_{\tau_1}^{\geq R_0}}^2
\notag\\
&\qquad +\Wnorm{\varphi_{(1)}}{1}{0}{\MMonetwo^{R_0-1, R_0}}^2
+\sum_{\tau\in\{\tau_1,\tau_2\}}\Wnorm{\varphi_{(1)}}{1}{0}{\Sigma_{\tau}^{R_0-1, R_0}}^2\notag\\
&\qquad + \bigg|\int_{\MMonetwo^{\geq R_0-1}} \chi^2\tilde\vartheta_{(1)} \times r^p V  \varphi_{(1)} d^4\mu\bigg|\notag\\
&\qquad + \bigg|\int_{\MMonetwo^{\geq R_0-1}} \chi^2\tilde\vartheta_{(1)} \times (1+r^{-\delta/2})U \varphi_{(1)} d^4\mu\bigg|\bigg).
\end{align} 
Manifestly, to show \eqref{generalrp:M} in this case, it suffices to estimate the last two lines with $\tilde\vartheta_{(1)}$ replaced by $\tilde\vartheta_{(1)} - V(r\vartheta)= O(r^{-2} ) (rV)^{\leq 1} \varphi + r^{-2} \Deltas\varphi$. 

For the second last line of \eqref{generalrp:M:commrVonce} with such a replacement, it is bounded using the Cauchy--Schwarz inequality by
\begin{align*}
\frac{1}{\ep}\Wnorm{\varphi}{1}{p-3}{\MMonetwo^{\geq R_0-1}}^2 + \ep \Wnorm{\varphi_{(1)}}{1}{p-3}{\MMonetwo^{\geq R_0-1}}^2
+ \bigg|\int_{\MMonetwo^{\geq R_0-1}} \chi^2r^{p-2}\Deltas\varphi V  \varphi_{(1)} d^4\mu\bigg|,
\end{align*}
and for this last term, we can use integration by parts first in $\nablas$ and then in  $V$ to find it bounded by
\begin{align*}
&\int_{\II_{\tau_1,\tau_2}}r^{p-2}\bigg[\frac{1}{\ep} \abs{\nablas \varphi}^2 + \ep \abs{\nablas\varphi_{(1)}}^2\bigg]d\tau d^2\mu
+\sum_{\tau\in\{\tau_1,\tau_2\}}\int_{\Sigma_{\tau}^{\geq R_0-1}} r^{p-3-\eta} \bigg[\frac{1}{\ep} \abs{\nablas \varphi}^2 + \ep \abs{\nablas\varphi_{(1)}}^2\bigg] d^3\mu \notag\\
&+ \int_{\MMonetwo^{\geq R_0-1}} r^{p-3} \bigg[\frac{1}{\ep} \abs{\nablas \varphi}^2 + \ep \abs{\nablas\varphi_{(1)}}^2 
+\abs{rV\nablas \varphi}^2\bigg]  d^4\mu;
\end{align*}
hence the second last line of \eqref{generalrp:M:commrVonce} with the above replacement is controlled by  {
\begin{align}
\label{generalrp:CommrVonce:middle:2:M}
& \ep\bigg( \Wnorm{\varphi_{(1)}}{1}{p-3}{\MMonetwo^{\geq R_0}}^2
+\int_{\II_{\tau_1,\tau_2}}r^{p-2} \abs{\nablas\varphi_{(1)}}^2 d\tau d^2\mu
+\sum_{\tau\in\{\tau_1,\tau_2\}}\int_{\Sigma_{\tau}^{\geq R_0}} r^{\peta}  \abs{\nablas\varphi_{(1)}}^2d^3\mu\bigg) \notag\\
& \ep\bigg( \Wnorm{\varphi_{(1)}}{1}{p-3}{\MMonetwo^{R_0-1, R_0}}^2
+\sum_{\tau\in\{\tau_1,\tau_2\}}\int_{\Sigma_{\tau}^{ R_0-1, R_0}} r^{\peta}  \abs{\nablas\varphi_{(1)}}^2d^3\mu\bigg) \notag\\
&
+\frac{1}{\ep}\bigg(\Wnorm{\varphi}{1}{p-3}{\MMonetwo^{\geq R_0-1}}^2 
+\int_{\II_{\tau_1,\tau_2}}r^{p-2} \abs{\nablas \varphi}^2 d\tau d^2\mu
+\sum_{\tau\in\{\tau_1,\tau_2\}}\int_{\Sigma_{\tau}^{\geq R_0-1}} r^{\peta}  \abs{\nablas \varphi}^2 d^3\mu\bigg)\notag\\
&+ \int_{\MMonetwo^{\geq R_0-1}} r^{p-3} \abs{rV\nablas \varphi}^2  d^4\mu .
\end{align} 
By taking $\ep$ suitably small, the first line of \eqref{generalrp:CommrVonce:middle:2:M} is absorbed by the left-hand side of \eqref{generalrp:M:commrVonce},  and we can add in a suitably large multiple of the estimate proven in Step 1 to absorb the last three lines of \eqref{generalrp:CommrVonce:middle:2:M}. 
}

We estimate the last line of \eqref{generalrp:M:commrVonce} with the above replacement in a similar manner, and the only difference lies in that we use integration by parts first in $\nablas$ and then in $U$ to estimate the term $|\int_{\MMonetwo^{\geq R_0-1}} \chi^2r^{-2}(1+r^{-\delta/2})\Deltas\varphi U  \varphi_{(1)} d^4\mu |$.  In doing so, we find the last line of \eqref{generalrp:M:commrVonce} with the above replacement is controlled by
\begin{align}
\label{eq:generalrp:wave:commrv:error:2}
&\ep\bigg( \Wnorm{\varphi_{(1)}}{1}{\delta-3}{\MMonetwo^{\geq R_0}}^2
+\Wnorm{U\varphi_{(1)}}{0}{-1-\delta}{\MMonetwo^{\geq R_0}}^2
+\sum_{\tau\in\{\tau_1,\tau_2\}}\int_{\Sigma_{\tau}^{\geq R_0}} r^{\peta}  \abs{\nablas\varphi_{(1)}}^2d^3\mu\bigg) \notag\\
&\ep\bigg( \Wnorm{\varphi_{(1)}}{1}{\delta-3}{\MMonetwo^{R_0-1, R_0}}^2
+\Wnorm{U\varphi_{(1)}}{0}{-1-\delta}{\MMonetwo^{ R_0-1, R_0}}^2
+\sum_{\tau\in\{\tau_1,\tau_2\}}\int_{\Sigma_{\tau}^{ R_0-1, R_0}} r^{\peta}  \abs{\nablas\varphi_{(1)}}^2d^3\mu\bigg) \notag\\
&
+\frac{1}{\ep}\bigg(\Wnorm{\varphi}{1}{p-3}{\MMonetwo^{\geq R_0-1}}^2 
+\sum_{\tau\in\{\tau_1,\tau_2\}}\int_{\Sigma_{\tau}^{\geq R_0-1}} r^{\peta}  \abs{\nablas \varphi}^2 d^3\mu\bigg) \notag\\
&+ \int_{\MMonetwo^{\geq R_0-1}}  \big(r^{p-3}\abs{rV\nablas \varphi}^2 + r^{-1-\delta/2} \abs{U\nablas \varphi}^2\big) d^4\mu .
\end{align}
Again, in the above expression \eqref{eq:generalrp:wave:commrv:error:2}, by taking $\ep$ suitably small, the first line of \eqref{eq:generalrp:wave:commrv:error:2} is absorbed by the left-hand side of \eqref{generalrp:M:commrVonce},  and we can add in a suitably large multiple of the estimate proven in Step 1 to absorb the last {three} lines of \eqref{eq:generalrp:wave:commrv:error:2}. 

In summary, we obtain the following estimate for any $k\geq 0$:
\begin{align}
\label{generalrp:commrVonce:v2:M}
&\sum_{\abs{\alpha}\leq k-i}\sum_{i\leq 1}\Big(\Wnorm{rV\varphi_{\alpha,(i)}}{0}{p-2}{\Sigma_{\tau_2}^{\geq R_0}}^2
+\Wnorm{U\varphi_{\alpha,(i)}}{0}{-1-\eta}{\Sigma_{\tau_2}^{\geq R_0}}^2\notag\\
&\qquad\qquad\quad\,\,
+\Wnorm{\nablas^{\leq 1}\varphi_{\alpha,(i)}}{0}{\peta}{\Sigma_{\tau_2}^{\geq R_0}}^2
+\Flux{0}{\tau_1,\tau_2}{p}{\varphi_{\alpha,(i)}}\notag\\
&
\qquad\qquad\quad\,\,
 +\Wnorm{\varphi_{\alpha,(i)}}{1}{p-3}{\MMonetwo^{\geq R_0}}^2
+\Wnorm{U\varphi_{\alpha,(i)}}{0}{-1-\delta/2}{\MMonetwo^{\geq R_0}}^2\Big)\notag\\
\lesssim {}&\sum_{\abs{\alpha}\leq k-i}\sum_{i\leq 1}\bigg(
\Wnorm{rV\varphi_{\alpha,(i)}}{0}{p-2}{\Sigma_{\tau_1}^{\geq R_0}}^2
+\Wnorm{U\varphi_{\alpha,(i)}}{0}{-1-\eta}{\Sigma_{\tau_1}^{\geq R_0}}^2
+\Wnorm{\nablas^{\leq 1}\varphi_{\alpha,(i)}}{0}{\peta}{\Sigma_{\tau_1}^{\geq R_0}}^2\notag\\
&\qquad \qquad\qquad +\Wnorm{\varphi_{\alpha,(i)}}{1}{0}{\MMonetwo^{R_0-1, R_0}}^2
+\sum_{\tau\in\{\tau_1,\tau_2\}}\Wnorm{\varphi_{\alpha,(i)}}{1}{0}{\Sigma_{\tau}^{R_0-1, R_0}}^2\notag\\
&\qquad \qquad\qquad + \bigg|\int_{\MMonetwo^{\geq R_0-1}} \chi^2\vartheta_{\alpha,(i)} \times r^p V  \varphi_{\alpha,(i)} d^4\mu\bigg|\notag\\
&\qquad \qquad\qquad + \bigg|\int_{\MMonetwo^{\geq R_0-1}} \chi^2\vartheta_{\alpha,(i)} \times (1+r^{-\delta/2})U \varphi_{\alpha,(i)} d^4\mu\bigg|\bigg),
\end{align} 
where $\varphi_{\alpha, (i)}=\DDb_1^{\alpha} (rV)^i \varphi$ and $\vartheta_{\alpha, (i)}=\DDb_1^{\alpha} (r^{-1}(rV)^i (r\vartheta))$.

\textbf{Step 3: $\mathbb{X}=\DDb$ with arbitrary order in $rV$.}
In the end, we show the above estimate \eqref{generalrp:commrVonce:v2:M}  holds true if we replace $\sum_{i\leq 1}$ by $\sum_{i\leq k}$ on both sides, by induction on the order of the composition of $rV$. This then   proves the desired estimate \eqref{generalrp:M}. 

We assume the estimate \eqref{generalrp:commrVonce:v2:M}  holds  if we replace $\sum_{i\leq 1}$ by $\sum_{i\leq j-1}$ on both sides, with $j\leq k$, and based on this assumption, we prove the estimate  \eqref{generalrp:commrVonce:v2:M} with $\sum_{i\leq 1}$ by $\sum_{i\leq j}$ on both sides.
By commuting $\DDb_1^{\alpha}(rV)^j$, $\abs{\alpha}+j\leq k$,
we obtain the following equation analogous to equation \eqref{eq:generalrp:wavecommrV:1:M}:
\begin{align}
\label{eq:generalrp:wavecommrV:i:M}
&- UV \varphi_{\alpha,(j)}+r^{-2}\Deltas\varphi_{\alpha,(j)} - (b_V+j) r^{-1} V\varphi_{\alpha,(j)} - b_0 r^{-2}\varphi_{\alpha,(j)}\notag\\
=&\vartheta_{\alpha, (j)}
+j r^{-2}\Deltas\varphi_{\alpha, (j-1)}+\sum_{\abs{\beta}+n\leq \abs{\alpha}+j, 
n\leq j}\Big(O(r^{-2})\DDb_1^{\beta}(rV)^{n}\varphi\Big)
 \doteq \tilde\vartheta_{\alpha, (j)},
\end{align}
which is easily derived by the commutator relation \eqref{comm:rVwithgeneralwave:M}. Again, this equation is in the form of equation \eqref{InhoWave:rplemma:M} with all the assumptions satisfied, thus we obtain a similar estimate as \eqref{generalrp:M:commrVonce} but replacing $\sum_{i\leq 1}$ by $\sum_{i\leq j}$ on both sides. It suffices to bound the terms
\begin{align*}
\bigg|\int_{\MMonetwo^{\geq R_0-1}} \chi^2r^p (\tilde\vartheta_{\alpha, (j)}- \vartheta_{\alpha, (j)}) V  \varphi_{\alpha,(j)}d^4\mu\bigg|
+ \bigg|\int_{\MMonetwo^{\geq R_0-1}} \chi^2(1+r^{-\frac{\delta}{2}}) (\tilde\vartheta_{\alpha, (j)}- \vartheta_{\alpha, (j)})U \varphi_{\alpha,(j)} d^4\mu\bigg|.
\end{align*}
The remaining analysis in estimating these terms is exactly the same as in Step 2, and the assumed estimate \eqref{generalrp:commrVonce:v2:M} for $i\leq j-1$ are used therein; we omit these similar details. 
\end{proof}

 In the end,  we present the $r^p$ estimates near infinity for an inhomogeneous wave equation with radial symmetry in the following lemma. This lemma differs from the above Lemma \ref{lem:rpForInhoWave:M} for inhomogeneous wave equations without symmetry in the sense that the upper bound requirement $p<2$, which is crucial in obtaining a positive definite integral of the angular derivatives, is no longer needed. Consequently, we can extend the range of the parameter $p$ in the $r^p$ estimates, which are used to derive further energy decay estimates in later applications to $\ell=0$ mode of the scalar field in Section \ref{subsect:ptw:ell=0mode}.

\begin{lemma}[$r^p$ estimates for wave equation under radial symmetry]
\label{lem:rpInfinity-radial}
Let $\delta \in (0,\frac12)$ be a small constant. Let $b_V$ be a non-negative constant.
Assume the scalar function $\varphi$ with radial symmetry satisfies an inhomogeneous wave equation
\begin{align}
\label{eq:rpInfinity-radial}
UV\varphi + b_V r^{-1}V\varphi =\vartheta.
\end{align} 

Then, for any $p\geq \delta, k\in \mathbb{N}$, there exist constants ${R}_0={R}_0( b_V, \delta, k)>10$ and $C_0=C_0(b_V, \delta, k)>0$ such that for all $\tau_2>\tau_1\geq \tau_0$, 
\begin{align}
\label{eq:rpInfinity-radial-EE}
&\Wnorm{rV\varphi}{k}{p-2}{\Sigma_{\tau_2}^{\geq R_0}}^2
+\Wnorm{U\varphi}{k}{-1-\eta}{\Sigma_{\tau_2}^{\geq R_0}}^2
+\Wnorm{\varphi}{k}{\peta}{\Sigma_{\tau_2}^{\geq R_0}}^2  \notag
\\
&
+\Wnorm{V\varphi}{k}{p-3}{\MMonetwo^{\geq R_0}}^2
+\Wnorm{\varphi}{k}{\min\{p-3, -1-\delta/2\}}{\MMonetwo^{\geq R_0}}^2
+\Wnorm{U\varphi}{k}{-1-\delta/2}{\MMonetwo^{\geq R_0}}^2\notag\\
\leq {}& C_0\bigg(
\Wnorm{rV\varphi}{k}{p-2}{\Sigma_{\tau_1}^{\geq R_0}}^2
+\Wnorm{U\varphi}{k}{-1-\eta}{\Sigma_{\tau_1}^{\geq R_0}}^2
+\Wnorm{\varphi}{k}{\peta}{\Sigma_{\tau_1}^{\geq R_0}}^2\notag\\
&\qquad +\sum_{\mathbb{X}\in \{\mathbf{1}, U, rV\}}\bigg[\Wnorm{\mathbb{X}\varphi}{k}{0}{\MMonetwo^{R_0-1, R_0}}^2
+\sum_{\tau\in\{\tau_1,\tau_2\}}\Wnorm{\mathbb{X}\varphi}{k}{0}{\Sigma_{\tau}^{R_0-1, R_0}}^2\bigg]\notag\\
&\qquad +\sum_{\mathbb{X}\in \{\del_{t}, rV\}}\sum_{|\alpha|\leq k} \bigg|\int_{\MMonetwo^{\geq R_0-1}} \chi^2(r) r^{-1} \mathbb{X}^{\alpha}(r\vartheta) \cdot r^p V \mathbb{X}^{\alpha} \varphi d^4\mu\bigg|\notag\\
&\qquad +\sum_{\mathbb{X}\in \{\del_{t}, rV\}}\sum_{|\alpha|\leq k} \bigg|\int_{\MMonetwo^{\geq R_0-1}} \chi^2(r)r^{-1} \mathbb{X}^{\alpha}(r\vartheta) \cdot (1+r^{-\delta/2})U \mathbb{X}^{\alpha} \varphi d^4\mu\bigg|\bigg),
\end{align}
where $\chi(r):=\chi_1(r-R_0)$ with $\chi_1(x)$ being a standard smooth cutoff function that equals $1$ for $x\geq 0$ and vanishes identically for $x\leq -1$.
\end{lemma}

\begin{proof}
The wave equation \eqref{eq:rpInfinity-radial} for radially symmetric scalar $\varphi$ is equivalent to 
\begin{align}
\label{eq:Inhowave:radsym:equivalent}
UV\varphi -r^{-2}\Deltas\varphi + b_V r^{-1}V\varphi =\vartheta
\end{align}
since $\nablas\varphi=0$ and $\Deltas\varphi=0$ hold for radially symmetric function $\varphi$.
We can thus run the proof of Lemma \ref{lem:rpForInhoWave:M} for this wave equation \eqref{eq:Inhowave:radsym:equivalent}. As we have remarked after the derivation of \eqref{generalrp:commrV:once:prel:M} in the proof of Lemma \ref{lem:rpForInhoWave:M}, the restriction of $p<2$ is to make the coefficient of the term $\Wnorm{\nablas\varphi}{0}{p-3}{\MM_{\tau_1,\tau_2}^{\geq R_0}}^2$ positive. This is however irrelevant in the current situation since $\nablas\varphi=0$. In view of this fact, the proof can be extended to any $p\geq \delta$ and thereby yields the desired estimate \eqref{eq:rpInfinity-radial-EE}. 
\end{proof}

%%%%%%%%%%%%%%%%%%
\subsection{A hierarchy of estimates implies decay}
\label{subsect:HierImplyDecay}
%%%%%%%%%%%%%%%%%%

We present a lemma showing that a hierarchy of estimates implies a rate of decay for the energy in the hierarchy. This is a basic lemma that will be frequently used to derive energy decay estimates.

The way this lemma is stated is in the same spirit of \cite[Lemma 5.2]{andersson2019stability}.\footnote{In fact, the current statement exhibits only  the simpler, special  case with $\gamma=0$ and $k'=0$ in \cite[Lemma 5.2]{andersson2019stability}.} The main ingredient in its proof is an application of the mean-value principle, and we guide the reader to a proper proof therein. In applications, $k$ represents a regularity level, $p$ represents a weight that arises from $r^p$ estimates, $\tau$ represents a time function, and $F(k,p,\tau)$ represents a weighted energy. 

\begin{lemma}[A hierarchy of estimates implies decay]
\label{lem:HierachyToDecay}
Let $D\geq 0$ and $\tau_0\geq 1$. Let $p_1,p_2\in\RR$ be such that $p_1\leq p_2-1$, and let $k_0\in\NN^+$ be suitably large\footnote{It suffices to require $k_0$ to be larger than the smallest positive value that is not less than $p_2-p_1$.}. Let $F:\{0,\ldots,k_0\}\times[p_1-1,p_2]\times[\tau_0,\infty)\rightarrow[0,\infty)$ be such that $F(k,p,\tau)$ is Lebesgue measurable in $\tau$ for each $k$ and $p$. 

If the following hierarchy of estimates hold:
\begin{subequations}
\begin{enumerate}[label=\arabic*)]
\item\label{pt:implydecay:1}{} [monotonicity] for all $k,k_1,k_2\in\{0,\ldots,k_0\}$ with $k_1\leq k_2$, all $p, p',p''\in[p_1,p_2]$ with $p'\leq p''$, and all $\tau\geq \tau_0$,
\begin{align}
F(k_1,p,\tau)\lesssim{}& F(k_2,p,\tau) ,
\label{HierarchyToDecay:Mono:k}\\
F(k, p',\tau)\lesssim{}& F(k, p'',\tau) ,
\label{HierarchyToDecay:Mono:p}
\end{align}

\item\label{pt:implydecay:2}{} [interpolation] for all $k\in\{0,\ldots,k_0\}$, all $p, p', p''\in[p_1,p_2]$ such that $ p'\leq p \leq p''$, and all $\tau\geq \tau_0$,
\begin{align}
F(k,p,\tau)
\lesssim{}&
F(k, p',\tau)^{\frac{ p''-p}{ p''- p'}}
F(k, p'',\tau)^{\frac{p- p'}{ p''- p'}} ,
\label{HierarchyToDecay:Interpolate}
\end{align}

\item\label{pt:implydecay:3}{} [an integrated inequality] for all $k\in\{0,\ldots,k_0\}$, $p\in[p_1,p_2]$, and $\tau_2>\tau_1\geq \tau_0$,
\begin{align}
F(k,p,\tau_2)
+\int_{\tau_1}^{\tau_2} F(k,p-1,\tau) d \tau
\lesssim F(k,p,\tau_1) +\tau_1^{p-p_2}D,
\label{HierarchyToDecay:Evolution}
\end{align}
\end{enumerate}
\end{subequations}
then there exists a constant $C=C(p_1, p_2)>0$ such that for all $k\in\{0,\ldots,k_0\}$, all $p\in[p_1,p_2]$, and all $\tau_2>\tau_1\geq \tau_0$,
\begin{align}
F(k,p,\tau_2) \leq{}&C \la \tau_2-\tau_1\ra^{p-p_2}  (F(k, p_2,\tau_1) +D).
\label{HierarchyToDecay:conclusion}
\end{align}
\end{lemma}

\begin{remark}
Note that the assumptions \ref{pt:implydecay:1} and \ref{pt:implydecay:2} are trivially satisfied for the energies that appear in the global $r^p$ estimates, while the assumption \ref{pt:implydecay:3} is usually the global $r^p$ estimate we obtain.
\end{remark}

%%%%%%%%%%%%%%%%%%%
%%%%%%%%%%%%%%%%%%%

\section{Global existence and weak decay estimates}
\label{sect:existenceanddecay}

%%%%%%%%%%%%%%%%%%%
%%%%%%%%%%%%%%%%%%%

The aim of this section is twofold:  to show the global existence of solutions to the quasilinear wave equation with the null condition for small initial data, and to derive (weak) energy and pointwise decay estimates for the solution.  The standard energy estimates and Morawetz estimates for wave equations proven in Sections \ref{sect:BasicEnerEsti:M}--\ref{sect:HighEMEstis:M} are  combined with the $r^p$ estimates near infinity to yield a global $r^p$ estimate for wave equations in Section \ref{sect:GlobalrpForGeneWave:M}. Afterwards, we apply this global $r^p$ estimate to our model problem to show the global existence in Section \ref{sect:GlobalExist:Qwave:M} and deduce the decay estimates in Section \ref{sect:weakEnerPtwDec:M}.

%%%%%%%%%%%%%%%%%%%%
\subsection{Global $r^p$ estimate for wave equations}
\label{sect:GlobalrpForGeneWave:M}
%%%%%%%%%%%%%%%%%%%%

The formula \eqref{eq1-07-06-2022} for the wave operator acting on a scalar field $\psi$ can be written as
 \begin{equation}
\label{inhowave:RF:M}
-UV\Psi + r^{-2}\Deltas \Psi= r\Box \psi.
\end{equation}
This can be put into the form of \eqref{InhoWave:rplemma:M} with $\varphi=\Psi$, $\vartheta=r\Box\psi$, and $b_V=b_0=0$ that automatically {satisfies} the assumptions of Lemma \ref{lem:rpForInhoWave:M}.
Consequently,
the estimate \eqref{generalrp:M} applies; that is, for any $p\in [\delta, 2-\delta], k\in \mathbb{N}$, there exist constants $\bar{R}_0=\bar{R}_0(\delta, k)>10$ and $C_0=C_0(\delta, k)>0$ such that for all $R_0\geq \bar{R}_0$ and $\tau_2>\tau_1\geq \tau_0$,
\begin{align}
\label{eq:rpInfty:InhoWave:M}
&\Wnorm{rV\Psi}{k}{p-2}{\Sigma_{\tau_2}^{\geq R_0}}^2
+\Wnorm{U\Psi}{k}{-1-\eta}{\Sigma_{\tau_2}^{\geq R_0}}^2
+\Wnorm{\nablas^{\leq 1}\Psi}{k}{\peta}{\Sigma_{\tau_2}^{\geq R_0}}^2
+\Flux{k}{\tau_1,\tau_2}{p}{\Psi}\notag\\
&
+\Wnorm{\Psi}{k+1}{p-3}{\MMonetwo^{\geq R_0}}^2
+\Wnorm{U\Psi}{k}{-1-\delta/2}{\MMonetwo^{\geq R_0}}^2\notag\\
\leq {}& C_0\bigg(
\Wnorm{rV\Psi}{k}{p-2}{\Sigma_{\tau_1}^{\geq R_0}}^2
+\Wnorm{U\Psi}{k}{-1-\eta}{\Sigma_{\tau_1}^{\geq R_0}}^2
+\Wnorm{\nablas^{\leq 1}\Psi}{k}{\peta}{\Sigma_{\tau_1}^{\geq R_0}}^2
\notag\\
&\qquad 
+\Wnorm{\Psi}{k+1}{0}{\MMonetwo^{R_0-1, R_0}}^2
+\sum_{\tau'\in\{\tau_1,\tau_2\}}\Wnorm{\Psi}{k+1}{0}{\Sigma_{\tau'}^{R_0-1, R_0}}^2\notag\\
&\qquad +\sum_{|\alpha|\leq k} \bigg|\int_{\MMonetwo^{\geq R_0-1}} \chi^2(r) r^{-1}\DDb^{\alpha}(r^2\Box\psi) \times r^p V \DDb^{\alpha} \Psi d^4\mu\bigg|\notag\\
&\qquad +\sum_{|\alpha|\leq k} \bigg|\int_{\MMonetwo^{\geq R_0-1}} \chi^2(r) r^{-1}\DDb^{\alpha}(r^2\Box\psi) \times (1+r^{-\delta/2})U \DDb^{\alpha} \Psi d^4\mu\bigg|\bigg).
\end{align}

 By adding a suitably large multiple of the high-order energy--Morawetz estimate \eqref{eq:HighEnerMora:psi:M} to the above $r^p$ estimate \eqref{eq:rpInfty:InhoWave:M} near infinity, the third last line in \eqref{eq:rpInfty:InhoWave:M}, in which the terms are supported in a compact region $R_0-1\leq r\leq R_0$, can be absorbed.  Further, by the formula \eqref{standardEner:zero:M} of the standard high-order energy, it holds for any $\tau\geq \tau_0$ that
\begin{align*}
\Wnorm{\Psi}{k+1}{-2}{\Sigma_{\tau}^{0, R_0}}^2
+\Wnorm{U\Psi}{k}{-1-\eta}{\Sigma_{\tau}^{0, R_0}}^2
\lesssim_{k, R_0} \SEnergy{k}{\tau}{\Psi}\lesssim_{k} \Wnorm{\Psi}{k+1}{-2}{\Sigma_{\tau}}^2+\Wnorm{U\Psi}{k}{-1-\eta}{\Sigma_{\tau}}^2.
\end{align*}
Therefore, we arrive at the following global $r^p$ estimate:

\begin{theorem}[Global $r^p$ estimate for wave equations]
\label{thm:globalrp:InhoWave:M}
Let $\delta\in (0,\frac{1}{2})$. For any $p\in [\delta, 2-\delta]$ and any $k\in \mathbb{N}$, there exist constants ${R}_0={R}_0(\delta, k)>10$ and $C_0=C_0(\delta, k)>0$ such that for any $\tau_2>\tau_1\geq \tau_0$,
\begin{align}
\label{eq:globalrp:InhoWave:M}
&\Energy{k}{p}{\tau_2}{\Psi} 
+\Flux{k}{\tau_1,\tau_2}{p}{\Psi}
+\Wnorm{\Psi}{k+1}{p-3}{\MMonetwo}^2
+\Wnorm{U\Psi}{k}{-1-\delta/2}{\MMonetwo}^2\notag\\
\leq {}& C_0\bigg(
\Energy{k}{p}{\tau_1}{\Psi} + \sum_{i=1,2,3}\Error{k}{i}{\Psi}+\Error{k}{4,p}{\Psi}\bigg),
\end{align}
where the error integral terms are defined by
\begin{subequations}
\label{def:errorinte:globalrp:inhowave:M}
\begin{align}
\label{errorinte:globalrp:inhowave:1:M}
\Error{k}{1}{\Psi}:=&\sum_{|\alpha|\leq k} \bigg|\int_{\MMonetwo}\del_t(r\del^{\alpha}\psi)\cdot r\Box (\del^{\alpha}\psi) d^4\mu \bigg|,\\
\label{errorinte:globalrp:inhowave:2:M}
\Error{k}{2}{\Psi}:=&
\sum_{|\alpha|\leq k} \bigg|\int_{\MMonetwo}\big(2 - (1+r)^{-\delta}\big)\del_r(r\del^{\alpha}\psi)\cdot  r\Box (\del^{\alpha}\psi) d^4\mu \bigg|,\\
\label{errorinte:globalrp:inhowave:3:M}
\Error{k}{3}{\Psi}:=&\sum_{|\alpha|\leq k} \bigg|\int_{\MMonetwo^{\geq R_0-1}} \chi^2(r) r^{-1} (1+r^{-\frac{\delta}{2}})U \DDb^{\alpha} \Psi \cdot \DDb^{\alpha}(r^2\Box\psi) d^4\mu\bigg|,\\
\label{errorinte:globalrp:inhowave:4:M}
\Error{k}{4,p}{\Psi}:=&\sum_{|\alpha|\leq k} \bigg|\int_{\MMonetwo^{\geq R_0-1}} \chi^2(r) r^{p-1} V \DDb^{\alpha} \Psi \cdot \DDb^{\alpha}(r^2\Box\psi)d^4\mu\bigg|,
\end{align}
\end{subequations}
{and the weighted energies are defined as in Section \ref{subsect:energies}.} 
\end{theorem}

 For convenience, let us define for  $k\in\NN$, $p\geq 0$, $\tau_0\leq \tau_1<\tau_2$ and any scalar $\varphi$ that
\begin{align}
\label{def:LEFM:decay:M}
\LEFM{k}{p}{\tau_1,\tau_2}{\varphi}:=&\Energy{k}{p}{\tau_2}{\varphi} +\Flux{k}{\tau_1,\tau_2}{p}{\varphi}
+\Wnorm{\varphi}{k+1}{p-3}{\MMonetwo}^2
+\Wnorm{U\varphi}{k}{-1-\delta/2}{\MMonetwo}^2,\\
\label{def:EFM:Decay:M}
\EFM{k}{p}{\tau_1,\tau_2}{\varphi}:=&\LEFM{k}{p}{\tau_1,\tau_2}{\varphi}+\Energy{k}{p}{\tau_1}{\varphi}.
\end{align}

%%%%%%%%%%%%%%%%%%%%
\subsection{Global $r^p$ estimate and global existence for solutions to the quasilinear wave with null condition}
\label{sect:GlobalExist:Qwave:M}
%%%%%%%%%%%%%%%%%%%%

This subsection is to study the quasilinear wave equation \eqref{Qwave:M} with the null condition \eqref{nullcond:M} by applying the above Theorem \ref{thm:globalrp:InhoWave:M} to this model equation and absorbing the resulting error integrals in \eqref{def:errorinte:globalrp:inhowave:M} by the left-hand side, thus proving a global $r^p$ estimate as stated in the following theorem.

\begin{theorem}[Global $r^p$ estimate for quasilinear wave equations with null condition]
\label{thm:globalrp:M}
Let $k\geq 7$, let $\delta \in (0,\frac12)$ be a small constant, and let $p\in [\delta, 2-\delta]$. There exists an $\eps>0$ sufficiently small such that for any scalar $\psi$ solving the quasilinear wave equation \eqref{Qwave:M} with null condition \eqref{nullcond:M} and satisfying 
\begin{align}
\Energy{k}{{1+\delta}}{\tau_0}{\Psi}\leq \eps^2,
\end{align}
we have, for any $\tau_2>\tau_1\geq \tau_0$,
\begin{align}
\label{eq:globalrp:M}
\Energy{k}{p}{\tau_2}{\Psi} +\Flux{k}{\tau_1,\tau_2}{p}{\Psi}+ \Wnorm{\Psi}{k+1}{p-3}{\MMonetwo}^2
+\Wnorm{U\Psi}{k}{-1-\delta/2}{\MMonetwo}^2
\leq C \Energy{k}{p}{\tau_1}{\Psi} ,
\end{align}
with $C>0$ being a constant depending only on $\delta, k$,
\end{theorem}

\begin{remark}
Global existence of solutions to the quasilinear wave equation \eqref{Qwave:M} with null condition \eqref{nullcond:M}  for small initial data is an immediate consequence of the above Theorem \ref{thm:globalrp:M}. Further, this estimate \eqref{eq:globalrp:M} will soon be utilized to derive energy decay estimates. \end{remark}

\begin{proof}

As has been explained above, the proof is completed by applying the estimate \eqref{eq:globalrp:InhoWave:M} in the above Theorem \ref{thm:globalrp:InhoWave:M} to this quasilinear wave equation with null condition and absorbing the resulting error integrals in \eqref{def:errorinte:globalrp:inhowave:M} by the left-hand side of \eqref{eq:globalrp:InhoWave:M}. To apply Theorem \ref{thm:globalrp:InhoWave:M} in the remainder of this proof, we shall fix the constant $R_0=R_0(\delta,k)$.

For convenience, we define $\DDb_0=\del_{\tau}, \DDb_1=rV, \DDb_2=\Omega_1, \DDb_3=\Omega_2$, and $\DDb_4=\Omega_3$.

Let us first consider the error term $\Error{k}{1}{\Psi}$. Since $[\Box, \del]=0$, it holds $\Box(\del^{\alpha}\psi)=\del^{\alpha} (\Box \psi)=\del^{\alpha} (P^{abc}\del_c \psi \del_a\del_b\psi)$, and we obtain
\begin{align}
\Error{k}{1}{\Psi}=&\sum_{|\alpha|\leq k, \beta+\gamma=\alpha, |\beta|\neq k} \bigg|\int_{\MMonetwo}r^2 \del_t(\del^{\alpha}\psi)
\cdot 
P^{abc} \del_c ( \del^{\gamma} \psi) \del_a\del_b(\del^{\beta}\psi) 
 d^4\mu \bigg|\notag\\
 &+\sum_{|\alpha|= k} \bigg|\int_{\MMonetwo}r^2 P^{abc} \del_c \psi  \cdot \del_t(\del^{\alpha}\psi)
\del_a\del_b(\del^{\alpha}\psi) 
 d^4\mu \bigg|\notag\\
 := &\Error{k}{1,L}{\Psi}+\Error{k}{1,H}{\Psi}.
\end{align}

For the subterm $\Error{k}{1,L}{\Psi}$, the indices satisfy $\min\{|\gamma|+1, |\beta|+2\}\leq \frac{k+3}{2}$.  Also, observe that 
\begin{align}
\label{exp:Qterms:inDDb:away:M}
P^{abc}\del_c \psi_1 \del_a\del_b \psi_2 = O(r^{-3}) \DDb^{\leq 1} (r\psi_1) \DDb^{\leq 2} (r\psi_2), \qquad \forall \, r\geq 1,
\end{align}
which holds because of the null condition \eqref{nullcond:M}.
Hence, for this subterm with the integral over $\MMonetwo^{\geq 1}$,  it is bounded, using the Cauchy--Schwarz, by
\begin{align*}
&\sup_{\MMonetwo}\big(\norm{\Psi}{\frac{k+3}{2}}\big)
\big(\Wnorm{\Psi}{k+1}{\delta-3}{\MMonetwo^{\geq 1}}^2
+\Wnorm{U\Psi}{k}{-1-\delta/2}{\MMonetwo^{\geq 1}}^2\big).
\end{align*}
Using a Sobolev-type estimate  
\begin{align}
\label{eq:Sobolev:weakdecay:Wterm:M}
\sup_{\MMonetwo}\big(\norm{\Psi}{\frac{k+3}{2}}\big)\lesssim \sqrt{\sup_{\tau\in [\tau_1,\tau_2]}\Energy{k}{{1+\delta}}{\tau}{\Psi} }
\end{align}
which follows from the Sobolev inequality {\eqref{eq:Sobolev:2}} and the fact $\frac{k+3}{2}\leq k+1-3$ (since $k\geq 7$), this is further bounded by
\begin{align}
\label{exp:control:EM:mid:1:M}
\sqrt{\sup_{\tau\in [\tau_1,\tau_2]}\Energy{k}{{1+\delta}}{\tau}{\Psi} }\times 
\big(\Wnorm{\Psi}{k+1}{\delta-3}{\MMonetwo}^2
+\Wnorm{U\Psi}{k}{-1-\delta/2}{\MMonetwo}^2\big),
\end{align}
As to this subterm with the integral over $\MMonetwo^{\leq 1}$, it is again controlled by \eqref{exp:control:EM:mid:1:M} using the Cauchy--Schwarz. In conclusion, the error term $\Error{k}{1,L}{\Psi}$ is bounded by formula \eqref{exp:control:EM:mid:1:M}.

We next consider the other subterm $\Error{k}{1,H}{\Psi}$, and we split it into two terms $\Error{k}{1,H, \leq r_0}{\Psi}$ and $\Error{k}{1,H,\geq r_0}{\Psi}$ that are with the integrals over $\MMonetwo^{\leq r_0}$ and $\MMonetwo^{\geq r_0}$, respectively. The parameter $r_0\in [1,2]$ is chosen such that 
\begin{align}
\label{eq:r0choice:meanvalue:globalexistLM}
\int_{\tau_1}^{\tau_2}\int_{\mathbb{S}^2} |\del^{\leq k+1}\psi(r_0)|^2 d\tau d^2\mu \leq \int_{\MMonetwo\cap\{1\leq r\leq 2\}}  |\del^{\leq k+1}\psi|^2 d^4\mu.
\end{align}
In particular, we shall notice that the integrand in this subterm involves ${(k+2)}$-th order of derivative term, which is present because of the quasilinear nature of the model equation and can not be bounded by the above approach. To reduce the regularity of the terms in the integrand, we shall make use of the following elementary formula with $f_{abc}$ being coefficient functions
 \begin{align}
 \label{Qterm:IBPFormula:M}
f^{abc} \del_{c}\varphi \del_{a}\del_{b}\varphi
=&\frac{1}{2}\Big(\del_{a} (f^{abc} \del_{c}\varphi  \del_{b}\varphi )
+\del_{b} (f^{abc} \del_{c}\varphi  \del_{a}\varphi)
- \del_{c} (f^{abc} \del_{a}\varphi  \del_{b}\varphi)\notag\\
&\quad -\del_{a} (f^{abc}) \del_{c}\varphi  \del_{b}\varphi 
-\del_{b} (f^{abc}) \del_{c}\varphi  \del_{a}\varphi
+ \del_{c} (f^{abc}) \del_{a}\varphi  \del_{b}\varphi
\Big).
 \end{align}

Indeed, for the term $\Error{k}{1,H, \leq r_0}{\Psi}$, we
can use the above formula \eqref{Qterm:IBPFormula:M} with $\varphi=\del^{\alpha}\psi$, $f^{abc}=P^{abc}\del_c\psi$ for $c=0$, and $f^{abc}=0$ for $c\neq 0$. 
Then this together with \eqref{eq:r0choice:meanvalue:globalexistLM} immediately yields that both the arising flux terms and the arising spacetime terms, hence the term $\Error{k}{1,H, \leq r_0}{\Psi}$, are bounded by 
\begin{align}
\label{exp:control:EM:mid2:2:M}
\sqrt{\sup_{\tau\in [\tau_1,\tau_2]}\Energy{k}{{1+\delta}}{\tau}{\Psi} }\times 
\EFM{k}{p}{\tau_1,\tau_2}{\Psi}.
\end{align}
For the other term $\Error{k}{1,H, \geq r_0}{\Psi}$, because of \eqref{exp:Qterms:inDDb:away:M}, it {is bounded by} 
\begin{align}
\label{formula:Qwave:AfterIBP:1:M}
&\sum_{|\alpha|= k} \bigg|\int_{\MMonetwo^{\geq r_0}}r^{-1} \DDb \psi  \cdot \del_t(r \del^{\alpha}\psi)
\DDb^{\leq 1}(r\del^{\alpha}\psi) 
 d^4\mu \bigg|\notag\\
&+
\sum_{|\alpha|= k} \bigg|\int_{\MMonetwo^{\geq r_0}}h^{ijl} \DDb_l  \psi  \cdot \del_t(r \del^{\alpha}\psi)
\DDb_i\DDb_j(r\del^{\alpha}\psi) 
 d^4\mu \bigg|,
\end{align}
where $h^{ijl}=O(r^{-1})$ and $h^{ijl}=h^{jil}$ for $i,j,l\in \{0,1,2,3,4\}$. The term in the second line of the above \eqref{formula:Qwave:AfterIBP:1:M}  arises again from the quasilinear nature of the considered wave equation, and we shall use a similar formula as \eqref{Qterm:IBPFormula:M} to transform this term into flux terms and spacetime terms. Fortunately, because of \eqref{eq:Sobolev:weakdecay:Wterm:M}, the arising flux terms and spacetime terms are controlled by
\eqref{exp:control:EM:mid2:2:M} as well. The term in the first line of \eqref{formula:Qwave:AfterIBP:1:M} is even simpler, since a direct application of the Cauchy--Schwarz implies that it is bounded by \eqref{exp:control:EM:mid:1:M}. In total, the subterm $\Error{k}{1,H}{\Psi}$ is bounded by \eqref{exp:control:EM:mid2:2:M}.

In view of the above discussions for the subterms $\Error{k}{1,L}{\Psi}$  and $\Error{k}{1,H}{\Psi}$, we conclude that the error term $\Error{k}{1}{\Psi}$  is bounded by  \eqref{exp:control:EM:mid2:2:M}.

The same argument simply applies to the terms $\Error{k}{2}{\Psi}$ and $\Error{k}{3}{\Psi}$ without essentially making any changes, and this shows the error terms $\Error{k}{2}{\Psi}$ and $\Error{k}{3}{\Psi}$  are both bounded by \eqref{exp:control:EM:mid2:2:M}.

In the end, we estimate the error term $\Error{k}{4,p}{\Psi}$, where the integrated region is a submanifold of $\MMonetwo\cap\{r\geq 1\}$. Because of the formula \eqref{exp:Qterms:inDDb:away:M} and the following  relations in region $r\geq 1$:
 \begin{align*}
 [\DDb, \DDb]=& {\DDb}, & \DDb( r)=& O(r), 
 \end{align*}
 we derive 
  \begin{align*}
 \DDb^{\alpha}(r^2\Box\psi)
 =&  \DDb^{\alpha}(r^2 P^{abc}\del_c\psi \del_a\del_b\psi)\notag\\
 =& \sum_{\substack{|\beta|\leq |\alpha|+1, |\gamma|\leq |\alpha |+1\\ |\beta|+|\gamma|\leq |\alpha|+3}}O(r^{-1})\DDb^{\beta}\Psi \DDb^{\gamma}\Psi + r^2 P^{abc}\del_c(r^{-1}\Psi )\del_a\del_b(r^{-1}{ \DDb^{\alpha } \Psi})\notag\\
 =& \sum_{\substack{|\beta|\leq |\alpha|+1, |\gamma|\leq |\alpha |+1\\ |\beta|+|\gamma|\leq |\alpha|+3}}O(r^{-1})\DDb^{\beta}\Psi \DDb^{\gamma}\Psi + Q^{ijl}\DDb_l \Psi \DDb_{i}\DDb_j ( \DDb^{\alpha } \Psi)
 + R^{ij}\Psi \DDb_{i}\DDb_j ( \DDb^{\alpha } \Psi),
 \end{align*}
where $Q^{ijl}$ and $R^{ij}$ 
are both $O(r^{-1})$ functions and symmetric in $(i,j)$. Consequently, we can decompose the error term $\Error{k}{4,p}{\Psi}$ into
\begin{align}
\Error{k}{4,p}{\Psi}= \sum_{m=1,2,3}\Error{k}{4,p,m}{\Psi},
\end{align}
where the subterms are 
\begin{subequations}
\begin{align}
\Error{k}{4,p,1}{\Psi}=&\sum_{\substack{|\alpha|\leq k\\|\beta|\leq |\alpha|+1, |\gamma|\leq |\alpha| +1\\ |\beta|+|\gamma|\leq |\alpha|+3}}\bigg|\int_{\MMonetwo^{\geq R_0-1}} \chi^2(r) O(r^{p-2}) V \DDb^{\alpha} \Psi \cdot  \DDb^{\beta}\Psi \DDb^{\gamma}\Psi d^4\mu\bigg|,
\\
\Error{k}{4,p,2}{\Psi}=&\sum_{|\alpha|\leq k} \bigg|\int_{\MMonetwo^{\geq R_0-1}} \chi^2(r) r^{p-1} V \DDb^{\alpha} \Psi \cdot Q^{ijl}\DDb_l \Psi \DDb_{i}\DDb_j ( \DDb^{\alpha } \Psi) d^4\mu\bigg|,
\\
\Error{k}{4,p,3}{\Psi}=&\sum_{|\alpha|\leq k} \bigg|\int_{\MMonetwo^{\geq R_0-1}} \chi^2(r) r^{p-1} V \DDb^{\alpha} \Psi \cdot R^{ij}\Psi \DDb_{i}\DDb_j ( \DDb^{\alpha } \Psi)d^4\mu\bigg|,
\end{align}
\end{subequations}
and we now estimate each subterm separately.

For the subterm $\Error{k}{4,p,1}{\Psi}$, we first observe that $\min\{|\beta|, |\gamma|\}\leq \frac{k+3}{2}$, thus by the above same argument, it is bounded by \eqref{exp:control:EM:mid2:2:M} if $k\geq 7$. 

The subterm $\Error{k}{4,p,2}{\Psi}$ is estimated in a similar manner as the one in estimating $\Error{k}{1,H, \geq r_0}{\Psi}$. The integrand contains one factor which has ${(k+2)}$-th order of regularity and thus can not be controlled  by either energy or  spacetime integrals. We again utilize formula \eqref{Qterm:IBPFormula:M} to transform this subterm into energy terms and spacetime terms that contain at most ${(k+1)}$-th order of derivatives of $\Psi$. This trivially implies that this subterm is bounded using the Cauchy--Schwarz by \eqref{exp:control:EM:mid2:2:M} if $k\geq 7$.

The approach in estimating the last subterm $\Error{k}{4,p,3}{\Psi}$ is identical to (actually, simpler than) the one in treating the subterm $\Error{k}{4,p,2}{\Psi}$ and thus omitted. This subterm is bounded by \eqref{exp:control:EM:mid2:2:M} as well if $k\geq 7$.

In summary, we obtain that for any $k\geq 7$, $p\in [\delta, 2-\delta]$, and any $\tau_2>\tau_1\geq \tau_0$,
\begin{align}
\label{eq:globalrp:QWave:Aprior:step2:M}
&\Energy{k}{p}{\tau_2}{\Psi} 
+\Flux{k}{\tau_1,\tau_2}{p}{\DDb^{\alpha}\Psi}
+\Wnorm{\Psi}{k+1}{p-3}{\MMonetwo}^2
+\Wnorm{U\Psi}{k}{-1-\delta/2}{\MMonetwo}^2\notag\\
\lesssim_{\delta,k} {}&
\Energy{k}{p}{\tau_1}{\Psi} +\sqrt{\sup_{\tau\in [\tau_1,\tau_2]}\Energy{k}{{1+\delta}}{\tau}{\Psi} }\times \EFM{k}{p}{\tau_1,\tau_2}{\Psi}\notag\\
\lesssim_{\delta,k} {}&
\Energy{k}{p}{\tau_1}{\Psi} +\sqrt{\sup_{\tau\in [\tau_1,\tau_2]}\Energy{k}{{1+\delta}}{\tau}{\Psi} }\times 
\Big(\sum_{\tau\in\{\tau_1,\tau_2\}}\Energy{k}{p}{\tau}{\Psi} +\Flux{k}{\tau_1,\tau_2}{p}{\Psi}
\notag\\
&\qquad\qquad\qquad+\Wnorm{\Psi}{k+1}{p-3}{\MMonetwo}^2
+\Wnorm{U\Psi}{k}{-1-\delta/2}{\MMonetwo}^2\Big).
\end{align}
Consider ${p=1+\delta}$ first. Due to $\Energy{k}{{1+\delta}}{\tau_0}{\Psi}\leq \eps^2$ that holds by assumption, the above inequality implies that for sufficiently small $\eps$, it holds for any $\tau\geq \tau_0$ that
\begin{align*}
\Energy{k}{{1+\delta}}{\tau}{\Psi} 
+\Flux{k}{\tau_0,\tau}{{1+\delta}}{\DDb^{\alpha}\Psi}
+\Wnorm{\Psi}{k+1}{-2}{\MM_{\tau_0,\tau}}^2
+\Wnorm{U\Psi}{k}{-1-\delta/2}{\MM_{\tau_0,\tau}}^2
\lesssim_{\delta,k} {}&
\Energy{k}{{1+\delta}}{\tau_0}{\Psi}.
\end{align*}
This in particular yields that $\sup_{\tau\geq \tau_0}\Energy{k}{{1+\delta}}{\tau}{\Psi} \lesssim_{\delta,k} \eps^2$. We next substitute this back into the inequality \eqref{eq:globalrp:QWave:Aprior:step2:M}, which then proves inequality \eqref{eq:globalrp:M} immediately by taking $\eps$ sufficiently small.
\end{proof}

%%%%%%%%%%%%%%%%%%
\subsection{Weak energy and pointwise decay estimates}
\label{sect:weakEnerPtwDec:M}
%%%%%%%%%%%%%%%%%%

To begin with, we utilize the proven estimate \eqref{eq:globalrp:M} to derive energy decay estimates. This also includes a useful statement that the energy of $\del_{\tau}\Psi$ enjoys faster $\tau$-decay than the one of $\Psi$.

\begin{proposition}[Weak energy decay estimates]
\label{prop:WeakEnerDecay}
Let $k\geq 7$, and let $\delta \in (0,\min\{\frac12, 2\eta\})$. There exists a sufficiently small $\eps>0$ such that for any scalar $\psi$ solving the quasilinear wave equation \eqref{Qwave:M} with null condition \eqref{nullcond:M} and satisfying
\begin{align}
\label{assump:weakEnerDecay:2-delta:M}
\Energy{k}{2-\delta}{\tau_0}{\Psi}\leq \eps^2,
\end{align}
\begin{enumerate}[label=\roman*)]
\item\label{point1:weakenergydec:M}
it holds for any $p\in [\delta, 2-\delta]$  and $\tau_2>\tau_1\geq \tau_0$ that
\begin{align}
\label{weakenerdec:psi:M}
&\Energy{k}{p}{\tau_2}{\Psi}
+\Flux{k}{\tau_2,\infty}{p}{\Psi}
+ \Wnorm{\Psi}{k+1}{p-3}{\MMtwoinfty}^2
+\Wnorm{U\Psi}{k}{-1-\delta/2}{\MMtwoinfty}^2\notag\\
 \lesssim_{\delta,k} &\la \tau_2-\tau_1\ra^{-(2-\delta)+p} \Energy{k}{2-\delta}{\tau_1}{\Psi};
\end{align}
\item\label{point2:weakenergydec:M} and it holds for any
$n\leq k-7$, $p\in [\delta, 2-\delta]$ and $\tau_2>\tau_1\geq \tau_0$ that
\begin{align}
\label{weakenerdec:psi:full:M}
&\Energy{k-n}{p}{\tau_2}{\del_{\tau}^n\Psi} 
+\Flux{k-n}{\tau_2,\infty}{p}{\del_{\tau}^{n}\Psi}
+ \Wnorm{\del_{\tau}^n\Psi}{k-n+1}{p-3}{\MMtwoinfty}^2
+\Wnorm{U\del_{\tau}^n\Psi}{k-n}{-1-\delta/2}{\MMtwoinfty}^2\notag\\\lesssim_{\delta,k} &\la \tau_2-\tau_1\ra^{-(2-\delta)-(2-2\delta)n+p} \Energy{k}{2-\delta}{\tau_1}{\Psi} .
\end{align}
\end{enumerate}
\end{proposition}

\begin{proof}
We present the proofs of points \ref{point1:weakenergydec:M} and \ref{point2:weakenergydec:M} separately in this proof. The proof of point \ref{point2:weakenergydec:M} is further divided into $2$ steps.

\underline{Proof of Point \ref{point1:weakenergydec:M} }.
By Theorem \ref{thm:globalrp:M}, the assumption \eqref{assump:weakEnerDecay:2-delta:M} for the initial data yields the global $r^p$ estimate \eqref{eq:globalrp:M} for any $p\in [\delta, 2-\delta]$. Following \eqref{def:globalrpEnergy:highorder:M}, we further define $\Energy{k}{p}{\tau}{\varphi}:=0$ for $p<0$. By this definition, one finds $\Wnorm{\Psi}{k+1}{p-3}{\MMonetwo}^2
+\Wnorm{U\Psi}{k}{-1-\delta/2}{\MMonetwo}^2\gtrsim \int_{\tau_1}^{\tau_2}\Energy{k}{p-1}{\tau}{\Psi} d \tau$  holds for any $p\in [\delta, 2-\delta]$ if $\delta\leq 2\eta$. This together with the proven estimate \eqref{eq:globalrp:M} then yields that for any $p\in [\delta, 2-\delta]$, $k\geq 7$ and any $\tau_2>\tau_1\geq\tau_0$, 
\begin{align}
\Energy{k}{p}{\tau_2}{\Psi} + \int_{\tau_1}^{\tau_2}\Energy{k}{p-1}{\tau}{\Psi} d \tau
\lesssim_{\delta, k} \Energy{k}{p}{\tau_1}{\Psi} .
\end{align}
Inequality \eqref{weakenerdec:psi:M} for any $k\geq 7$ then follows by an application of Lemma \ref{lem:HierachyToDecay}.
This proves point \ref{point1:weakenergydec:M}.

\underline{Proof of Point \ref{point2:weakenergydec:M}}.
We shall prove the estimate \eqref{weakenerdec:psi:full:M} of point \ref{point2:weakenergydec:M} by an induction argument on the value of $n$. Assume it holds for all $n'\leq n_0-1$, where $n_0\leq k-7$, that is, assume for any $p\in [\delta, 2-\delta]$ and any $\tau_2>\tau_1'\geq \tau_0$,
\begin{align}
\label{inductionLweakenerdec:deltaupsi:M}
\Energy{k-n'}{p}{\tau_2}{\del_{\tau}^{n'}\Psi} \lesssim_{\delta,k} \la \tau_2-\tau_1'\ra^{-(2-\delta)-(2-2\delta)n'+p} \Energy{k}{2-\delta}{\tau_1'}{\Psi},
\end{align}
and we show inequality \eqref{weakenerdec:psi:full:M} for $n=n_0$. 

\underline{Step 1}.
Our first step is to show that for any $p\in [\delta, 2-\delta]$, $n_0\leq {k-7}$ and $\tau_2>\tau_1\geq \tau_0$,
\begin{align}
\label{eq:globalrp:deltau:M}
&\Energy{k-n_0}{p}{\tau_2}{\del_{\tau}^{n_0}\Psi} 
+\Flux{k-n_0}{\tau_1,\tau_2}{p}{\del_{\tau}^{n_0}\Psi}
+ \Wnorm{\del_{\tau}^{n_0}\Psi}{k-{n_0}+1}{p-3}{\MMonetwo}^2
+\Wnorm{U\del_{\tau}^{n_0}\Psi}{k-{n_0}}{-1-\delta/2}{\MMonetwo}^2\notag\\
\lesssim_{\delta, k} {}&\Energy{k-{n_0}}{p}{\tau_1}{\del_{\tau}^{n_0}\Psi} 
+\eps \sum_{0\leq j_0\leq {n_0}-1}\la \tau_1-\tau_0\ra^{-(2-2\delta)({n_0}-j_0)}\Energy{k-{n_0}}{p}{\tau_1}{\del_{\tau}^{j_0}\Psi}.
\end{align}

 Clearly, since $[\del_{\tau}, \Box]=0$, the estimate \eqref{eq:globalrp:InhoWave:M}, with $\psi$ and $\Psi$ replaced by $\del_{\tau}^n\psi$ and $\del_{\tau}^n\Psi$ respectively, holds true. That is,
\begin{align}
\label{eq:globalrp:InhoWave:deltau:M}
\LEFM{k-n}{p}{\tau_1,\tau_2}{\del_{\tau}^n\Psi} 
\lesssim_{\delta,k} {}& 
\Energy{k-n}{p}{\tau_1}{\del_{\tau}^n\Psi} + \sum_{i=1,2,3}\Error{k-n}{i}{\del_{\tau}^n\Psi}+\Error{k-n}{4,p}{\del_{\tau}^n\Psi},
\end{align}
where the error integral terms are
\begin{subequations}
\label{def:errorinte:globalrp:inhowave:deltau:M}
\begin{align}
\label{errorinte:globalrp:inhowave:deltau:1:M}
\Error{k-n}{1}{\del_{\tau}^n\Psi}:=&\sum_{|\alpha|\leq k-n} \bigg|\int_{\MMonetwo}\del_t(r\del^{\alpha}\del_{\tau}^n\psi)\cdot r\Box (\del^{\alpha}\del_{\tau}^n \psi) d^4\mu \bigg|,\\
\label{errorinte:globalrp:inhowave:deltau:2:M}
\Error{k-n}{2}{\del_{\tau}^n\Psi}:=&
\sum_{|\alpha|\leq k-n} \bigg|\int_{\MMonetwo}\big(2 - (1+r)^{-\delta}\big)\del_r(r\del^{\alpha}\del_{\tau}^n\psi)\cdot  r\Box (\del^{\alpha}\del_{\tau}^n\psi) d^4\mu \bigg|,\\
\label{errorinte:globalrp:inhowave:deltau:3:M}
\Error{k-n}{3}{\del_{\tau}^n\Psi}:=&\sum_{|\alpha|\leq k-n} \bigg|\int_{\MMonetwo^{\geq R_0-1}} \chi^2(r) r^{-1} (1+r^{-\frac{\delta}{2}})U \DDb^{\alpha} \del_{\tau}^n\Psi \cdot \DDb^{\alpha}(r^2\Box\del_{\tau}^n\psi) d^4\mu\bigg|,\\
\label{errorinte:globalrp:inhowave:deltau:4:M}
\Error{k-n}{4,p}{\del_{\tau}^n\Psi}:=&\sum_{|\alpha|\leq k-n} \bigg|\int_{\MMonetwo^{\geq R_0-1}} \chi^2(r) r^{p-1} V \DDb^{\alpha} \del_{\tau}^n\Psi \cdot \DDb^{\alpha}(r^2\Box\del_{\tau}^n\psi)d^4\mu\bigg|.
\end{align}
\end{subequations}

The proof of the above inequality \eqref{eq:globalrp:deltau:M} follows in the same way as the one in proving \eqref{eq:globalrp:M}, and we describe the difference in estimating, for instance, the error term $\Error{k-n}{1}{\del_{\tau}^n\Psi}$. Note that 
$$
\Box (\del^{\alpha}\del_{\tau}^n \psi) = \del^{\alpha}\del_{\tau}^n (\Box \psi)=\sum_{m=0}^n \sum_{\beta +\gamma=\alpha} P^{abc}\del_c(\del^{\gamma} \del_{\tau}^{n-m} \psi) \cdot \del_a\del_b(\del^{\beta}\del_{\tau}^{m}\psi),
$$
hence, we derive
\begin{align}
&\Error{k-n}{1}{\del_{\tau}^n\Psi}\notag\\
=&\sum_{\substack{|\alpha|\leq k-n, \beta+\gamma=\alpha\\ 0\leq m\leq n, |\beta|+m< k}} \bigg|\int_{\MMonetwo}r^2 \del_{\tau}(\del^{\alpha}\del_{\tau}^n\psi)
\cdot 
P^{abc}\del_c(\del^{\gamma} \del_{\tau}^{n-m} \psi) \cdot \del_a\del_b(\del^{\beta}\del_{\tau}^{m}\psi) 
 d^4\mu \bigg|\notag\\
 &+\sum_{|\alpha|= k-n} \bigg|\int_{\MMonetwo}r^2 P^{abc} \del_c \psi  \cdot \del_{\tau}(\del^{\alpha}\del_{\tau}^n\psi)
\del_a\del_b(\del^{\alpha}\del_{\tau}^n\psi) 
 d^4\mu \bigg|\notag\\
 := &\Error{k-n}{1,L}{\del_{\tau}^n\Psi}
 +\Error{k-n}{1,H}{\del_{\tau}^n\Psi}.
\end{align}
For the subterm $\Error{k-n}{1,L}{\del_{\tau}^n\Psi}$, at least one of the orders of $\del$-derivatives (we also count the $\del_{\tau}^{n-m}$ and $\del_{\tau}^m$) in the last two factors in the integrand  is less than $\frac{k+3}{2}$; without loss of generality, let us assume the last factor contains the lower order of $\del$-derivatives. If $m<n$,  we can use the formula \eqref{exp:Qterms:inDDb:away:M} in $r\geq 1$ and the H\"older inequality and take the $L^{\infty}$ norm for the last factor in the integrand to bound this subterm by
\begin{align}
\label{exp:control:EM:mid:deltau:1:M}
&\sum_{m=0}^{n-1}\sup_{\MMonetwo}\Big(\norm{\del_{\tau}^m\Psi}{\frac{k+3}{2}-m}\Big)
\Big(\EFM{k-n}{p}{\tau_1,\tau_2}{\del_{\tau}^n\Psi}  \Big)^{\frac{1}{2}}
\Big(\EFM{k-n}{p}{\tau_1,\tau_2}{\del_{\tau}^{n-m}\Psi}\Big)^{\frac{1}{2}}\notag\\
\lesssim &\,\eps \EFM{k-n}{p}{\tau_1,\tau_2}{\del_{\tau}^n\Psi}  +\eps^{-1}\sum_{m=0}^{n-1}\sup_{\MMonetwo}\Big(\norm{\del_{\tau}^m\Psi}{\frac{k+3}{2}-m}\Big)^2
\EFM{k-n}{p}{\tau_1,\tau_2}{\del_{\tau}^{n-m}\Psi}.
\end{align} 
Instead, if $m=n$, then we take $L^{\infty}$ norm for the second factor and use the Cauchy--Schwarz inequality for the product of the first and the last factors. This allows us to bound this subterm by
\begin{align}
\label{exp:control:EM:mid:deltauhigh:5:M}
&\sup_{\MMonetwo}\big(\norm{\Psi}{k-n-2}\big)
\EFM{k-n}{p}{\tau_1,\tau_2}{\del_{\tau}^n\Psi}  .
\end{align}

For $m<n$, by using the Sobolev inequality {\eqref{eq:Sobolev:2}} and the fact $\frac{k+3}{2}-m\leq k-m-2$ since $k\geq n+7$, it holds 
\begin{align}
\label{eq:Sobolev:weakdecay:Wterm:deltau:M}
\sup_{\MMonetwo}\big(\norm{\del_{\tau}^m\Psi}{\frac{k+3}{2}-m}\big)^2\lesssim_{k} \sup_{\tau\in [\tau_1,\tau_2]}\Energy{k-m}{{1+\delta}}{ \tau}{\del_\tau^m\Psi} .
\end{align}
Also, by the Sobolev inequality {\eqref{eq:Sobolev:2}}, we have 
\begin{align}
\label{eq:Sobolev:weakdecay:deltauprim:M}
\sup_{\MMonetwo}\big(\norm{\Psi}{k-n-2}\big)\lesssim_{k} \sqrt{\sup_{\tau\in [\tau_1,\tau_2]}\Energy{k}{{1+\delta}}{ \tau}{\Psi }}
\lesssim_{\delta,k} \eps.
\end{align}
By plugging these Sobolev inequalities back into \eqref{exp:control:EM:mid:deltau:1:M} and \eqref{exp:control:EM:mid:deltauhigh:5:M}, we find the subterm $\Error{k-n}{1,L}{\del_{\tau}^n\Psi}$ is bounded by 
\begin{align}
\label{eq:iteration:deltau:n:Decay:Subterm1:M}
&\eps\EFM{k-n}{p}{\tau_1,\tau_2}{\del_{\tau}^n\Psi} 
+ \eps^{-1}\sum_{m=0}^{n-1}{\sup_{\tau\in [\tau_1,\tau_2]}\Energy{k-m}{{1+\delta}}{ \tau}{\del_\tau^m\Psi} }\cdot 
\EFM{k-n}{p}{\tau_1,\tau_2}{\del_{\tau}^{n-m}\Psi}\notag\\
\lesssim_{\delta,k}&\,\eps \EFM{k-n}{p}{\tau_1,\tau_2}{\del_{\tau}^n\Psi}
+\eps^{-1} \Energy{k}{2-\delta}{\tau_0}{\Psi}\sum_{0\leq  m \leq {n}-1}\la \tau_1-\tau_0\ra^{-(2-2\delta)m-1+\delta}\EFM{k-n}{p}{\tau_1,\tau_2}{\del_{\tau}^{n-m}\Psi}\notag\\
\lesssim_{\delta,k}&\,\eps \EFM{k-n}{p}{\tau_1,\tau_2}{\del_{\tau}^n\Psi}
+\eps\sum_{1\leq  m \leq {n}-1}\la \tau_1-\tau_0\ra^{-(2-2\delta)m-1+\delta}\EFM{k-n}{p}{\tau_1,\tau_2}{\del_{\tau}^{n-m}\Psi},
\end{align}
where we have used in the first step the assumed estimate \eqref{inductionLweakenerdec:deltaupsi:M}.
For the other subterm ${\Error{k-n}{1,H}{\del_{\tau}^n\Psi}}$, we again use integration by parts formula \eqref{Qterm:IBPFormula:M} to transform it into energy terms and spacetime integrals and find it bounded by the last line of \eqref{eq:iteration:deltau:n:Decay:Subterm1:M}.

In conclusion, we have for the error integral terms in \eqref{def:errorinte:globalrp:inhowave:deltau:M} that
\begin{align}
&\sum_{i=1,2,3}\Error{k-n}{i}{\del_{\tau}^n\Psi}
+\Error{k-n}{4,p}{\del_{\tau}^n\Psi}\notag\\
\lesssim_{\delta,k}&\,\eps \EFM{k-n}{p}{\tau_1,\tau_2}{\del_{\tau}^n\Psi}
+\eps\sum_{1\leq  m \leq {n}-1}\la \tau_1-\tau_0\ra^{-(2-2\delta)m-1+\delta}\EFM{k-n}{p}{\tau_1,\tau_2}{\del_{\tau}^{n-m}\Psi}.
\end{align}
We substitute this estimate back into \eqref{eq:globalrp:InhoWave:deltau:M}, arriving at
\begin{align}
\label{eq:iteration:deltau:n:kj:M}
\LEFM{k-n}{p}{\tau_1,\tau_2}{\del_{\tau}^n\Psi} 
\lesssim_{\delta,k} {}& 
\Energy{k-n}{p}{\tau_1}{\del_{\tau}^n\Psi} + \eps \EFM{k-n}{p}{\tau_1,\tau_2}{\del_{\tau}^n\Psi}\notag\\
&
+\eps\sum_{1\leq  m \leq {n}-1}\la \tau_1-\tau_0\ra^{-(2-2\delta)m-1+\delta}\EFM{k-n}{p}{\tau_1,\tau_2}{\del_{\tau}^{n-m}\Psi}.
\end{align}
By formula \eqref{def:EFM:Decay:M}  that says $\EFM{k-n}{p}{\tau_1,\tau_2}{\del_{\tau}^n\Psi}= \LEFM{k-n}{p}{\tau_1,\tau_2}{\del_{\tau}^n\Psi}+\Energy{k-n}{p}{\tau_1}{\del_{\tau}^n\Psi}$,  we can then take $\eps$ suitably small such that the following holds for any $n\leq {k-7}$:
\begin{align}
\label{eq:iteration:deltau:alln:Decay:M}
\EFM{k-n}{p}{\tau_1,\tau_2}{\del_{\tau}^n\Psi} 
\lesssim_{k,\delta} {}& 
\Energy{k-n}{p}{\tau_1}{\del_{\tau}^n\Psi} 
\notag\\
&
+\eps\sum_{1\leq  m \leq {n}-1}\la \tau_1-\tau_0\ra^{-(2-2\delta)m-1+\delta}\EFM{k-n}{p}{\tau_1,\tau_2}{\del_{\tau}^{n-m}\Psi}.
\end{align}

In the end, we prove the estimate \eqref{eq:globalrp:deltau:M} by induction in $n_0$. First, the case $n_0=0$ holds in view of the proven estimate \eqref{eq:globalrp:M}. Next, assume the estimate \eqref{eq:globalrp:deltau:M} holds for $n_0\leq n-1$, and we prove it for $n_0=n\leq k-7$. In fact, this is manifestly true once we substitute the estimate \eqref{eq:globalrp:deltau:M} for $n_0\leq n-1$ into the last term of \eqref{eq:iteration:deltau:alln:Decay:M}. This hence proves \eqref{eq:globalrp:deltau:M} for all $n_0\leq k-7$.

\underline{Step 2}.
Afterwards, by the proven inequality \eqref{eq:globalrp:deltau:M}, we have
\begin{align}
&\Energy{k-n_0}{p}{\tau_2}{\del_{\tau}^{n_0}\Psi} 
+ \int_{\tau_1}^{\tau_2}\Energy{k-n_0}{p-1}{\tau}{\del_{\tau}^{n_0}\varphi} d \tau
\notag\\
\lesssim_{\delta,k} {}&  \Energy{k-{n_0}}{p}{\tau_1}{\del_{\tau}^{n_0}\Psi} 
+\eps\sum_{0\leq j_0\leq n_0-1}\la \tau_1-\tau_0\ra^{-(2-2\delta)({n_0}-j_0)}\Energy{k-n_0}{p}{\tau_1}{\del_{\tau}^{j_0}\Psi}.
\end{align}
We apply Lemma \ref{lem:HierachyToDecay} to deduce for any $n_0\leq k-7$ that
\begin{align}
\label{eq:induction:highorder:deltau:pf:M}
&\Energy{k-n_0}{p}{\tau_2}{\del_{\tau}^{n_0}\Psi} 
\notag\\
\lesssim_{\delta,k} {}&  \la \tau_2-\tau_1'\ra^{-2+\delta +p}
\Energy{k-{n_0}}{2-\delta}{\tau_1'}{\del_{\tau}^{n_0}\Psi} 
+\eps\sum_{0\leq j_0\leq n_0-1} \la {\tau_2-\tau_1'}\ra^{-(2-2\delta)({n_0}-j_0)}\Energy{k-n_0}{p}{\tau_1'}{\del_{\tau}^{j_0}\Psi}\notag\\
\lesssim_{\delta,k} {}&  \la \tau_2-\tau_1'\ra^{-2+\delta +p}
\Energy{k-{n_0}}{2-\delta}{\tau_1'}{\del_{\tau}^{n_0}\Psi} \notag\\
&
+\eps\sum_{0\leq j_0\leq n_0-1} \la {\tau_2-\tau_1'}\ra^{-(2-2\delta)({n_0}-j_0)}\la \tau_1' -\tau_1\ra^{-2+\delta-(2-2\delta)j_0+p}\Energy{k-1}{2-\delta}{\tau_1}{\Psi},
\end{align}
where in the last step we have used the assumed estimate \eqref{inductionLweakenerdec:deltaupsi:M}.

By definition \eqref{def:globalrpEnergy:highorder:M}, we have for any $k_1\geq 0$ (we will take $k_1=k-n_0$) that
\begin{align}
\label{eq:EnerDecay:FasterDelTau:Step3:M}
&\Energy{k_1}{2-\delta}{\tau}{\del_{\tau}^{n_0}\Psi} \notag\\
=&\Wnorm{rV\del_{\tau}(\del_{\tau}^{{n_0}-1}\Psi)}{k_1}{-\delta}{\Sigma_{\tau}}^2
+\Wnorm{U\del_{\tau}^{{n_0}}\Psi}{k_1}{-1-\eta}{\Sigma_{\tau}}^2
+\Wnorm{\nablas\del_{\tau}^{{n_0}}\Psi}{k_1}{\peta}{\Sigma_{\tau}}^2 
+\Wnorm{\del_{\tau}^{{n_0}}\Psi}{k_1+1}{\peta}{\Sigma_{\tau}}^2 
\notag\\
\lesssim &\Wnorm{rVV(\del_{\tau}^{{n_0}-1}\Psi)}{k_1}{-\delta}{\Sigma_{\tau}}^2
+\Wnorm{{\la r\ra^2}UV(\del_{\tau}^{{n_0}-1}\Psi)}{k_1}{-2-\delta}{\Sigma_{\tau}}^2\notag\\
&
+\Wnorm{U\del_{\tau}^{{n_0}-1}\Psi}{k_1+1}{-1-\eta}{\Sigma_{\tau}}^2
+\Wnorm{\nablas\del_{\tau}^{{n_0}-1}\Psi}{k_1+1}{\peta}{\Sigma_{\tau}}^2 
+\Wnorm{\del_{\tau}^{{n_0}-1}\Psi}{k_1+2}{\peta}{\Sigma_{\tau}}^2 \notag\\
\lesssim &\Energy{k_1+1}{\delta}{\tau}{\del_{\tau}^{n_0-1}\Psi}
+\Wnorm{ r^3\del_{\tau}^{{n_0}-1}(\Box \psi)}{k_1}{-2-\delta}{\Sigma_{\tau}}^2,
\end{align}
where we have used in the second step $\del_{\tau}=\frac{1}{2} (U + V)$ and in the third step ${r^2 UV (\del_{\tau}^{{n_0}-1}\Psi)=\Deltas (\del_{\tau}^{{n_0}-1}\Psi) - r^3\del_{\tau}^{{n_0}-1}(\Box \psi)}$ which follows from equation \eqref{inhowave:RF:M}. For this last term, it satisfies
\begin{align*}
\abs{r^3\del_{\tau}^{{n_0}-1}(\Box \psi)}\lesssim &\sum_{m=0,1,\ldots, {n_0}-1}\norm{\del_{\tau}^{{n_0}-1-m}\Psi}{2} \norm{\del_{\tau}^m\Psi}{1},
\end{align*}
and consequently, by taking $k_1=k-n_0$, we deduce
\begin{align}
&\Wnorm{ r^3\del_{\tau}^{{n_0}-1}(\Box \psi)}{k-n_0}{-2-\delta}{\Sigma_{\tau}}^2\notag\\
\lesssim_{\delta,k}  & \sum_{\frac{{n_0}-1}{2}\leq m\leq {n_0}-1}\langle \tau-\tau_1\rangle^{-(2-2\delta)({n_0}-1-m)-1 +{2\delta}} \Energy{k}{2-\delta}{\tau_1}{\Psi}\cdot \Wnorm{\del_{\tau}^{m}\Psi}{k-n_0+2}{-2}{\Sigma_{\tau}}^2\notag\\
\lesssim_{\delta,k}  &\sum_{\frac{{n_0}-1}{2}\leq m\leq {n_0}-1}\langle \tau-\tau_1\rangle^{-(2-2\delta)({n_0}-1-m)-1 +{2\delta}} \Energy{k}{2-\delta}{\tau_1}{\Psi}\cdot \Energy{k-n_0+1}{\delta}{\tau}{\del_{\tau}^{m}\Psi} .
\end{align}
Here, in the first step, we have used {the estimate \eqref{eq:Sobolev:weakdecay:Wterm:deltau:M}}   and the energy decay estimate \eqref{inductionLweakenerdec:deltaupsi:M} to derive pointwise decay estimates.
We plug this estimate back into inequality \eqref{eq:EnerDecay:FasterDelTau:Step3:M} with $k_1=k-n_0$ and $\tau=\tau_1'$ and find
\begin{align}
&\Energy{k-n_0}{2-\delta}{\tau_1'}{\del_{\tau}^{n_0}\Psi} \notag\\
\lesssim_{\delta,k} &\Energy{k-n_0+1}{\delta}{\tau_1'}{\del_{\tau}^{n_0-1}\Psi}
+\eps^2 \sum_{\frac{{n_0}-1}{2}\leq m\leq {n_0}-1}\langle \tau_1'-\tau_1\rangle^{-(2-2\delta)({n_0}-1-m)-1 +\delta}  \Energy{k-n_0+1}{\delta}{\tau_1'}{\del_{\tau}^{m}\Psi} \notag\\
\lesssim_{\delta,k} &\la \tau_1'-\tau_1\ra^{-(2-2\delta)n_0} \Energy{k}{2-\delta}{\tau_1}{\Psi},
\end{align}
by using the assumed estimate \eqref{inductionLweakenerdec:deltaupsi:M} in the last step.
Afterwards, we substitute this estimate into \eqref{eq:induction:highorder:deltau:pf:M} and take $\tau_1'=\tau_1 + \frac{\tau_2-\tau_1}{2}$, arriving at
\begin{align}
\Energy{k-n_0}{p}{\tau_2}{\del_{\tau}^{n_0}\Psi} 
\lesssim_{\delta,k} & \la \tau_2-\tau_1\ra^{-2+\delta -(2-2\delta)n_0+p} \Energy{k}{2-\delta}{\tau_1}{\Psi}.
\end{align}
This thus proves inequality \eqref{inductionLweakenerdec:deltaupsi:M} for any $n'\leq k-7$. 
 In the end, the spacetime {integral} terms on the left-hand side of \eqref{weakenerdec:psi:full:M} are easily controlled by  combining inequality \eqref{inductionLweakenerdec:deltaupsi:M} for any $n'\leq k-7$ with inequality \eqref{eq:globalrp:deltau:M}.
\end{proof}

Next, we base on the energy decay estimates in Proposition \ref{prop:WeakEnerDecay} to conclude the following pointwise decay estimates:

\begin{proposition}[Weak pointwise decay estimates]
\label{prop:weakpointwise:M}
Let $k\geq 8$ and let $\delta \in (0,\min\{\frac12, 2\eta\})$. There exists a sufficiently small $\eps>0$ such that for any scalar $\psi$ solving the wave equation \eqref{Qwave:M} with null condition \eqref{nullcond:M} and satisfying
\begin{align}
\Energy{k}{2-\delta}{\tau_0}{\Psi}\leq \eps^2,
\end{align}
it holds true for any $n\leq k-8$ and $\tau\geq \tau_0$ that 
\begin{align}
\label{weakpointwise:M}
\norm{\del_{\tau}^n\psi}{(k-n-3)}\lesssim_{\delta, k} {}& \eps v^{-1}\tau^{-\frac{1}{2}-(1-\delta)n + \frac{3\delta}{2}}.
\end{align}
\end{proposition}

\begin{proof}
The pointwise decay estimates are proven by utilizing the proven energy decay estimates \eqref{weakenerdec:psi:full:M} together with the Sobolev inequalities in Lemma \ref{lem:Sobolev}. 

First, we obtain from \eqref{eq:Sobolev:2} with $q=\delta$ and \eqref{weakenerdec:psi:full:M} that  
\begin{align}
\label{weakpointwise:Sobolev2:M}
\norm{\del_{\tau}^n\Psi}{k-n-2}\lesssim_{\delta, k} {}& \tau^{-\frac{1}{2}-(1-\delta)n + \delta/2} \big(\Energy{k}{2-\delta}{\tau_0}{\Psi}\big)^{1/2}.
\end{align}
Next, we apply the Sobolev inequality \eqref{eq:Sobolev:3} to find
\begin{align}
\label{weakpointwise:Sobolev3:M}
\norm{\del_{\tau}^n\psi}{(k-n-3)}\lesssim_{\delta, k} {}& \big(\Wnorm{\del_{\tau}^n\Psi}{k-n}{\delta-3}{\MMtwoinfty} \Wnorm{\del_{\tau}^{n+1}\Psi}{k-n}{\delta-3}{\MMtwoinfty}\big)^{1/2}\notag\\
\lesssim_{\delta, k} {}&\big( \tau^{-(1-\delta)(n+1)}\tau^{-(1-\delta)(n+2)}\big)^{1/2} \big(\Energy{k}{2-\delta}{\tau_0}{\Psi}\big)^{1/2}\notag\\
\lesssim_{\delta, k} {}& \eps\tau^{-(2n+3)(1-\delta)/2}
\end{align}
for any $n+1\leq k-7$, that is, $n\leq k-8$. 
By combining the above estimates \eqref{weakpointwise:Sobolev2:M} and \eqref{weakpointwise:Sobolev3:M} and using $v\lesssim \max\{\tau, r\}$, we thus obtain the desired estimate \eqref{weakpointwise:M}.
\end{proof}

\begin{remark}
The above pointwise decay estimate \eqref{weakpointwise:M} for $\psi$ provides pointwise decay estimates also for its derivatives. As can be seen above, the pointwise decay rates are invariant under operations of $\{\tau^{1-\delta}\del_\tau, (r+1)V, \nablas\}$.
\end{remark}

%%%%%%%%%%%%%%%%%%%
%%%%%%%%%%%%%%%%%%%

\section{Sharp pointwise decay and leading-order term}
\label{sect:sharpdecay}

The goal of this section is to extract the leading-order term in the asymptotics for the solution $\psi$. 
We now provide the statement of the main result.

\begin{theorem}
\label{thm:leadingorder:Qwave}
Consider the quasilinear wave equation \eqref{Qwave:M} with null condition \eqref{nullcond:M}, and let $k \geq 20$ be an integer. Assume the initial energy bounds 
\begin{equation}
\Energy{k}{p=2-\delta}{\tau_0}{\Psi} 
+
\Energy{k-10}{p=3-\delta}{\tau_0}{\Psi_{\ell=0}} 
+
\Energy{k-10}{p=1+6 \delta}{\tau_0}{r^2V\Psi_{\ell\geq 1}}
<
\varepsilon^2,
\end{equation}
for $0<\varepsilon\ll 1$ sufficiently small and $ \delta\ll 1/2$ small enough.
We additionally assume that there exist constants $c_{init}$ and $D>0$ such that
\begin{equation}
\label{assump:initialdataasymp}
\Big| {1\over 4\pi}\big(r^2 V \int_{\mathbb{S}^2}\Psi \, d^2\mu\big)\big|_{\Sigma_{\tau_0}}  - c_{init} \Big|
\leq D r^{-\delta}.
\end{equation}
Then there exists a constant $C>0$ which depends only on $\delta$ such that the precise pointwise asymptotic for the solution reads
\begin{equation}
\big|\psi - c_{total} u^{-1} v^{-1} \big|
\leq C\big( D + c_{total} + \varepsilon \big) v^{-1}u^{-1-\delta},
\end{equation}
in which 
\begin{equation}
c_{total} = c_{init} + c_{\II},
\qquad\quad
\text{with   \,\, } c_{\II} =2\int_{\II_{\geq \tau_0}} r^3 (-\Box \psi) \, d^2 \mu d\tau.
\end{equation}
\end{theorem}

We denote the nonlinearities in the equation \eqref{Qwave:M} by 
\begin{equation*}
P(\partial \psi,\partial^{2}\psi) :=P^{\alpha\beta \gamma} \partial_\gamma \psi \partial_{\alpha}\del_{\beta} \psi.
\end{equation*}
We first prepare some useful estimates for the nonlinearities $P(\partial \psi,\partial^{2}\psi)$.

\begin{lemma}[Estimate for null forms]\label{lem:est-null}
We have for $k\leq 8$ that 
\begin{equation}\label{eq:est-null}
\aligned
\big|P(\partial \psi,\partial^{2}\psi) \big|_{(k)}
\lesssim
|\partial_t \psi \partial_t \underline{\partial} \psi|_{(k)} 
+
|\underline{\partial} \psi \partial^{2}_t  \psi|_{(k)} 
+ 
|\partial_t \psi \underline{\partial}^{2} \psi|_{(k)} 
+
|\underline{\partial} \psi \partial_t \underline{\partial} \psi|_{(k)}
+
|\underline{\partial} \psi \underline{\partial}^{2} \psi|_{(k)},
\endaligned
\end{equation}
in which $\underline{\partial} \in \{V, r^{-1} \Omega_a \}$. 
In particular, for $r\geq 1$ we have
\begin{equation}\label{eq:est-null2}
\big|P(\partial \psi,\partial^{2}\psi) \big|_{(k)}
\lesssim
r^{-3} | \mathbb{D}^{\leq 1} \Psi \partial_t \mathbb{D}^{\leq 1} \Psi |_{k} 
+
r^{-4} | \mathbb{D}^{\leq 1} \Psi \mathbb{D}^{\leq 2} \Psi |_{k}. 
\end{equation}
\end{lemma}
\begin{proof}
We express $(\partial_\alpha)$ in terms of the vector fields $(\partial_t, V, r^{-1} \Omega_a)$, and then the null condition on $P^{\alpha\beta\gamma}$ gives the desired result.
\end{proof}

\begin{lemma}
	For $k\leq 8, n\leq 3$, we have
	\begin{equation}\label{eq:NL-pw}
\aligned
\left|\partial_\tau^n P(\partial \psi,\partial^{2}\psi) \right|_{(k)}
\lesssim_{\delta}  \varepsilon^2 \tau^{(n+3)(-1+\delta)+2\delta} \langle r\rangle^{-3}.
\endaligned
\end{equation}
\end{lemma}
\begin{proof}
From Proposition \ref{prop:weakpointwise:M}, we get
$$
|\partial_\tau^n \underline{\partial} \psi|_{(k=8)} 
\lesssim_{\delta}  \varepsilon \langle r\rangle^{-1} v^{-1} \tau^{-1/2+3\delta/2 + n(-1+\delta)}, \,\, | \partial_\tau^n \underline{\partial}^2 \psi|_{(k=8)}\lesssim_{\delta}  \varepsilon \langle r\rangle^{-2} v^{-1} \tau^{-1/2+3\delta/2 + n(-1+\delta)}
$$
for $\underline{\partial} \in \{V, r^{-1} \Omega_a \}$, and
$$
|\partial_\tau^n \psi|_{(k=10)}
\lesssim_{\delta} 
\varepsilon v^{-1} \tau^{-1/2+3\delta/2 + n(-1+\delta)}.
$$
Thus inserting these into \eqref{eq:est-null}, together with the fact $v^{-1}\leq \tau^{-1}$, yields the desired results.
\end{proof}

\subsection{Pointwise decay for $\ell=0$ model}
\label{subsect:ptw:ell=0mode}

In this subsection, we consider the case of the $0$-th mode of the radiation field $\Psi$, which we recall is given by 
$$
\Psi_{\ell=0} = {1\over 4\pi}\int_{\mathbb{S}^2} \Psi \, d^2\mu
$$ 
and satisfies the equation
\begin{equation}\label{eq:0-mode}
U V \Psi_{\ell=0} = r (-\Box \psi)_{\ell=0}.
\end{equation}
For easy readability, we write the equation for $rV \Psi_{\ell=0}$, which is obtained by acting $rV$ to both sides of \eqref{eq:0-mode}
\begin{equation}\label{eq:0-mode2}
U V (rV \Psi_{\ell=0}) + {V\over r} (rV \Psi_{\ell=0}) = {(1+rV) \big( r (-\Box \psi)_{\ell=0}\big)} + {1\over r^2} (rV \Psi_{\ell=0}),
\end{equation}
as well as the one for $(rV)^2 \Psi_{\ell=0}$, which is obtained by acting $rV$ to \eqref{eq:0-mode2}
\begin{equation}
\aligned
&U V ((rV)^2 \Psi_{\ell=0}) + {2V\over r} ((rV)^2 \Psi_{\ell=0}) 
\\
= &{\big(1+ 2rV + (rV)^2\big)\big( r (-\Box \psi)_{\ell=0}\big)} + {3\over r^2} ((rV)^2 \Psi_{\ell=0}) -{1\over r^2} (rV \Psi_{\ell=0}).
\endaligned
\end{equation}
Since $\Psi_{\ell=0}$ is radially symmetric, the global $r^p$ energy functional takes the form (see Lemma \ref{lem:globalrp-radial} below)
\begin{equation}
\Energy{k}{p}{\tau}{\Psi_{\ell=0}} 
=\Wnorm{rV\Psi_{\ell=0}}{k}{p-2}{\Sigma_{\tau}}^2
+\Wnorm{U\Psi_{\ell=0}}{k}{-1-\eta}{\Sigma_{\tau}^{}}^2
+\Wnorm{\Psi_{\ell=0}}{k+1}{\peta}{\Sigma_{\tau}}^2.
\end{equation}

The following basic energy and Morawetz estimates will play an important role in the finite $r$ region, and it is the radial case of Proposition \ref{prop:HighEnerMora:psi:M}.
\begin{proposition}[{High-order energy–Morawetz estimate}]\label{prop:BEAM-radial}
Assume the scalar function $\varphi$ with radial symmetry satisfies an inhomogeneous wave equation
\begin{align}
\label{eq:basicEnerMora-radial}
UV\varphi =\vartheta,
\end{align} 
then we have
\begin{equation}\label{eq:BEAM-radial}
\aligned
&\SEnergy{k}{\tau_2}{\varphi}
+\sum_{\mathbb{X}\in\{ \partial_t, V \}, |\alpha|\leq k} \int_{\MMonetwo} {|\mathbb{X}^\alpha \partial_t \varphi|^2 + |\mathbb{X}^\alpha \partial_r \varphi|^2 + |r^{-1} \mathbb{X}^\alpha \varphi|^2 \over (1+r)^{1+\delta}}  \,  d^4\mu 
\\
\lesssim 
&\SEnergy{k}{\tau_1}{\varphi} 
+
\sum_{\mathbb{X}\in\{ \partial_t, V \},  |\alpha|\leq k} \bigg|\int_{\MMonetwo}\del_t(\mathbb{X}^{\alpha}\varphi)\cdot \mathbb{X}^\alpha\vartheta \, d^4\mu \bigg|
\\
+
&\sum_{\mathbb{X}\in\{ \partial_t, V \}, |\alpha|\leq k} \bigg|\int_{\MMonetwo}\big(2 - (1+r)^{-\delta}\big)\del_r(\mathbb{X}^{\alpha}\varphi)\cdot \mathbb{X}^\alpha\vartheta \, d^4\mu \bigg|,
\endaligned
\end{equation}
in which
$$
\aligned
\SEnergy{k}{\tau}{\varphi}
=
&\sum_{\mathbb{X}\in\{ \partial_t, V \},  |\alpha|\leq k} \int_{\Sigma_\tau} \Big( \hh' |V\mathbb{X}^\alpha \varphi|^2 + (1-\hh') \big( |\partial_r \mathbb{X}^\alpha \varphi|_{k, \{\partial_t, V \}}^2 + (\partial_t \mathbb{X}^\alpha \varphi)^2\big) + r^{-2} |\mathbb{X}^\alpha\varphi|^2 \Big) \, d^3\mu,
\\
\simeq&\sum_{\mathbb{X}\in\{ \partial_t, V \},  |\alpha|\leq k} \int_{\Sigma_\tau} \Big( |\overline{\partial}_r \mathbb{X}^\alpha\varphi|^2 + (1-(\hh')^2) \big| \overline{\partial}_\tau \mathbb{X}^\alpha \varphi \big|^2 + r^{-2} |\mathbb{X}^\alpha \varphi|^2 \Big) \, d^3\mu.
\endaligned
$$
\end{proposition}

\begin{remark}
We note that Proposition \ref{prop:BEAM-radial} differs from Proposition \ref{prop:HighEnerMora:psi:M}. In Proposition \ref{prop:BEAM-radial}, the solution $\Psi_{\ell=0}$ has radial symmetry, and the bounds for the source terms in \eqref{eq:BEAM-radial} are simpler as it suffices to apply only the operators $\partial_t, V$ to equation \eqref{eq:basicEnerMora-radial}.
\end{remark}

Next, we restate the global $r^p$ estimates for an inhomogeneous wave equation with radial symmetry, where the range of admissible $p$ is larger than the non-radial case of Theorem \ref{thm:globalrp:InhoWave:M}. This is obtained by simply applying Lemma \ref{lem:rpInfinity-radial} and Proposition \ref{prop:BEAM-radial} to equation \eqref{eq:globalrp-radial}.

\begin{lemma}[Global $r^p$ estimates an inhomogeneous wave equation with radial symmetry]
\label{lem:globalrp-radial}
Fix $\tau_0\geq 1$.
Let $\delta \in (0,\frac12)$ be a small constant. 
Let $\Psi_{\ell=0}$ be the $0$-th mode of $\Psi$ which satisfies the equation
\begin{align}
\label{eq:globalrp-radial}
-UV\Psi_{\ell=0} =r(-\Box \psi)_{\ell=0}.
\end{align} 

Then, for any $p\geq \delta, k\in \mathbb{N}$, there exist constants $R_0=R_0(\delta, k) >10$ and $C_0=C_0(\delta, k)>0$ such that for all $\tau_2>\tau_1\geq \tau_0$, 
\begin{align}
\label{eq:globalrp-radial-EE}
&\Energy{k}{p}{\tau_2}{\Psi_{\ell=0}} + \int_{\tau_1}^{\tau_2} \Energy{k}{p-1}{\tau}{\Psi_{\ell=0}} \, d\tau
\\
\leq {}& C_0\bigg(
\Energy{k}{p}{\tau_1}{\Psi_{\ell=0}} + \sum_{i=1,2,3}\Error{[k]}{i}{\Psi_{\ell=0}}+\Error{[k]}{4,p}{\Psi_{\ell=0}}\bigg),
\end{align}
where the error integral terms are defined by 
\begin{subequations}
\label{def:errorinte:globalrp-radial}
\begin{align}
\label{errorinte:globalrp-radial}
\Error{[k]}{1}{\Psi_{\ell=0}}:=&\sum_{\mathbb{X}\in\{ \partial_t, V \},  |\alpha|\leq k} \bigg|\int_{\MMonetwo}\del_t(\mathbb{X}^{\alpha}\Psi_{\ell=0})\cdot \mathbb{X}^\alpha\big( r(\Box \psi)_{\ell=0}\big) \, d^4\mu \bigg|,\\
\label{errorinte:globalrp-radial}
\Error{[k]}{2}{\Psi_{\ell=0}}:=&
\sum_{\mathbb{X}\in\{ \partial_t, V \}, |\alpha|\leq k} \bigg|\int_{\MMonetwo}\big(2 - (1+r)^{-\delta}\big)\del_r(\mathbb{X}^{\alpha}\Psi_{\ell=0})\cdot \mathbb{X}^\alpha\big( r(\Box \psi)_{\ell=0}\big) \, d^4\mu \bigg|,\\
\label{errorinte:globalrp-radial}
\Error{[k]}{3}{\Psi_{\ell=0}}:=&\sum_{\mathbb{X}\in\{ \partial_t, rV \}, |\alpha|\leq k} \bigg|\int_{\MMonetwo^{\geq R_0-1}} \frac{\chi^2(r)  (1+r^{-\frac{\delta}{2}})}{{r}}U \mathbb{X}^{\alpha} \Psi_{\ell=0} \cdot \mathbb{X}^{\alpha} \big(({r^2}\Box\psi)_{\ell=0}\big) \, d^4\mu\bigg|,\\
\label{errorinte:globalrp-radial}
\Error{[k]}{4,p}{\Psi_{\ell=0}}:=&\sum_{\mathbb{X}\in\{ \partial_t, rV \}, |\alpha|\leq k} \bigg|\int_{\MMonetwo^{\geq R_0-1}} \chi^2(r) r^{{p-1}} V \mathbb{X}^{\alpha} \Psi_{\ell=0} \cdot \mathbb{X}^{\alpha} \big(({r^2}\Box\psi)_{\ell=0} \big)\, d^4\mu\bigg|,
\end{align}
\end{subequations}
where $\chi(r):=\chi_1(r-R_0)$ with $\chi_1(x)$ being a standard smooth cutoff function that equals $1$ for $x\geq 0$ and vanishes identically for $x\leq -1$.
The above energy functional for a radially symmetric function $\varphi$ takes the form
\begin{align}
\label{def:globalrpEnergy-radial}
\Energy{k}{p}{\tau}{\varphi}:=  \Wnorm{rV\varphi}{k}{p-2}{\Sigma_{\tau}}^2
+\Wnorm{U\varphi}{k}{-1-\eta}{\Sigma_{\tau}}^2
+\Wnorm{\varphi}{k+1}{\peta}{\Sigma_{\tau}}^2.
\end{align}
\end{lemma}
We note for $p\leq 1$ it holds that
\begin{equation}
\|\overline{\partial}_r \varphi\|_{W^{k=0}_{p}}^2
+
\| \varphi\|_{W^{k=0}_{p-2}}^2
\lesssim
\Energy{0}{p}{\tau}{\varphi}.
\end{equation}

We now establish the following integrated energy decay estimate result.

\begin{proposition}\label{prop:radial-E-decay01}
We have for $k\leq 6, n\leq 3, p\in [\delta, 3-\delta], \tau_2>\tau_1 \geq \tau_0$ that
\begin{equation}
\label{eq:globalrp:withdecay:Psiz:M}
\aligned
&\Energy{k}{p}{\tau_2}{\partial_\tau^n \Psi_{\ell=0}} + \int_{\tau_1}^{\tau_2} \Energy{k}{p-1}{\tau}{\partial_\tau^n \Psi_{\ell=0}} \, d\tau
\\
\leq 
&C \Energy{k}{p}{\tau_1}{\partial_\tau^n \Psi_{\ell=0}}
+
{C \delta^{-2} \over 3-p} \epsilon^4 \tau_1^{(2n+4)(-1+\delta)+7\delta}.
\endaligned
\end{equation}
\end{proposition}
\begin{proof}
Based on the integrated energy decay estimates in Lemma \ref{lem:globalrp-radial}, we only need to bound the source terms.

\paragraph{Step 1.}

For the lower order regularity case, we recall
$$
UV\Psi_{\ell=0} = r (-\Box \psi)_{\ell=0}
$$
and start with bounding $\Error{[0]}{p}{\Psi_{\ell=0}}$ by
$$
\aligned
\Error{[0]}{4,p}{\Psi_{\ell=0}}
\leq
& \int_{\tau_1}^{\tau_2} \int_{\Sigma_\tau} \big| r (\Box \psi)_{\ell=0} r^p V\Psi_{\ell=0} \big| \, d^4 \mu
\\
\leq
&\int_{\tau_1}^{\tau_2}  \int_{\Sigma_{\tau}} \Big(\delta^{-2}  \tau^{1+\delta} r^{p+2} \big| \Box \psi \big|^2 + \delta^2 \tau^{-1-\delta} r^p | V\Psi_{\ell=0} |^2 \Big) \, d^4 \mu.
\endaligned
$$
We recall from \eqref{eq:NL-pw} that
\begin{align}
\label{eq:L2boundofBoxpsi:weak:M}
\big| \Box \psi \big|^2 = \big|P(\del\psi, \del^2\psi)\big|^2
\lesssim
\epsilon^{4} \langle r\rangle^{-6} \tau^{-6+10\delta}.
\end{align}
Hence, we get for $0<p<3$ that
$$
\aligned
\Error{[0]}{4,p}{\Psi_{\ell=0}}
\leq
& \int_{\tau_1}^{\tau_2} \int_{\Sigma_{\tau}}  \delta^2 \tau^{-1-\delta} r^p | V\Psi_{\ell=0} |^2 \, d^4 \mu
+
{C\over 3-p} \delta^{-2} \epsilon^4 \tau_1^{-4+11\delta}
\\
\leq 
&\delta \max_{\tau_1 \leq \tau \leq \tau_2} \int_{\Sigma_{\tau}}  r^p | V\Psi_{\ell=0} |^2 \, d^3 \mu 
+
{C\over 3-p} \delta^{-2} \epsilon^4 \tau_1^{-4+11\delta}.
\endaligned
$$

We continue to bound $\Error{[0]}{1}{\Psi_{\ell=0}}$, and we first look at the integrand
$$
\aligned
2 \big| r(\Box \psi)_{\ell=0} \partial_t \Psi_{\ell=0} \big|
\leq
\delta^{-1} \langle r\rangle^2 \big| r\Box \psi \big|^2 + \delta \langle r\rangle^{-2} \big( |V\Psi_{\ell=0}|^2 +  |U \Psi_{\ell=0}|^2\big).
\endaligned
$$
We recall from \eqref{eq:NL-pw}  that
$$
\aligned
\langle r\rangle^2 \big| r\Box \psi \big|^2
\lesssim
\epsilon^4 \langle r\rangle^{-2} \tau^{-6+10\delta},
\endaligned
$$
which is integrable in $r$.
In summary, we get
\begin{align*}
\Error{[0]}{1}{\Psi_{\ell=0}}
\leq
&\int_{\tau_1}^{\tau_2} \int_{\Sigma_\tau}  2 \big| r(\Box \psi)_{\ell=0} \partial_t \Psi_{\ell=0} \big|  \, d^4 \mu 
\notag\\
\leq
&\int_{\tau_1}^{\tau_2} \int_{\Sigma_\tau} \Big( \delta \langle r\rangle^{-2} \big( |V\Psi_{\ell=0}|^2 +  |U \Psi_{\ell=0}|^2\big) + C \delta^{-1} \epsilon^4 \langle r\rangle^{-2} \tau^{-6+10\delta} \Big) \, d^4 \mu
\notag\\
\leq
& C \delta \int_{\tau_1}^{\tau_2} \Energy{0}{p-1}{\tau}{\Psi_{\ell=0}} \, d\tau + C \delta^{-1} \epsilon^4 \tau_1^{-4}.
\end{align*}
In the same way, we obtain that 
$$
\Error{[0]}{2}{\Psi_{\ell=0}}
+\Error{[0]}{3}{\Psi_{\ell=0}}
\leq
 C \delta \int_{\tau_1}^{\tau_2} \Energy{0}{p-1}{\tau}{\Psi_{\ell=0}} \, d\tau + C \delta^{-1} \epsilon^4 \tau_1^{-4}.
$$

Gathering these estimates, we have
$$
\aligned
& \max_{\tau\in [\tau_1, \tau_2]} \Energy{0}{p}{\tau_2}{\Psi_{\ell=0}} + \int_{\tau_1}^{\tau_2} \Energy{0}{p-1}{\tau}{\Psi_{\ell=0}} \, d\tau
\\
\leq
& C \delta \int_{\tau_1}^{\tau_2} \Energy{0}{p-1}{\tau}{\Psi_{\ell=0}} \, d\tau  + \delta \max_{\tau\in [\tau_1, \tau_2]} \Energy{0}{p}{\tau}{\Psi_{\ell=0}} + {C \delta^{-2} \over 3-p} \epsilon^4 \tau_1^{-4+11\delta},
\endaligned
$$
and the smallness of $\delta$ yields the desired result.

\paragraph{Step 2.}
In a very similar way, we can derive the integrated energy estimates for higher-order regularity cases based on the pointwise estimates in \eqref{eq:NL-pw}.
\end{proof}

\begin{proposition}[Decay of energy]\label{prop:radial-E-decay1}
For $k\leq 6, n\leq 3, p \in [\delta, 3-\delta], \tau_2 > \tau_1 \geq \tau_0$, we have 
\begin{equation}\label{eq:energy-decay-radial}
\aligned
&\Energy{k}{p}{\tau_2}{\partial_\tau^n \Psi_{\ell=0}} + \Wnorm{\partial_\tau^n \Psi_{\ell=0}}{k+1}{p-3}{\MMtwoinfty}^2
+\Wnorm{U\partial_\tau^n \Psi_{\ell=0}}{k}{-1-\delta/2}{\MMtwoinfty}^2
\\
\lesssim_{\delta} &\la \tau_2 - \tau_1\ra^{-(3-\delta)+p} \Energy{k}{3-\delta}{\tau_1}{\partial_\tau^n \Psi_{\ell=0}} + \varepsilon^4 \la \tau_2 - \tau_1\ra^{(2n+4)(-1+\delta)+7\delta}.
\endaligned
\end{equation}
\end{proposition}
\begin{proof}
An application of Lemma \ref{lem:HierachyToDecay} to the hierarchy of estimates \eqref{eq:globalrp:withdecay:Psiz:M} yields the desired results.
\end{proof}

\begin{proposition}[Decay of energy II]\label{prop:energy-decay-radial2}
For $k\leq 5, n\leq 3, \tau \geq \tau_0$, we have 
\begin{equation}\label{eq:energy-decay-radial2}
\aligned
&\Energy{k}{p}{\tau}{\partial_\tau^n \Psi_{\ell=0}} + \Wnorm{\partial_\tau^n \Psi_{\ell=0}}{k+1}{p-3}{\MM_{\tau,\infty}}^2
+\Wnorm{U\partial_\tau^n \Psi_{\ell=0}}{k}{-1-\delta/2}{\MM_{\tau,\infty}}^2
\\
\lesssim_{\delta} &\, \varepsilon^2 \tau^{-(3-\delta)-2n+p}, 
\qquad
p \in [\delta, 3-\delta].
\endaligned
\end{equation}
\end{proposition}
\begin{proof}
From the equation
$$
VU\Psi_{\ell=0} = r(-\Box \psi)_{\ell=0},
$$
we deduce
$$
2V\partial_t \Psi_{\ell=0} = {1\over r^2} (rV)^2 \Psi_{\ell=0} - {1\over r^2} (rV\Psi_{\ell=0}) + r(-\Box \psi)_{\ell=0},
$$
which gives
$$
\aligned
2r^{3/2-\delta/2} V\partial_t \Psi_{\ell=0} 
= & r^{-1/2-\delta/2} (rV)^2 \Psi_{\ell=0} - r^{-1/2-\delta/2} (rV\Psi_{\ell=0}) + r^{5/2-\delta/2} (-\Box \psi)_{\ell=0}.
\endaligned
$$

Thus we get from \eqref{eq:energy-decay-radial} that
$$
\aligned
\Energy{0}{p}{\tau}{\partial_t\Psi_{\ell=0}} 
\lesssim_{\delta}
&\, \tau^{-3+\delta+p}  \Energy{0}{p=3-\delta}{\frac{\tau}{2}}{\partial_t\Psi_{\ell=0}} + \varepsilon^4 \tau^{-6+13\delta} 
\\
\lesssim_{\delta}
&\, \tau^{-3+\delta+p} \big( \Energy{1}{p=1-\delta}{\frac{\tau}{2}}{\Psi_{\ell=0}}   + \| r^{5/2-\delta/2}(-\Box \psi)_{\ell=0} \|^2_{W^{k=0}_{p=0}(\Sigma_{\frac{\tau}{2}})} \big) + \varepsilon^4 \tau^{-6+13\delta}
\\
\lesssim_{\delta}
&\, \varepsilon^4 \tau^{-6+13\delta} + \tau^{-5+\delta+p} \Energy{1}{p=3-\delta}{\frac{\tau}{4}}{\Psi_{\ell=0}}\\
\lesssim_{\delta}
&\, \varepsilon^4 \tau^{-6+13\delta} + \tau^{-5+\delta+p} \Energy{1}{p=3-\delta}{\tau_0}{\Psi_{\ell=0}},
\endaligned
$$
in which we used the following bound that follows from \eqref{eq:L2boundofBoxpsi:weak:M}
$$
\| r^{5/2-\delta/2}(-\Box \psi)_{\ell=0} \|^2_{W^{k=0}_{p=0}}
\leq
\| r^{5/2-\delta/2}\Box \psi \|^2_{W^{k=0}_{p=0}}
\lesssim_{\delta}
\varepsilon^4 \tau^{-6+10\delta}.
$$
Hence, we have
\begin{equation}
\aligned
\Energy{0}{p}{\tau}{\partial_t\Psi_{\ell=0}} 
\lesssim_{\delta} 
\varepsilon^2 \tau^{-5+\delta+p},
\qquad
p \in [\delta, 3-\delta].
\endaligned
\end{equation}

Repeating the above argument leads us to the higher-order regularity cases in \eqref{eq:energy-decay-radial2}.
\end{proof}

\begin{proposition}[Pointwise decay] \label{prop:decay-01}
We have for $k\leq 3, n\leq 2$ that   
\begin{align}
|\partial_t^n \psi_{\ell=0}|_{(k)} \lesssim_{\delta} \varepsilon  v^{-1} \tau^{-1-n + \delta}. \label{eq:decay-01}
\end{align}
\end{proposition}
\begin{proof}

We recall from Proposition \ref{prop:energy-decay-radial2} that for $k\leq 3, n\leq 2$,
$$
\Energy{k+2}{p=1+\delta}{\tau}{\partial_t^n \Psi_{\ell=0}} \lesssim_{\delta} \varepsilon^2 \tau^{-2-2n+2\delta},
\qquad
\Energy{k+2}{p=1-\delta}{\tau}{\partial_t^n \Psi_{\ell=0}} \lesssim_{\delta} \varepsilon^2 \tau^{-2-2n},
$$
and we apply the Sobolev inequality \eqref{eq:Sobolev:2} with $q=\delta$ to get
\begin{equation}\label{eq:decay001}
\aligned
|\partial_t^n \Psi_{\ell = 0}|_{k}
\lesssim_{\delta} 
\varepsilon   \tau^{-1 - n +\delta/2}.
\endaligned
\end{equation}

On the other hand, the Sobolev inequality \eqref{eq:Sobolev:3} gives that
\begin{equation}
\aligned
|\partial_t^n \psi_{\ell=0}|_{(k)}
\lesssim
&\Big(\Wnorm{\partial_t^n\Psi_{\ell=0}}{3}{-3}{\MM_{\tau,\infty}}^2
\Wnorm{\del_{\tau}\partial_t^{n+1} \Psi_{\ell=0}}{3}{-3}{\MM_{\tau,\infty}}^2\Big)^{\frac{1}{4}}
\lesssim_{\delta} 
\varepsilon \tau^{-2-n+\delta},
\endaligned
\end{equation}
in which we used the estimates in \eqref{eq:energy-decay-radial2} in the last step.

Finally, the fact $v \lesssim (2r+\tau)$ yields \eqref{eq:decay-01}.
\end{proof}

Consequently, we can benefit extra $\tau$-decay when the solution $\psi_{\ell=0}$ is acted by $\partial_r, \overline{\partial}_r$.
\begin{lemma}
The following estimates hold
\begin{equation}\label{eq:partial-rho}
|\partial_r \psi_{\ell=0}| + |\overline{\partial}_r \psi_{\ell=0}|
\lesssim_{\delta} \varepsilon   v^{-1} \tau^{-2+3\delta},
\qquad
r\leq {t\over 2}.
\end{equation}
\end{lemma}
\begin{proof}
We first show  
\begin{equation}\label{eq:partial-r}
|\partial_r \psi_{\ell=0}| \lesssim_{\delta} \varepsilon  \tau^{-3+3\delta},
\qquad
r\leq {t\over 2}.
\end{equation}
We recall that the equation of $\psi_{\ell=0}$ reads
$$
\partial_r \partial_r \psi_{\ell=0} + {2\over r} \partial_r \psi_{\ell=0} = \partial_t \partial_t \psi_{\ell=0} + (\Box \psi)_{\ell=0}, 
$$
which gives
$$
\partial_r (r^2 \partial_r \psi_{\ell=0}) = r^2 [\partial_t \partial_t \psi_{\ell=0} + (\Box \psi)_{\ell=0}].
$$
We note
$$
\big|r^2 [\partial_t \partial_t \psi_{\ell=0} + (\Box \psi)_{\ell=0}] \big|
\lesssim_{\delta}
\varepsilon  r \tau^{-3+\delta}.
$$
Thus we integrate from $(t, r=0)$ to $(t, r=r'\leq t/2)$ get
$$
(r^2 \partial_r \psi_{\ell=0})|\vert_{r=r'}
\lesssim
\int_0^{r'} \big|r^2 [\partial_t \partial_t \psi_{\ell=0} + (\Box \psi)_{\ell=0}] \big| \, dr
\lesssim_{\delta}
\varepsilon  (r')^2 t^{-3+3\delta}
\lesssim_{\delta}
\varepsilon   (r')^2 \tau^{-3+3\delta},
$$
which yields \eqref{eq:partial-r}.
This together with \eqref{relation:partialderis:diffcoords} then yields \eqref{eq:partial-rho}.
\end{proof}

The following intermediary result for $V\Psi$ will be used to prove Proposition \ref{prop:decay3}.
\begin{lemma}\label{lem:V-Psi}
 For $n\leq 1, \delta_1=20\delta \ll 1/10$ we have
\begin{equation}\label{eq:V-Psi}
 \big|V \partial_t^n \mathbb{D}^{\leq 2} \Psi\big|
\lesssim_{\delta} (\varepsilon + D+c_{init}) v^{-2+2\delta_1} \tau^{(1-n)/2+3\delta/2},
\qquad\quad
\text{for   } \,\, r \geq v^{1-2\delta_1}.
\end{equation}
\end{lemma}

\begin{proof}
We first recall
$$
UV\Psi = r^{-2} {\Delta\mkern-11.5mu/} \Psi + r(-\Box \psi).
$$
By the pointwise bounds in Proposition \ref{prop:weakpointwise:M}, we get
$$
| r^{-2} {\Delta\mkern-11.5mu/} \Psi |
=
| r^{-1} {\Delta\mkern-11.5mu/} \psi |
\lesssim_{\delta}
\varepsilon r^{-1} v^{-1} u^{-1/2+3\delta/2}
\lesssim_{\delta}
\varepsilon v^{-2+2\delta_1} u^{-1/2+3\delta/2},
\qquad
\text{for   } r \geq v^{1-2\delta_1}
$$
and 
$$
|r(-\Box \psi)|
\lesssim_{\delta}
\varepsilon^2 v^{-2} u^{-2}.
$$

Thus for any $(u, v)$ such that $r(u, v) \geq v^{1-2\delta_1}$, we have
$$
\aligned
|v^2 V\Psi(u, v) - v^2 V\Psi(u_{\tau_0}(v), v)|
\lesssim
&\int_{u_{\tau_0}(v)}^{u} v^2|UV\Psi(u', v)| \, du'
\\
\lesssim
&\int_{u_{\tau_0}(v)}^{u} \Big( | v^2 r^{-2} {\Delta\mkern-11.5mu/} \Psi | +  |v^2 r(-\Box \psi)|   \Big) \, du'
\\
\lesssim_{\delta}&\varepsilon \int_{u_{\tau_0}(v)}^{u} \Big(  v^{2\delta_1} (u')^{-1/2+3\delta/2} + (u')^{-2}  \Big)  \, du'
\\
\lesssim_{\delta}
&\varepsilon  u^{1/2+3\delta/2} v^{2\delta_1},
\endaligned
$$
in which $u_{\tau_0}(v)$ is such that $(u_{\tau_0}(v), v, \omega)\in\Sigma_{\tau_0}$.
Recall $|V\Psi(u_{\tau_0}(v), v)| \lesssim_{\delta} (c_{init} +D)v^{-2+2\delta_1}$, and we obtain
\begin{equation}
|V\Psi(u, v)| \lesssim_{\delta} (\varepsilon +D+ c_{init}) v^{-2+2\delta_1} u^{1/2+3\delta/2},
\qquad
\text{for   } r\geq v^{1-2\delta_1}.
\end{equation}

In a similar way, we can derive the estimates for $n=1$ case in \eqref{eq:V-Psi}.
\end{proof}

\begin{proposition}[Pointwise decay III] \label{prop:decay3}
In the spacetime region $\{ r \geq v^{1-2\delta_1} \}$ it holds that
\begin{equation}
| V\Psi_{\ell=0} - c_{total} v^{-2}| \lesssim_{\delta} \big( D + c_{init}+\varepsilon\big) v^{-2} \tau^{-\delta},
\end{equation}
in which $c_{total} = c_{init} + c_{\II}$ with $c_{\II} = 2\int_{\II_{\geq \tau_0}} r^3 (-\Box \psi) \, d^2 \mu d\tau$ being finite.
\end{proposition}

\begin{proof}
We note   
$$
U \big( v^2 V \Psi_{\ell=0} \big) = rv^2 (-\Box \psi)_{\ell=0},
$$
thus we get by applying the fundamental theorem of calculus that
$$
\aligned
&\big(v^2 V \Psi_{\ell=0}\big) (u, v) - \big(v^2 V \Psi_{\ell=0}\big) (u_{\tau_0}(v), v)\\
=
&{1\over 2}\int_{u_{\Sigma_{\tau_0}}}^u  rv^2 (-\Box \psi)_{\ell=0}  \, du
\\
=
&2\int_{\tau_0}^\infty  r^3 (-\Box \psi)_{\ell=0}  \, du\Big|_{\II_{\geq \tau_0}}
-
2\int_{u}^\infty  r^3 (-\Box \psi)_{\ell=0}  \, du\Big|_{\II_{\geq \tau_0}}
\\
-
&2\int_{\tau_0}^{u_{\tau_0}(v)}  r^3 (-\Box \psi)_{\ell=0}  \, du\Big|_{\II_{\geq \tau_0}}
+
{1\over 4}\int_{v}^\infty V \int_{u_0}^u  rv^2 (-\Box \psi)_{\ell=0}  \, du dv.
\endaligned
$$

First, one finds from initial data assumption \eqref{assump:initialdataasymp} that
$$
\big|\big(v^2 V \Psi_{\ell=0}\big) (u_{\tau_0}(v), v) - c_{init} \big|
\lesssim_{\delta}
(D+c_{init}) v^{-\delta}.
$$

Next, we estimate
$$
\aligned
\Big|\int_{u}^\infty  r^3 (-\Box \psi)_{\ell=0}  \, du\Big|_{\II_{\geq \tau_0}} \Big|
\lesssim
\int_{u}^\infty  \big\| r^3 \Box \psi\big\|_{L^\infty(\mathbb{S}^2)} \, du\Big|_{\II_{\geq \tau_0}} 
\lesssim_{\delta}
\varepsilon^2 \int_{u}^\infty (u')^{-2} \, du'
\lesssim_{\delta}
\varepsilon^2 u^{-1}.
\endaligned
$$

We also note that
$$
|u_{\tau_0}(v)  - \tau_0| \lesssim r^{-\eta} \lesssim v^{-\eta(1-2\delta_1)}, 
$$
which further yields
$$
\Big| \int_{\tau_0}^{u_{\tau_0}(v)}  r^3 (-\Box \psi)_{\ell=0}  \, du\Big|_{\II_{\geq \tau_0}} \Big|
\lesssim
\varepsilon^2 v^{-\eta(1-2\delta_1)}.
$$

Then we bound the last term 
$$
\Big|\int_{v}^\infty V \int_{u_{\tau_0}(v)}^u  rv^2 (-\Box \psi)_{\ell=0}  \, du dv \Big|
=
\Big| \int_{v}^\infty \int_{u_{\tau_0}(v)}^u  V\big(rv^2 (-\Box \psi)_{\ell=0} \big) \, du dv \Big|.
$$
We recall from \eqref{eq:est-null2} that for $r\geq 1$ it holds
$$
\aligned
&\big| V\big(rv^2 (-\Box \psi) \big) \big|
\\
\lesssim
&V({v^2\over r^2}) \Big( \big| \mathbb{D}^{\leq 1} \Psi \partial_t \mathbb{D}^{\leq 1}\Psi  \big|
+
r^{-1} \big| \mathbb{D}^{\leq 1} \Psi \mathbb{D}^{\leq 2}\Psi  \big| \Big)
\\
+
&{v^2 \over r^2} \Big( \big| V\mathbb{D}^{\leq 1} \Psi \partial_t \mathbb{D}^{\leq 1}\Psi  \big|
+
\big| \mathbb{D}^{\leq 1} \Psi V \partial_t \mathbb{D}^{\leq 1}\Psi  \big|
+
r^{-1} \big| V \mathbb{D}^{\leq 1} \Psi \mathbb{D}^{\leq 2}\Psi  \big|
+
r^{-1} \big| \mathbb{D}^{\leq 1} \Psi V \mathbb{D}^{\leq 2}\Psi  \big| \Big)
\endaligned
$$
By the estimate in Lemma \ref{lem:V-Psi}, we get
$$
\big| V\big(rv^2 (-\Box \psi) \big) \big|
\lesssim_{\delta}
\varepsilon^2 v^{-1-2\delta_1} u^{-1-2\delta_1},
\qquad
\text{for  }
r\geq v^{1-2\delta_1}.
$$
Thus we arrive at
$$
\aligned
&\Big|\int_{v}^\infty V \int_{u_{\tau_0}(v)}^u  rv^2 (-\Box \psi)_{\ell=0}  \, du dv \Big|
\\
\lesssim
&\Big|\int_{v}^\infty  \int_{u_{\tau_0}(v)}^u  \big\| V\big(rv^2 \Box \psi\big)\big\|_{L^\infty(\mathbb{S}^2)}  \, du dv \Big|
\lesssim_{\delta}
\varepsilon^2 v^{-\delta_1}.
\endaligned
$$

The proof is completed by noting $v^{-1} \leq \tau^{-1}$.
\end{proof}

\begin{proposition}[Pointwise decay IV]\label{prop:decay-outer}
We have in the region $\{ r \geq v^{1-\delta_1} \}$ that
\begin{equation}
\Big| \psi_{\ell=0} -  { c_{total} \over v u} \Big|
\lesssim_{\delta}
(D+\varepsilon + c_{init} ) {1\over v u}  u^{-\delta}.
\end{equation}
\end{proposition}

\begin{proof}
We work in the $(u, v)$ coordinates. For any $(u, v) \in \{ r \geq v^{1-\delta_1} \}$, we have
$$
\aligned
&\Psi_{\ell=0}(u, v) - \Psi_{\ell=0}(u, v_{\gamma}(u))
=
{1\over 2} \int_{v_{\gamma}(u)}^v V\Psi_{\ell=0} \, ds
\\
=
&{1\over 2} \int_{v_{\gamma}(u)}^v c_{total} s^{-2} \, ds 
+
{1\over 2} \int_{v_{\gamma}(u)}^v  \Big( V\Psi_{\ell=0} - c_{total} s^{-2} \Big) \, ds
\\
=
&c_{total} {r\over v u}
-
{1\over 2} \int_{u}^{v_{\gamma}(u)}c_{total} s^{-2} \, ds 
+
{1\over 2} \int_{v_{\gamma}(u)}^v \Big( V\Psi_{\ell=0} - c_{total} s^{-2} \Big) \, ds, 
\endaligned
$$
in which $(u, v_{\gamma}(u))$ lies on the curve $\gamma=\{ r = v^{1-2\delta_1} \}$ and
$$
u = v_{\gamma(u)} - 2 (v_{\gamma}(u))^{1-2\delta_1}.
$$

First, we note 
$$
\aligned
\int_{v_{\gamma(u)}}^v  |V\Psi_{\ell=0} - c_{total} s^{-2}| \, ds
\lesssim_{\delta}
(D+\varepsilon + c_{init} ) \int_u^v s^{-2} u^{-\delta} \, ds
\lesssim_{\delta} 
(D+\varepsilon + c_{init} ) {r\over v u} u^{-\delta}.
\endaligned
$$

Second, we find, according to \eqref{eq:decay-01}, that
$$
\aligned
|\Psi_{\ell=0}(u, v_{\gamma}(u))|
\lesssim_{\delta} 
\varepsilon  {v_{\gamma}(u) - u \over v_{\gamma}(u) u} u^{\delta}
=
\varepsilon  {r\over u v} {u^{\delta} \over v_{\gamma}(u)^{2\delta_1}} \Big( {u\over r} + 2 \Big).
\endaligned
$$
We observe that
$$
r \geq v^{1-\delta_1} \geq u^{1-\delta_1},
$$
and 
$$
u \leq v_{\gamma}(u).
$$
Thus we have
$$
|\Psi_{\ell=0}(v_{\gamma}(u), u)|
\lesssim_{\delta} 
\varepsilon {r\over u v} u^{-\delta_1+\delta}.
$$

Similarly, we can deduce
$$
\int_{u}^{v_{\gamma}(u)} c_{total} s^{-2} \, ds 
\lesssim
c_{total} {v_{\gamma}(u) - u \over v_{\gamma}(u) u} 
\lesssim
c_{total} {r\over u v} u^{-\delta_1}.
$$

Finally, in view of the relation $\delta_1=20\delta$, we arrive at the desired result.
\end{proof}

\begin{proposition}[Pointwise decay V]\label{prop:decay-inner}
We have in the region $\{ 0\leq r \leq t\}$ that
\begin{equation}\label{eq:decay-inner}
\Big| \psi_{\ell=0} -  { c_{total} \over v u} \Big|
\lesssim_{\delta} 
(D+\varepsilon + c_{init} ) {1\over v u}  u^{-\delta}.
\end{equation}
\end{proposition}

\begin{proof}
According to Proposition \ref{prop:decay-outer}, we only need to derive \eqref{eq:decay-inner} in the inner region $\{ 0 \leq r \leq v^{1-\delta_1} \}$.
We emphasize that for a point in this region $\{ 0 \leq r \leq v^{1-\delta_1} \}$ it satisfies
$$
\aligned
v \lesssim \tau\simeq u \leq t \leq v.
\endaligned
$$

We first note (see \eqref{eq:partial-rho})
$$
|\overline{\partial}_r \psi_{\ell=0} (\rho, \tau)|
\lesssim_{\delta} 
\varepsilon  v^{-1} \tau^{-2+3\delta}
\lesssim_{\delta}  \varepsilon    \tau^{-3+3\delta}.
$$
To proceed, we have
$$
\aligned
\psi_{\ell=0}(\rho', \tau) 
=
\psi_{\ell=0}(\rho=v^{1-\delta_1}, \tau) + \int_{v^{1-\delta_1}}^{\rho'} \overline{\partial}_r \psi_{\ell=0} \, d\rho.
\endaligned
$$
We find
\begin{equation*}
\aligned
\Big| \int_{v^{1-\delta_1}}^{\rho'} \overline{\partial}_r \psi_{\ell=0} \, d\rho\Big|
\lesssim_{\delta} 
&\varepsilon  \int^{v^{1-\delta_1}}_0   \tau^{-3+3\delta}\, d\rho
\lesssim_{\delta} 
\varepsilon v^{1-\delta_1} \tau^{-3+3\delta}
\lesssim_{\delta} 
\varepsilon  {1\over uv} u^{-\delta_1 + 3\delta}.
\endaligned
\end{equation*}

According to Proposition \ref{prop:decay-outer}, we also have
$$
\aligned
\Big| \psi_{\ell=0}(\rho=v^{1-\delta_1}, \tau) - {c_{total} \over uv} \Big|
\lesssim_{\delta} 
(D+\varepsilon + c_{init} ) {1\over v u}  u^{-\delta}.
\endaligned
$$

Gathering these, we arrive at the desired estimates in \eqref{eq:decay-inner}.
\end{proof}

%%%%%%%%%%%%%%%%%%%%
\subsection{Pointwise decay for $\ell \geq 1$ model}
%%%%%%%%%%%%%%%%%%%%

We denote
$$
\Psi_{\ell\geq 1} = \Psi - \Psi_{\ell=0}, 
$$
and we have
\begin{equation}\label{eq:geq1}
UV \Psi_{\ell\geq 1} - r^{-2} {\Delta\mkern-11.5mu/}  \Psi_{\ell\geq 1} = r (- \Box \psi )_{\ell \geq 1},
\end{equation}
i.e.,
\begin{equation}
U (r^2 V \Psi_{\ell\geq 1}) + 2 r V \Psi_{\ell\geq 1} - {\Delta\mkern-11.5mu/}  \Psi_{\ell\geq 1} = r^3 (- \Box \psi )_{\ell \geq 1}.
\end{equation}
We then act $V$ to both sides of the above equation to get
\begin{equation}\label{eq:Psi-1}
U V \Psi^{(1)}_{\ell\geq 1} + 2 r^{-1} V\Psi^{(1)}_{\ell\geq 1} + r^{-2} (- {\Delta\mkern-11.5mu/} - 2) \Psi^{(1)}_{\ell\geq 1} = V \big( r^3 (-\Box \psi)_{\ell\geq 1} \big),
\end{equation}
in which $\Psi^{(1)}_{\ell\geq 1} := r^2 V\Psi_{\ell\geq 1}$.

We emphasize that the spectrum of $-{\Delta\mkern-11.5mu/}  \Psi_{\ell\geq 1}^{(1)}$ is greater than or equal to $2$, which means the following holds
\begin{equation}\label{eq:spectrum-gap}
\int_{\mathbb{S}^2} \Big(\big|{\nabla\mkern-11.5mu/}  \Psi_{\ell\geq 1}^{(1)}|^2 - 2 |\Psi_{\ell\geq 1}^{(1)} |^2 \Big) \, d^2\mu 
\geq 0.
\end{equation}

We first build the weak energy decay (and global $r^p$ estimate) for the component $\Psi_{\ell\geq 1}$.

\begin{proposition}
Under the same assumptions as in Theorem \ref{thm:leadingorder:Qwave}, we have, for $k\leq 3$ and $\tau_2> \tau_1 \geq \tau_0$, that
\begin{subequations}
\begin{align}
\label{eq:higher-weak-decay}
&\Energy{k}{p}{\tau_2}{\Psi_{\ell\geq 1}} + \Wnorm{\Psi_{\ell\geq 1}}{k+1}{p-3}{\MMtwoinfty}^2
+\Wnorm{U\Psi_{\ell\geq 1}}{k}{-1-\delta/2}{\MMtwoinfty}^2
\notag\\
\lesssim_{\delta}  & \langle \tau_2 - \tau_1\rangle^{-2+\delta+p}\Energy{k}{2-\delta}{\tau_1}{\Psi_{\ell\geq 1}} + \varepsilon^4 \la \tau_2 -\tau_1\ra^{-4+11\delta},
\qquad
p \in [\delta, 2-\delta], 
\\
\label{eq:higher-weak-decay-tau}
\vspace{20mm}
&\Energy{k}{p}{\tau_2}{\partial_\tau\Psi_{\ell\geq 1}} + \Wnorm{\partial_\tau\Psi_{\ell\geq 1}}{k+1}{p-3}{\MMtwoinfty}^2
+\Wnorm{U\partial_\tau\Psi_{\ell\geq 1}}{k}{-1-\delta/2}{\MMtwoinfty}^2
\notag\\
\lesssim_{\delta}  & \langle \tau_2 - \tau_1\rangle^{-2+\delta+p}\Energy{k}{2-\delta}{\tau_1}{\partial_\tau \Psi_{\ell\geq 1}} + \varepsilon^4 \la \tau_2 - \tau_1\ra^{-6+13\delta},
\qquad
p \in [\delta, 2-\delta]. 
\end{align}
\end{subequations}

\end{proposition}

\begin{proof}
The proof is similar to the one for Proposition \ref{prop:radial-E-decay1}. It relies on the global $r^p$ estimate for wave equations established in Theorem \ref{thm:globalrp:InhoWave:M} for $\Psi=\Psi_{\ell\geq 1}$, as well as the estimates on the error terms which are controlled by the pointwise decay of the solution as in Proposition \ref{prop:radial-E-decay01}.
\end{proof}

When considering $\Psi_{\ell\geq 1}^{(1)}$, we define a new global $r^p$ energy functional which is
\begin{equation}
\widetilde{\mathbf{E}}^{(k+1)}_{p, \tau} [\Psi_{\ell\geq 1}]
=
\Energy{k}{p}{\tau}{\Psi_{\ell\geq 1}^{(1)}}
+ \SEnergy{k+1}{\tau}{\Psi_{\ell\geq 1}}  .
\end{equation}

In a similar way, we apply Lemma \ref{lem:rpForInhoWave:M} (specifically we need $2')$ with $b_{0,0}=-2$) and Proposition \ref{prop:HighEnerMora:psi:M} to show the global $r^p$ estimate for the component $\Psi_{\ell\geq 1}^{(1)}$, which is stated now.

\begin{proposition}
Under the same assumptions as in Theorem \ref{thm:leadingorder:Qwave}, we have, for $k\leq 3$ and $\tau_2> \tau_1 \geq \tau_0$, that  
\begin{align}
\label{eq:higher-globalrp}
&\widetilde{\mathbf{E}}^{(k+1)}_{p, \tau_2} [\Psi_{\ell\geq 1}]    
+ \Wnorm{\Psi_{\ell\geq 1}^{(1)}}{k+1}{p-3}{\MMtwoinfty}^2
+\Wnorm{U\Psi_{\ell\geq 1}^{(1)}}{k}{-1-\delta/2}{\MMtwoinfty}^2  
 \notag\\
\lesssim_{\delta}  & \langle \tau_2 - \tau_1\rangle^{-1-6\delta+p}\widetilde{\mathbf{E}}^{(k+1)}_{1+6\delta, \tau_1} [\Psi_{\ell\geq 1}]  + \varepsilon^4 \la \tau_2-\tau_1\ra^{-2+5\delta},
\qquad
p \in [\delta, 1+6\delta], 
\\
\label{eq:higher-globalrp-tau}
&\widetilde{\mathbf{E}}^{(k+1)}_{p, \tau_2} [\partial_\tau \Psi_{\ell\geq 1}]  
+ \Wnorm{\partial_\tau\Psi_{\ell\geq 1}^{(1)}}{k+1}{p-3}{\MMtwoinfty}^2
+\Wnorm{U\Psi_{\ell\geq 1}^{(1)}}{k}{-1-\delta/2}{\MMtwoinfty}^2   
 \notag\\
\lesssim_{\delta}  & \langle \tau_2 - \tau_1\rangle^{-1-6\delta+p}\widetilde{\mathbf{E}}^{(k+1)}_{1+6\delta, \tau_1} [\partial_\tau \Psi_{\ell\geq 1}]  + \varepsilon^4 \la \tau_2 -\tau_1\ra^{-4+5\delta},
\qquad
p \in [\delta, 1+6\delta]. 
\end{align}
\end{proposition}

\begin{proof}
The proof is very similar to the one for Propositions \ref{prop:radial-E-decay01} and \ref{prop:radial-E-decay1}. Here we only list the pointwise decay of the source term in \eqref{eq:Psi-1}. Based on the weak pointwise decay estimates in Proposition \ref{prop:weakpointwise:M} and the estimates for null forms in \eqref{eq:est-null}, we can derive
\begin{equation}
\big|V\big(r^3 \Box \psi \big) \big|_{k}
\lesssim_{\delta} 
\varepsilon^2 \la r \ra^{-1-6\delta} \tau^{-2+\delta},
\end{equation}
which is enough for us to get the desired results.
\end{proof}

Next result builds a relation between $\Energy{k}{p}{\tau_2}{\Psi_{\ell\geq 1}}$ and $\widetilde{\mathbf{E}}^{(k+1)}_{p, \tau_2} [\Psi_{\ell\geq 1}] $, which reads as $\Energy{k}{p=2-\delta}{\tau_2}{\Psi_{\ell\geq 1}} \leq C \widetilde{\mathbf{E}}^{(k+1)}_{p=\delta, \tau_2} [\Psi_{\ell\geq 1}] $. Thus extra $\tau$ decay for $\Energy{k}{p}{\tau_2}{\Psi_{\ell\geq 1}}$ can be obtained.

\begin{lemma}\label{lem:higher-relation}
For $k\leq 3$ and $\tau\geq\tau_0$, we have
\begin{equation}\label{eq:energy503}
\aligned
\Energy{k}{2-\delta}{\tau}{\Psi_{\ell\geq 1}}
\leq
&C  \widetilde{\mathbf{E}}^{(k+1)}_{\delta, \tau} [\Psi_{\ell\geq 1}],  
\\
\Energy{k}{2-\delta}{\tau}{\partial_\tau \Psi_{\ell\geq 1}}
\leq
&C \widetilde{\mathbf{E}}^{(k+1)}_{\delta, \tau} [\partial_\tau \Psi_{\ell\geq 1}],   
\endaligned
\end{equation}
in which $\Psi_{\ell\geq 1}^{(1)} = r^2V \Psi_{\ell\geq 1}$.
\end{lemma}
\begin{proof}
First, we note
$$
\Energy{k}{2-\delta}{\tau}{\Psi_{\ell\geq 1}} - \Wnorm{rV\Psi_{\ell\geq 1}}{k}{-\delta}{\Sigma_{\tau}}^2
\lesssim
\Energy{k}{\delta}{\tau}{\Psi_{\ell\geq 1}^{(1)}}.
$$

Next, we find
$$
\aligned
\sum_{|\alpha|\leq k}\int_{\Sigma_\tau} \langle r\rangle^{-\delta} |rV \mathbb{D}^\alpha \Psi_{\ell\geq 1}|^2 \, d^3 \mu 
\lesssim
&\sum_{|\alpha|\leq k}\int_{\Sigma_\tau}  r^{-2-\delta} |r^2V \mathbb{D}^\alpha \Psi_{\ell\geq 1}|^2 \, d^3 \mu
\\
\lesssim
&\sum_{|\alpha|\leq k}\int_{\Sigma_\tau}  r^{-2-\delta} \big| \mathbb{D}^\alpha \Psi_{\ell\geq 1}^{(1)} \big|^2  \, d^3 \mu 
\\
\lesssim
&\Energy{k}{\delta}{\tau}{\Psi_{\ell\geq 1}^{(1)}}
\endaligned
$$
and in the last step we used the following commutation facts
$$
[r^2V, rV] = -r^2V,
\qquad
[r^2V, \partial_t]= [r^2V, \Omega]=0
$$
and in the last but two step we applied Hardy inequality.

Thus we get
$$
\aligned
&\Energy{k}{2-\delta}{\tau}{\Psi_{\ell\geq 1}}
\\
\lesssim
&\sum_{|\alpha|\leq k}\int_{\Sigma_\tau} \langle r\rangle^{-\delta} |rV \mathbb{D}^\alpha \Psi_{\ell\geq 1}|^2 \, d^3 \mu  + \SEnergy{k}{\tau}{\Psi_{\ell\geq 1}}  
+ \big(\Energy{k}{2-\delta}{\tau}{\Psi_{\ell\geq 1}} - \Wnorm{rV\Psi_{\ell\geq 1}}{k}{-\delta}{\Sigma_{\tau}}^2 \big)
\\
\lesssim
&\Energy{k}{\delta}{\tau}{\Psi_{\ell\geq 1}^{(1)}},
\endaligned
$$
which completes the proof.
\end{proof}

As a consequence, we have the following improved energy decay results.
\begin{proposition}
For $\tau \geq \tau_0$, we have for $k\leq 2$ and $p \in [\delta, 2-\delta]$ that
\begin{subequations}
\begin{align}
\label{eq:higher-E-decay0}
\Energy{k}{p}{\tau}{\Psi_{\ell\geq 1}}  + \Wnorm{\Psi_{\ell\geq 1}}{k+1}{p-3}{\MM_{\tau,\infty}}^2
+\Wnorm{U\Psi_{\ell\geq 1}}{k}{-1-\delta/2}{\MM_{\tau,\infty}}^2
\lesssim_{\delta} 
&  \varepsilon^2  \tau^{-3-4\delta+p},
\\
\label{eq:higher-E-decay1}
\Energy{k}{p}{\tau}{\partial_\tau\Psi_{\ell\geq 1}}  + \Wnorm{\partial_\tau \Psi_{\ell\geq 1}}{k+1}{p-3}{\MM_{\tau,\infty}}^2
+\Wnorm{U\partial_\tau \Psi_{\ell\geq 1}}{k}{-1-\delta/2}{\MM_{\tau,\infty}}^2
\lesssim_{\delta} 
&  \varepsilon^2  \tau^{-5 - 2 \delta+p}.
\end{align}
\end{subequations}
\end{proposition}

\begin{proof}
First, for $k\leq 2$ and $\tau\geq \tau_1\geq \tau_1' \geq \tau_0$, we can show for any $\delta\leq p\leq 1+6\delta$,
\begin{equation}\label{eq:higher-E-decay2}
\aligned
\Energy{k}{p}{\tau}{\Psi_{\ell\geq 1}}  
\lesssim_{\delta}  &
\langle \tau - \tau_1 \rangle^{-2+\delta+p} \Energy{k}{2-\delta}{\tau_1}{\Psi_{\ell\geq 1}}  +  \varepsilon^4 (\tau_1)^{-4+3\delta}
\\
\lesssim_{\delta}  &
\langle \tau - \tau_1 \rangle^{-2+\delta+p} \langle \tau_1-\tau_1'  \rangle^{-1-6\delta + \delta}  \widetilde{\mathbf{E}}^{(k+1)}_{1+6\delta, \tau_1' } [\Psi_{\ell\geq 1}]  +  \varepsilon^4 (\tau_1')^{-4+3\delta }.
\endaligned
\end{equation}
The proof then follows from taking $ \tau_1=\tau/2, \tau_1' =\tau/4$.

In the same way, we can derive the improved energy decay for $\partial_\tau \Psi_{\ell\geq 1}$.
\end{proof}

Now, we are ready to build the pointwise decay for $\psi_{\ell\geq  1}$.
\begin{proposition}\label{prop:higher-point-decay}
For $\tau \geq \tau_0$, it holds that
\begin{equation}\label{eq:higher-point-decay}
|\psi_{\ell\geq 1}| \lesssim_{\delta}   \varepsilon  v^{-1} \tau^{-1-\delta}.
\end{equation}
\end{proposition}
\begin{proof}
The proof is similar to the one for Proposition \ref{prop:weakpointwise:M}, which is based on the energy decay estimates and Sobolev inequalities in Lemma \ref{lem:Sobolev}. 

On the one hand, we apply \eqref{eq:Sobolev:2} with $q=1/2$ and \eqref{eq:higher-E-decay0} with $p\in \{ 1/2, 3/2 \}$ to get
$$
|\Psi_{\ell\geq 1}| 
\leq C \big( \Energy{k=2}{p=1}{\tau}{\Psi_{\ell\geq 1}} \big)^{1/2}
\lesssim_{\delta}  \varepsilon  \langle \tau\rangle^{-1-\delta}.
$$
On the other hand, combining \eqref{eq:Sobolev:3}, \eqref{eq:higher-E-decay0} with $p=\delta$, and \eqref{eq:higher-E-decay1} with $p=\delta$ yields
$$
\aligned
|\psi_{\ell \geq 1}|
\leq &C \big(\Wnorm{{\Psi_{\ell\geq 1}}}{3}{-3}{\MM_{\tau, \infty}} \Wnorm{\del_{\tau}\Psi_{\ell\geq 1}}{3}{-3}{\MM_{\tau, \infty}}\big)^{1/2}
\\
\lesssim & \varepsilon  \langle \tau\rangle^{-2-\delta}.
\endaligned
$$

Finally, the simple fact $v \leq C(\tau + 2r)$ as well as the smallness of $\delta$ leads us to the desired result.
\end{proof}

%%%%%%%%%%%%%%%%%%%%
\subsection{Leading-order term}
%%%%%%%%%%%%%%%%%%%%

We are now ready to extract the leading-order term for the asymptotics of the solution $\psi$.
\begin{proof}[Proof of Theorem \ref{thm:leadingorder:Qwave}]
By the estimates for $\psi_{\ell=0}$ in Proposition \ref{prop:decay-inner} and the estimates for $\psi_{\ell\geq 1}$ in Proposition \ref{prop:higher-point-decay} , we get
$$
|\psi - c_{total} u^{-1} v^{-1}|
\leq
|\psi_{\ell=0} - c_{total} u^{-1} v^{-1}| + |\psi_{\ell\geq 1}|
\lesssim_{\delta} 
(D+c_{init}+\varepsilon)v^{-1}u^{-1-\delta}.
$$
This completes the proof.
\end{proof}

%%%%%%%%%%%%%%%%%%%%
%%%%%%%%%%%%%%%%%%%%
\section{Genericity of the optimal decay rate}
\label{sect:genericity}
%%%%%%%%%%%%%%%%%%%%
%%%%%%%%%%%%%%%%%%%%

We discuss if the optimal decay rate $v^{-1}u^{-1}$ shown in the previous section is generic or not. In Section \ref{subsect:pfMainThm}, we show that under a suitable topology of the initial data the decay rate $v^{-1}u^{-1}$ is generically sharp, that is, the constant $c_{total}\neq 0$ in Theorem \ref{thm:leadingorder:Qwave} generically holds true. This provides a proof for our main Theorem \ref{thm:main}. In Section \ref{subsect:pfthm:compsupp}, we further restrict the initial data to be smooth compactly supported, discuss the vanishing property of the constant $c_{total}$ based on the classification \eqref{eq:classify:nullform}  of the quasilinear null forms, and thus prove Theorem \ref{subsect:pfthm:compsupp}.

%%%%%%%%%%%%%%%%%%%%
\subsection{Proof of Theorem \ref{thm:main}}
\label{subsect:pfMainThm}
%%%%%%%%%%%%%%%%%%%%

In this section we show Theorem \ref{thm:main}. To be more specific, we prove that 
there exists a generic subset of the initial data set for which the global solution to the quasilinear wave equation \eqref{Qwave:M} with null condition admits the precise pointwise asymptotic as described in Theorem \ref{thm:leadingorder:Qwave}, such that the coefficient $c_{total}$ of the decay rate $v^{-1}u^{-1}$ in this leading-order term is generically non-zero. 

To begin with, we define an  ``admissible'' initial data set which can be viewed as the most general initial data set such that the global solution to \eqref{Qwave:M} enjoys the precise pointwise asymptotic with a decay rate $u^{-1}v^{-1}$ as in Theorem \ref{thm:leadingorder:Qwave}. 

\begin{definition}
\label{def:admissibledata}
Let $\psi$ be the local solution to \eqref{Qwave:M} with initial data $(\psi_0,\psi_1)$. 
Let $0<\vep_0\ll 1$ be suitably small such that the statement in Theorem \ref{thm:leadingorder:Qwave} holds true.  Such a pair of initial data is called  admissible initial data if it satisfies both the assumption \eqref{assump:initialdataasymp} for $c_{init}\in \mathbb{R}$ and 
\begin{equation}
\Etotal{k}{\psi}= \Energy{k}{p=2-\delta}{\tau_0}{\Psi} 
+\Energy{k-10}{p=3-\delta}{\tau_0}{\Psi_{\ell=0}} 
+\Energy{k-10}{p=1+6\delta}{\tau_0}{r^2V\Psi_{\ell\geq 1}}
< \vep_0^2
\end{equation}
for a fixed $k\geq 20$.
The set of all the admissible initial data is called as the admissible initial data set and denoted as $\set{k}$.
\end{definition}

We next define a distance function for any two given pairs of admissible initial data.

\begin{definition}
\label{def:distancefunction}
Let $(\psi_0,\psi_1)$ and $(\varphi_0, \varphi_1)$ be two pairs of admissible initial data in $\set{k}$, and let $\psi$ and $\varphi$ be the global solutions for these pairs of initial data respectively. Let $\Psi=r\psi$ and $\Phi=r\varphi$ be their radiation fields. Let $k\in \mathbb{N}$ and define a distance function
\begin{align}
\dist{k}{\psi,\varphi}:=&(\Etotal{k}{\psi-\varphi})^{\frac{1}{2}}+|c_{init}[\psi]- c_{init}[\varphi]|.
\end{align}
\end{definition}

The above distance function manifestly introduces a topology in the admissible initial data set $\set{k}$. 

We now show that the decay rate $v^{-1}u^{-1}$ is generically sharp in the admissible initial data under the topology induced by the above distance function. Specifically, we will establish the following theorem, which in turn yields our main Theorem \ref{thm:main}.

\begin{theorem}
\label{thm:genericity}
Let $k\geq 20$. Let $\setsub{k}$ be the subset of $\set{k}$ such that for any pair $(\psi_0,\psi_1)$ of initial data in this subset, its global solution $\psi$ satisfies
\begin{align}
\label{eq:genericity:ctotal}
|\psi - c_{total}[\psi] v^{-1} u^{-1}|\lesssim v^{-1} u^{-1-\delta_0}
\end{align} 
globally in the future of $\Sigma_{\tau_0}$ for some $\delta_0>0$ and  $c_{total}[\psi]\neq 0$. 
Then $\setsub{k}$ is open and dense in $\set{k}$ under the topology induced by the distance function introduced in Definition \ref{def:distancefunction}.
\end{theorem}

\begin{proof}
We first show $\setsub{k}$ is an open subset in $\set{k}$ and then prove that it is dense in $\set{k}$.

\underline{\bf $\setsub{k}$ is open in $\set{k}$.} 

Let $\psi$ be the global solution arising from a pair of admissible initial data $(\psi_0,\psi_1)$ and satisfy \eqref{eq:genericity:ctotal} for $c_{total}[\psi]\neq 0$. Our aim is to show that there exists a $\vep_{\star}>0$ such that for any pair $(\psi_0',\psi_1')$ of admissible initial data satisfying $\dist{k}{\psi,\psi'}< \vep_{\star}$, where $\psi'$ is the global solution with the initial data $(\psi_0',\psi_1')$, the estimate \eqref{eq:genericity:ctotal} holds for $\psi'$ with $c_{total}[\psi']\neq 0$. 

This is essentially to show that $c_{total}[\psi]$ is continuously dependent on its initial data in the topology induced by the distance function, and all the elements of its proof are in fact already included in the analysis in Section  \ref{sect:sharpdecay}. Below, we  provide a proof along this line. 

Without loss of generality, we may take $\vep_{\star}<\vep_0\ll 1$. Consider the difference function $\tilde{\psi}:=\psi-\psi'$. It satisfies a wave equation
\begin{align}
\label{waveeq:difference:genericity}
\Box \tilde{\psi}=&\widetilde{P}(\del \psi, \del^2\psi, \del \tilde{\psi}, \del^2\tilde{\psi}),
\end{align}
where  $\widetilde{P}(\del \psi, \del^2\psi, \del \tilde{\psi}, \del^2\tilde{\psi}):=P (\del \psi, \del^2\psi) - P(\del\psi', \del^2\psi')$. We denote the right-hand side by $\widetilde{P}$ for convenience. The decay estimates proven in Proposition \ref{prop:weakpointwise:M} apply to both $\psi$ and $\psi'$ and  thus yield 
that for $k\in \left\{0,1,2, 3\right\}$, 
\begin{align}\label{eq:NL-pw:difference}
|\widetilde{P}|_{(k)}
\lesssim \varepsilon_0 \vep_{\star} \tau^{-3+3\delta} \langle r\rangle^{-3}, \quad \forall k\in \left\{0,1,2, 3\right\}.
\end{align}
This implies
\begin{align}
\label{eq:cscri:difference:genericity}
|c_{\II }[\tilde{\psi}]|=2\bigg|\int_{\II_{\geq \tau_0}}r^3 \widetilde{P} d^2\mu d\tau\bigg| \lesssim \vep_0\vep_{\star}.
\end{align}
Further, by definition, we have $c_{init}[\tilde{\psi}]=c_{init}[\psi]-c_{init}[\psi']$,  hence it holds that $|c_{init}[\tilde{\psi}]|\leq \dist{k}{\psi,\psi'}<\vep_{\star}$. 

We next run the argument in Section \ref{sect:sharpdecay} and obtain that
\begin{align}
|\tilde{\psi} - c_{total}[\tilde{\psi}]v^{-1}u^{-1}|\lesssim v^{-1} u^{-1-\delta_0}
\end{align}
for some $\delta_0>0$ and $c_{total}[\tilde{\psi}]= c_{init}[\tilde{\psi}]+c_{\II}[\tilde{\psi}]$. In view of the above estimates for $c_{init}[\tilde{\psi}]$ and $c_{\II}[\tilde{\psi}]$, one has $|c_{total}[\tilde{\psi}]|\lesssim \vep_0\vep_{\star} + \vep_{\star}$. Given that 
$c_{total}[\psi']=c_{total}[\psi]-c_{total}[\tilde{\psi}]$ and $c_{total}[\psi]\neq 0$, it is manifest that one can take $0<\vep_{\star}\ll |c_{total}[\psi]|$ such that $c_{total}[\psi']\neq 0$.

\underline{\bf $\setsub{k}$ is dense in $\set{k}$.} 

The main observation is that if we perturb both the initial data and the constant $c_{init}$ up to $\vep$-order, the constant $c_{\II}$, an integral of the nonlinear part along $\II$, is only perturbed up to $\vep_0\vep$-order, which is of lower order compared to $\vep$-order, because of the quadratic nature of the nonlinear part.  We now discuss the details of the proof.

Let $\psi$ be the global solution arising from a pair of admissible initial data $(\psi_0,\psi_1)$ and satisfy \eqref{eq:genericity:ctotal} with $c_{total}[\psi]= 0$. We aim to show that for any $\vep_{\star}>0$, there exists a pair $(\psi_0',\psi_1')$ of admissible initial data satisfying $\dist{k}{\psi,\psi'}< \vep_{\star}$, where $\psi'$ is the global solution arising from the initial data $(\psi_0',\psi_1')$, such that the estimate \eqref{eq:genericity:ctotal} holds for $\psi'$ with $c_{total}[\psi']\neq 0$. 

The proof proceeds in the same manner as the above one in proving that $\setsub{k}$ is open in $\set{k}$. Let $\vep_{\star}>0$ be arbitrary, and without loss of generality, we can take $\vep_{\star}\ll 1$. Let $\set{k}_{\vep_{\star}}[\psi]$ be the set of all the pairs $(\psi_0',\psi_1')$ of admissible initial data satisfying $\dist{k}{\psi,\psi'}< \vep_{\star}$ with $\psi'$ being the global solution for  the initial data $(\psi_0',\psi_1')$. We define the difference function $\tilde{\psi}:=\psi-\psi'$, and it solves the wave equation \eqref{waveeq:difference:genericity} with the right-hand side denoted by $\widetilde{\mathcal{N}}$ again. The estimates \eqref{eq:NL-pw:difference} and \eqref{eq:cscri:difference:genericity} hold true as well. We can take a pair $(\psi_0',\psi_1')\in \set{k}_{\vep_{\star}}[\psi]$ such that $|c_{init}[\tilde{\psi}]|=\frac{1}{2}\vep_{\star}$, then it holds $|c_{total}[\tilde{\psi}]|=|c_{init}[\tilde{\psi}]+c_{\II}[\tilde{\psi}]|\geq \frac{1}{4}\vep_{\star}$ since $\vep_0\ll 1$. Therefore, $|c_{total}[\psi']|=|c_{total}[\psi]-c_{total}[\tilde{\psi}]|=|c_{total}[\tilde{\psi}]|\geq \frac{1}{4}\vep_{\star}>0$, which completes the proof.
\end{proof}

%%%%%%%%%%%%%%%%%%%%
\subsection{Proof of Theorem \ref{thm:main-compact}}
\label{subsect:pfthm:compsupp}
%%%%%%%%%%%%%%%%%%%%

This last section is devoted to proving Theorem \ref{thm:main-compact},  in which the sharp decay rates for solutions to the wave equation \eqref{Qwave:M} with null condition \eqref{nullcond:M} arising from smooth, compactly supported initial data are discussed. In this case, Theorem \ref{thm:leadingorder:Qwave} applies as well but with $c_{init}=0$, hence 
\begin{align}
c_{total}=c_{\II}=2\int_{\II_{\geq \tau_0}} r^3 (-\Box \psi) \, d^2 \mu d\tau
= -2\int_{\II_{\geq \tau_0}} r^3 P^{\alpha\beta\gamma}\del_{\gamma}\psi \del_{\alpha}\del_{\beta}\psi \, d^2 \mu d\tau.
\end{align} 

Recall from \eqref{eq:classify:nullform} the classification of the quasilinear null forms. Consider first the case that the quasilinear terms contain only forms belonging to the set $\mathbf{P}_1$. Then, the quasilinear wave equation is equivalent to the homogeneous linear wave equation $\Box \psi=0$ and the constant $c_{total}$ is manifestly vanishing. In fact,  it can be shown that $|\psi|\lesssim v^{-1} u^{-q}$ for any $q\in \mathbb{N}$, and due to the strong Huygens' principle, for any finite $r$, $\psi(\tau,r,\omega)=0$ for $\tau$ suitably large. 

Consider next the cases that the quasilinear terms are of the forms in the set $\mathbf{P}_2$ or in both sets $\{\mathbf{P}_i\}_{i=1,2}$. It is manifest that there is no difference in showing the vanishing property of $c_{total}$ in these two cases, and we consider only the first case which treats the quasilinear wave equation $\Box \psi = \partial_{\alpha} (\partial_\gamma \psi \partial^\gamma \psi)$. The form $\partial_\gamma \psi \partial^\gamma \psi$ is a typical semilinear null form and can be expressed as $O(r^{-3}) \DDb^{\leq 1}\Psi \DDb^{\leq 1}\Psi$. If $\alpha=0$, then since $\del_t=\del_{\tau}$, we have
$$
c_{total}=2\int_{\II_{\geq \tau_0}}\del_{\tau} (r^3 \partial_\gamma \psi \partial^\gamma \psi) d\tau=0.
$$
This thus yields that the decay rate of $\psi$ is faster than $v^{-1}u^{-1}$. If $\alpha=1,2,3$, we can expand $\del_{\alpha}=\frac{x^{\alpha}}{r}\del_r + O(r^{-1}) \nablas= \frac{x^{\alpha}}{r}V- \frac{x^{\alpha}}{r}\del_t + O(r^{-1}) \nablas$, which yields this quasilinear null form equals $O(r^{-3})\del_{\tau} (\DDb^{\leq 1}\Psi \DDb^{\leq 1}\Psi)+ O(r^{-4}) \DDb^{\leq 1}\Psi \DDb^{\leq 2}\Psi$, thus $c_{total}$ vanishes identically in this case as well. In summary, if the quasilinear terms are of the forms in the set $\mathbf{P}_2$ or in both sets $\{\mathbf{P}_i\}_{i=1,2}$, then $c_{total}=0$ and $\psi$ has faster decay than $v^{-1}u^{-1}$. In fact, one shows that $|\psi|\lesssim v^{-1}u^{-2}$, but we do not elaborate on its proof here\footnote{One simple proof is to redefine a function $\hat{\psi}=\psi-\frac{1}{2}\del_{\alpha}(\psi^2)$ which satisfies $\Box \hat{\psi}=-\del_{\alpha}(\psi \del_\alpha (\partial_\gamma \psi \partial^\gamma \psi) )$, and the approach used in this paper can be readily applied to show $|\psi|\lesssim v^{-1}u^{-2}$.}.

\bibliographystyle{amsplain}

\providecommand{\href}[2]{#2}

\end{document}